\documentclass[11pt]{article}

\usepackage{ucs}
\usepackage[utf8x]{inputenc}
\usepackage{graphicx}
\usepackage{amsfonts}
\usepackage{dsfont}
\usepackage{amssymb}
\usepackage{amsmath}
\usepackage{amsthm}
\usepackage{enumerate}
\usepackage{stmaryrd}
\usepackage{fullpage}
\usepackage{ifthen}
\usepackage{subfigure}
\usepackage[subfigure]{tocloft}
\usepackage{epic}
\usepackage{authblk}
\usepackage{textcomp}
\usepackage[small]{caption}
\usepackage{setspace}
\usepackage{tocloft}

\usepackage[hypertexnames=false,colorlinks=true,linkcolor=blue,citecolor=blue]{hyperref}
\usepackage[numbers,comma,square,sort&compress]{natbib}
\usepackage[top=1in, bottom=1in, left=1in, right=1in]{geometry}
\usepackage{palatino}
\usepackage{color}
\usepackage{titlesec}

\graphicspath{{eps/}{pdf/}}

\newcommand{\be}{\begin{equation}}
\newcommand{\ee}{\end{equation}}
\newcommand{\bqs}{\begin{equation*}}
\newcommand{\eqs}{\end{equation*}}

\newcommand{\R}{\mathbb{R}}
\newcommand{\W}{\mathcal{W}}

\numberwithin{equation}{section}

  {\begin{trivlist}\item[]{\bf Hypothesis #1 }\em}{\end{trivlist}}

\newcommand\fH[1]{\sbox0{#1}\dimen0=\ht0 \advance\dimen0 -1ex
  \sbox2{\'{}}\sbox2{\raise\dimen0\box2}%
  {\ooalign{\hidewidth\kern.1em\copy2\kern-.5\wd2\box2\hidewidth\cr\box0\crcr}}}

\theoremstyle{plain}
\newtheorem{theorem}{Theorem}[section]

\newtheorem{lem}[theorem]{Lemma}

\newtheorem{rmk}[theorem]{Remark}

\newtheorem{hyp}{Hypothesis}

\title{\vspace{-.4in}{\fontsize{18}{18}\selectfont \textbf{Saddle Transport and Chaos in the Double Pendulum}}\vspace{-.15in}}

\author[1]{Kadierdan Kaheman\footnote{Kaheman and Bramburger contributed equally to this manuscript.}\footnote{Corresponding authors: Kadierdan Kaheman~(kadierk@uw.edu) and Jason J. Bramburger~(jason.bramburger@concordia.ca).\\ Code available at: \href{https://github.com/dynamicslab/Saddle-Mediated-Transport-of-Double-Pendulum}{https://github.com/dynamicslab/Saddle-Mediated-Transport-of-Double-Pendulum}.}}
\author[2]{Jason J. Bramburger$^{*\dagger}$}
\author[3]{J. Nathan Kutz}
\author[1]{Steven L. Brunton}
\affil[1]{\small Department of Mechanical Engineering, University of Washington, Seattle, WA, 98105}
\affil[2]{\small Department of Mathematics and Statistics, Concordia University, Montr\'eal, QC, H3G 1M8}
\affil[3]{\small Department of Applied Mathematics, University of Washington, Seattle, WA, 98105\vspace{-.1in}}

\date{}

\begin{document}
\maketitle

\begin{abstract}

Pendulums are simple mechanical systems that have been studied for centuries and exhibit many aspects of modern dynamical systems theory. 
In particular, the double pendulum is a prototypical chaotic system that is frequently used to illustrate a variety of phenomena in nonlinear dynamics. 
In this work, we explore the existence and implications of codimension-1 invariant manifolds in the double pendulum, which originate from unstable periodic orbits around saddle equilibria and act as separatrices that mediate the global phase space transport.  
Motivated in part by similar studies on the three-body problem, we are able to draw a direct comparison between the dynamics of the double pendulum and transport in the solar system, which exist on vastly different scales. 
Thus, the double pendulum may be viewed as a table-top benchmark for chaotic, saddle-mediated transport, with direct relevance to energy-efficient space mission design. 
The analytical results of this work provide an existence result, concerning arbitrarily long itineraries in phase space, that is applicable to a wide class of two degree of freedom Hamiltonian systems, including the three-body problem and the double pendulum. 
This manuscript details a variety of periodic orbits corresponding to acrobatic motions of the double pendulum that can be identified and controlled in a laboratory setting.          

\end{abstract}


 \begin{spacing}{.8}
\setlength{\cftbeforesecskip}{4.pt}
\tableofcontents
 \end{spacing}

\section{Introduction }

Two of the oldest problems in classical mechanics involve the study of (1) celestial bodies interacting with each other gravitationally and (2) pendulums. These two types of problems may be derived from simple and foundational laws of nature, making them introductory examples in math and physics curricula. 
In the case of celestial dynamics, all one requires is Newton's second law, $F = ma$, and the law of universal gravitation to produce the equations of motion. These laws produce simple and surprisingly accurate models of how bodies interact in our solar system, which can be used to design fuel efficient satellite and shuttle transport~\cite{NBody1,RossBook,Surfing,MarsdenAMS,JJM1,Caillau} and provide evidence for the existence of extreme trans-Neptunian objects~\cite{PlanetNine,PlanetNine2,PlanetNine3}. 
Similar physical laws can be used to describe the motion of pendulums, which are bodies connected by rods swinging under the influence of gravity. Studies of a single pendulum, consisting of a single mass suspended by a rod, go back at least to Galileo, and their periodic motion has long been used in time keeping, and even to estimate the gravitational constant~\cite{SinglePendulum}. 
As depicted in Figure~\ref{fig:SaddleComparison}, one can add another mass and rod,  suspended from the mass of the first pendulum, to produce the {\em double pendulum}. The study of the double pendulum goes back at least to the work of Bernoulli~\cite{Bernoulli}, and now provides a prototypical example of a chaotic dynamical system~\cite{DPChaos}. Despite the many superficial differences between the motion of celestial bodies and pendulums, this manuscript will demonstrate that the phase space of the double pendulum bears a striking resemblance to that of the three-body problem in astrophysics, thus uniting the study of these centuries old problems.  
Thus, the double pendulum provides a simple experimental analog~\cite{kaheman2022experimental} with which to explore various learning and control algorithms for efficient space mission design. 

\begin{figure}[t]
    \centering
    \includegraphics[width=0.9\textwidth]{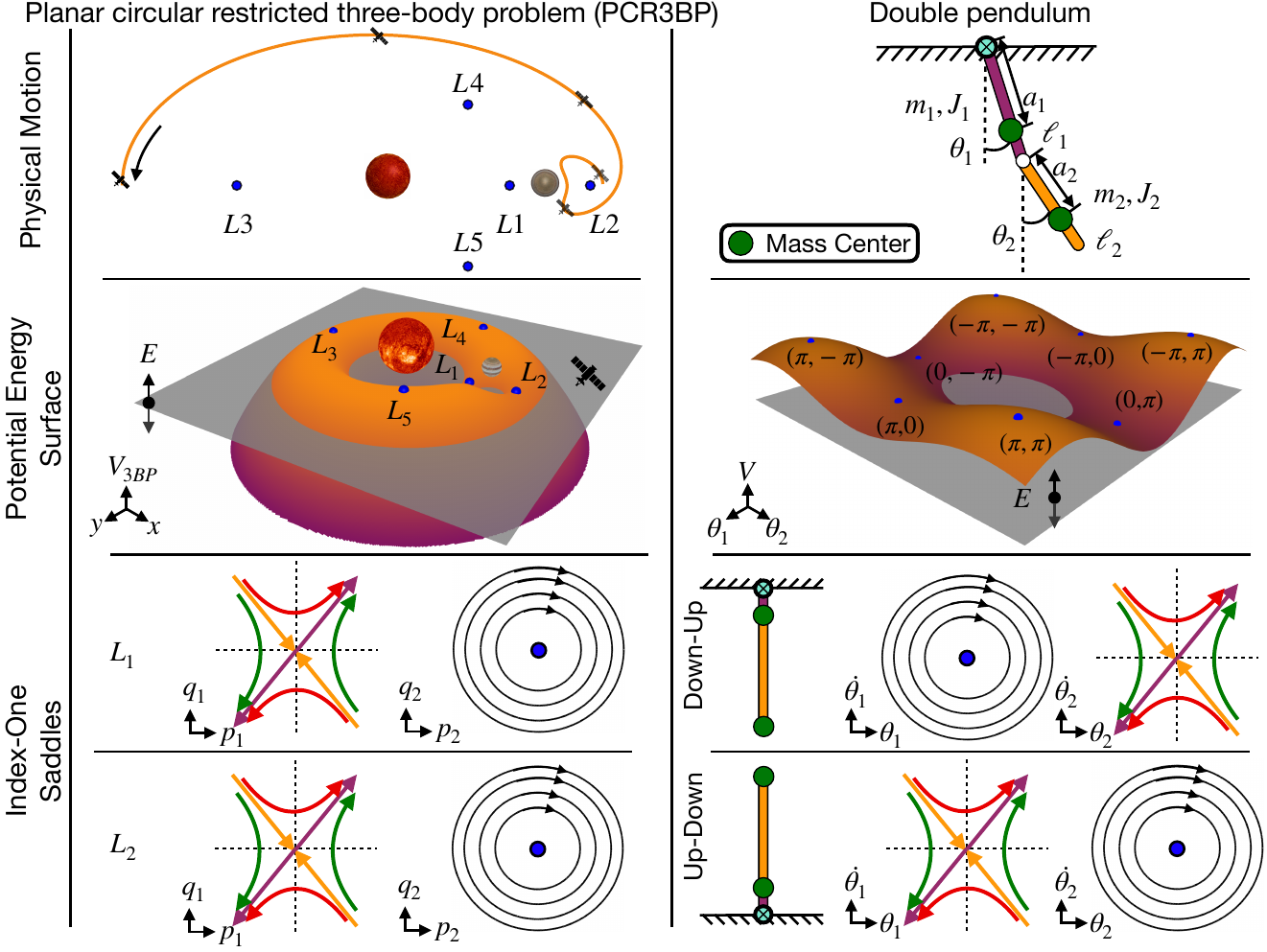}
    \caption{Despite representing dynamics on vastly different scales, the planar circular restricted 3-body problem (PCR3BP) and the double pendulum bear a striking resemblance. The Lagrange points $L_1$ and $L_2$ in the PCR3BP are analogous to the saddle Down-Up and Up-Down equilibria since they all have one-dimensional stable and unstable directions and a two-dimensional center direction.}
    \label{fig:SaddleComparison}
\end{figure}

\begin{figure}[t]
    \centering
    \includegraphics[width=1\textwidth]{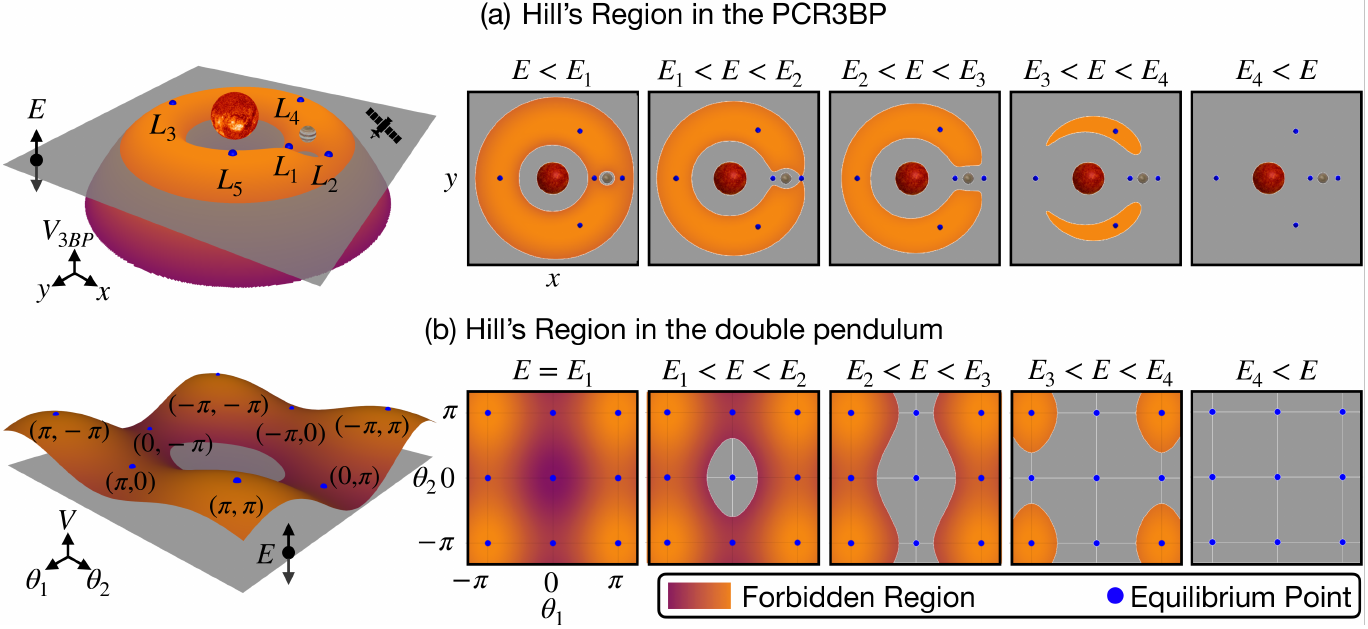}
    \caption{Hill's region comparison of PCR3BP and double pendulum. The energy values $E_i$ are meant to denote the transitions in which the Hill's region is extended to include another critical point of the energy surface.}
    \label{fig:HillsResionComparison}
\end{figure}

\subsection{Multi-Body Problems}

The study of celestial bodies in the presence of Newtonian gravity traditionally begins by examining only two bodies. This constitutes the so-called Kepler problem and it successfully describes the (approximate) elliptical motion of planets in the solar system about the sun, a satellite orbiting a planet, and binary stars orbiting each other. From a mathematical perspective, the Kepler problem can be reduced to a planar ordinary differential equation (ODE) for which all solutions lie along periodic orbits~\cite{Arnold}, thus giving simple and predictable motion. 

Adding even one more body to the system significantly complicates the dynamics, giving way to chaos~\cite{TBPreview}. This was first articulated by Poincar\'e in the late nineteenth century. 
In an effort to understand the motion of three celestial bodies -- the three-body problem -- he pioneered many of the methods that led to the development of modern chaos theory. 
Due to the complication of analyzing the three-body problem in general, a number of simplifications of the problem have been studied~\cite{TBPreview}. 
A prominent simplified model is the restricted three-body problem (R3BP), which consists of two massive bodies that generate a gravitational field and a third relatively massless body that moves in this field~\cite{Koon,RossBook}. 
In the planar circular restricted three-body problem (PCR3BP), the two massive bodies are assumed to form their own Kepler problem, maintaining a constant distance from each other and so generating a circular orbit, while the motion of the third body is confined to the orbital plane of the massive bodies. The result is a first-order four-dimensional ODE representing the planar position and velocity of the massless body, whose equations of motion are left to the appendix. Specific applications often take the massive bodies to be either the Earth and Moon~\cite{Caillau,NBody1} or the Sun and Jupiter~\cite{Koon,RossBook}, while the third body may represent a satellite, a shuttle, or a comet~\cite{TBPreview}. In Figure~\ref{fig:SaddleComparison} we present an illustration of the motion of a satellite governed by the PCR3BP in physical space, along with its potential energy landscape.     

Figure~\ref{fig:SaddleComparison} shows the  five steady-state equilibrium solutions of the PCR3BP, which are referred to as the Lagrange points $L_1$ through $L_5$. As the energy of the system increases, more of the potential energy surface, referred to as the {\em Hill's region}, becomes accessible to the third body as illustrated in Figure~\ref{fig:HillsResionComparison}. 
The Lagrange points $L_1$ and $L_2$ are saddle points, and for a continuum of energies slightly above the energy of these points, there exist unstable periodic orbits (UPOs)~\cite{Moser,Wiggins}, known as Lyapunov orbits, that resemble halos about $L_1$ and $L_2$. 
These Lyapunov orbits have two-dimensional stable and unstable manifolds, which we refer to throughout as {\em tubes} since they are diffeomorphic to a circle crossed with the real line. 
Moreover, these invariant manifold tubes have codimension-1 in the ambient phase space, acting as separatrices that define an interior and exterior region.
It has been proven that these tubes can connect to form homoclinic orbits to these UPOs~\cite{Hom3BP,McGehee,Conley}, resulting in trajectories of the system that can travel significant distances in phase space before returning to where they started. Further numerical studies~\cite{Koon} have demonstrated the presence of heteroclinic connections between Lyapunov orbits. These homoclinic and heteroclinic orbits were used to construct itineraries through phase space that allow one to jump from following the homoclinics to the heteroclinics and back. 
Thus, it is possible to construct trajectories that allow the exploration of phase space with no additional energy expenditure. Comparing observations of the orbit of the comet {\em Oterma} with the paths of the homoclinic and heteroclinic orbits in a Sun-Jupiter system reveals that its trajectory is guided by these global tube structures. Similarly, the trajectory for the {\em Genesis Discovery Mission} transited from the $L_1$ to the $L_2$ Lagrange point through nearly heteroclinic motion in the Earth-Moon system~\cite{LoGenesis,KoonGenesis}. 
Japan's {\em Hiten} lunar mission used this transport network for a low energy transfer~\cite{Hiten}, as did the ISEE-3 spacecraft~\cite{ISEE3}. 
Therefore, in the case of the PCR3BP we find that these invariant manifold tubes associated with the UPOs are crucial to not only understanding the dynamics of the phase space, but also for optimally designing itineraries for exploratory missions around planet and/or moon systems~\cite{RossL1}.

The discovery of the tube structure in the PCR3BP marks a major landmark in our understanding of celestial dynamics and has prompted the study of more complex multi-body systems~\cite{JJM2,JJM3,NBody1}, again aimed at optimal space mission design~\cite{NBody1,RossBook,Surfing,MarsdenAMS,JJM1,MasdemontReview}. Due to the abundance of these tubes in multi-body systems, they have been termed the {\em interplanetary transport network}~\cite{Interplanetary}. These tubes provide gravitationally determined pathways that require little energy to follow and can be used to design itineraries that explore major bodies of our solar system. Further, recent studies have employed computer-assisted proofs to obtain the existence of transverse homoclinic and heteroclinic connections which form the backbone of this network~\cite{CompAst1,CompAst2,CompAst3,CompAst4}.

\subsection{The Double Pendulum}

Much like multi-body systems, the base case of a single pendulum is deceptively simple compared to more general pendulum models. The single pendulum gives rise to purely periodic motion, and it resembles the Kepler problem at low energies with trajectories given by ellipsis in phase space.  Both the Kepler problem and the single pendulum have integrable dynamics, meaning that their trajectories are fully described by the one-dimensional level sets of their Hamiltonian functions~\cite{Integrable}. As a consequence of integrable structure, neither system is chaotic.  

The Kepler problem is to the single pendulum as the three-body problem is to the double pendulum. That is, like the three-body problem, the double pendulum has become a prototypical example of a chaotic system. In this work, we will develop this analogy, showing that the phase space dynamics of the double pendulum bears a striking resemblance to that of the PCR3BP. The double pendulum has two saddle steady-state solutions given by one arm hanging down with the other balanced in the upright position, referred to as the ``Down-Up" and ``Up-Down" equilibria. 
There exist a neighborhood about these points where the dynamics are topologically conjugate to those of the Lagrange points $L_1$ and $L_2$ in the PCR3BP. We illustrate this correspondence in Figure~\ref{fig:SaddleComparison} with a comparison of the potential energy landscapes of the PCR3BP and the double pendulum, a visualization of the Down-Up and Up-Down equilibria, and an illustration of the dynamics near the saddle equilibria. Like the PCR3BP, increasing the energy of the double pendulum opens up more of the potential energy landscape, increasing the size of its associated ``Hill's region'', as illustrated in Figure~\ref{fig:HillsResionComparison}. Using similar techniques to those developed for the PCR3BP~\cite{Koon}, we demonstrate numerically the existence of homoclinic and heteroclinic orbits between the UPOs at energies slightly above those of the Down-Up and Up-Down steady-states. In particular, we demonstrate that the phase space of the double pendulum is similarly organized by global tube structures that enable macroscopic transport in the system with no additional energy expenditure.            

Our work herein establishes the tube structure of the double pendulum, making it a low-stakes and relatively low-cost testbed to explore saddle mediated transport, with direct relevance to the PCR3BP~\cite{Koon}. For a fraction of the cost of building a satellite or space shuttle, a number of researchers have built double pendulums~\cite{spong1995swing,fantoni2000energy,driver2004design,timmermann2011discrete,hesse2018reinforcement,DPChaos,christini1996experimental,myers2020low,rubi2002swing,rafat2009dynamics,freidovich2008periodic,kaheman2019learning}, including the authors of this manuscript~\cite{kaheman2022experimental}.
Thus, the theoretical work put forth here can and will be tested on these physical realizations of the system in a follow-up investigation. 
Our work shows that it is possible to design an itinerary of complex acrobatic motions of the pendulum arms by combining a sequence of homoclinic and heteroclinic connections; homoclinic orbits correspond to full rotations of the unstable pendulum arm, while heteroclinic orbits connect the Up-Down and Down-Up orbits.  
These orbits provide a surprisingly tight analogy with orbits in the PCR3BP. 

\subsection{Contributions}

In this manuscript we present a detailed study of the double pendulum while drawing comparisons to previous work on the PCR3BP. In particular, we follow a similar numerical procedure to~\cite{RossBook,Koon,Barcelona} to determine the global tube structure emanating from the Down-Up and Up-Down saddle points of the double pendulum. We demonstrate that there exist homoclinic and heteroclinic connections that allow for macroscopic transport in both physical space and phase space. We then use these numerical findings to formulate a general set of hypotheses that give way to our major theoretical contribution of this manuscript. Specifically, we show that the connecting trajectories found numerically can be used to prove the existence of long homoclinic, heteroclinic, and periodic trajectories of the double pendulum that can spend arbitrarily long times near each saddle point before continuing on and transiting to the neighborhood of another saddle point. This results in a collection of trajectories that can be used to fully explore the phase space.

Our main result is sufficiently general to also apply to existing work on the PCR3BP, and so it therefore comes as a generalization of the itineraries result proven in~\cite{Koon}. Much like the work on the PCR3BP, we show that there is not only one trajectory that follows a given itinerary, but infinitely many. However, our proofs are constructive in that they explicitly demonstrate that the long trajectories are built up by shadowing the `base' homoclinic and heteroclinic orbits found to exist by the numerical methods herein. This comes from the fact that our proofs are similar to the work in~\cite{BramLDS,BramIsola}, whereby we move to local Poincar\'e sections near the UPOs and use the existence of heteroclinic connections to transit between these local sections. The result is a general description of saddle mediated transport in two degree of freedom Hamiltonian systems that goes beyond the existing theory.   

Outside of multi-body systems and the double pendulum, there are a number of other systems for which transport over vast regions of phase space can be attributed to the tubes of stable and unstable manifolds of UPOs near an index-1 saddle. A notable example is that of chemical reactions where tubes are shown to mediate the reaction and can be used to compute reaction rates and scattering phenomenon~\cite{Chemical,Chemical2,Chemical3}. Saddle mediated transport is also particularly useful in describing transitions in systems with two or more potential wells. For example, the works~\cite{Ship1,Ship2} provide a nonlinear model for ship motion with three potential wells: one for the ship being upright in the middle and wells on either side representing capsizing. Here the tubes associated to the index-1 saddles separating the potential wells explain how one transitions from the upright well into either of the capsizing wells~\cite{Naik}. Similar ideas have been used to explain snap-through buckling of a shallow arch~\cite{Arch} and the motion of a ball rolling on a saddle surface~\cite{Ball}. Our work is relevant to these studies since our analytical results are general enough to apply to all of these situations, thus extending the itineraries of the PCR3BP to a variety of Hamiltonian systems.


This paper is organized as follows. We begin in Section~\ref{sec:Background} with an introduction to index-1 saddles, of which $L_1$ and $L_2$ in the PCR3BP and the Down-Up and Up-Down equilibria in the double pendulum are examples. We describe the linear and nonlinear dynamics near these saddles and also review the importance of the global tube structure associated with them. Much of this is illustrated with recreations of known results for the PCR3BP. In Section~\ref{sec:EOM} we present a derivation of the equations of motion for the {\em physical double pendulum}. This model differs from the simple theoretical double pendulum as it includes the mass of the pendulum arms. A comparison between these two pendulum models is presented in Figure~\ref{Fig:DP_Illustrate} below and our motivation for studying the physical double pendulum comes from our physical model~\cite{kaheman2022experimental} and the goal of implementing our theory in practice in subsequent work. Section~\ref{sec:TubesDP} presents our numerical findings on the global tube structure of the double pendulum. These results include a description of the symmetries of our model in~\S\ref{sec:Symmetry}, linear analysis of the various steady-states which identifies the Down-Up and Up-Down equilibria as index-1 saddles in~\S\ref{sec:linearanalysis}, and a discussion of the numerically identified homoclinic and heteroclinic orbits in~\S\ref{sec:L1Hom}-\ref{sec:Heteroclinic}. Our main theoretical result is presented in Section~\ref{sec:Theory}, in particular in~\S\ref{sec:Theorem}. We then discuss the implication of our theoretical results on the double pendulum in~\S\ref{subsec:DPapplication} and discuss how these results extend the known results of the PCR3BP in~\S\ref{subsec:TBPapplication}. The proof of our main result is left to Section~\ref{sec:Proofs} and then we conclude in Section~\ref{sec:Conclusion} with a discussion of our findings and possible extensions.

\section{Index-1 Saddles}\label{sec:Background} 

Throughout this section we will seek to demonstrate the importance of saddle steady-states to the global dynamics of a Hamiltonian system.  In particular, the global dynamics of both the PCR3BP~\cite{RossBook,Koon} and the double pendulum are mediated by {\em index-1} saddle points. 
Index-1 saddle points are fixed points characterized by a single unstable saddle degree of freedom, with corresponding positive and negative eigenvalue, and the remainder of the degrees of freedom being linearly stable center dynamics, each with corresponding pairs of purely imaginary, complex conjugate eigenvalues.  
The terminology index-1 refers to the fact these orbits have one-dimensional stable and unstable manifolds associated to them.
Index-1 saddles are sometimes referred to as \emph{rank-1} saddles.  

Throughout we denote the Hamiltonian function by $\mathcal{H}$. 
Much of the theory presented for index-1 saddles generalizes to arbitrary numbers of degrees of freedom; however, for simplicity, we consider Hamiltonian systems with two degrees of freedom in this work. 
The resulting dynamical system has a four-dimensional phase space. Our interest will be in understanding the local dynamics near these index-1 saddles.


\subsection{Linear Dynamics Near an Index-1 Saddle}

Let us begin by assuming that we have an equilibrium of the resulting 4-dimensional ODE and that linearizing about this equilibrium results in four eigenvalues $\pm \lambda$ and $\pm \mathrm{i}\omega$, for some real $\lambda,\omega > 0$. Through an invertible transformation we may translate this equilibrium to the origin and bring the resulting linearized dynamics into the Jordan normal form
\begin{equation}\label{Linear}
	\begin{bmatrix}
		\dot p_1 \\ \dot q_1 \\ \dot p_2 \\ \dot q_2
	\end{bmatrix} = \begin{bmatrix}
		-\lambda & 0 & 0 & 0 \\
		0 & \lambda & 0 & 0 \\
		0 & 0 & 0 & -\omega \\
		0 & 0 & \omega & 0 \\
	\end{bmatrix}\cdot\begin{bmatrix}
		p_1 \\ q_1 \\ p_2 \\ q_2
	\end{bmatrix}. 
\end{equation} 
We refer the reader to \cite[Section~2.6]{RossBook} for a demonstration of how to arrive at the above normal form for index-1 saddles of the PCR3BP. We further comment that the linearization about all index-1 saddles in a 2 degree-of-freedom Hamiltonian system are topologically conjugate to the normal form \eqref{Linear}, with the proof being the same as the demonstration in \cite{RossBook}. The linearized system~\eqref{Linear} is itself a conservative system with quadratic Hamiltonian given by
\begin{equation}\label{LinearHam}
	\mathcal{H}_\mathrm{loc}(p_1,q_1,p_2,q_2) = \lambda p_1q_1 + \frac{\omega}{2}(p_2^2 + q_2^2).
\end{equation} 
Solutions of~\eqref{Linear} are given by
\begin{equation}\label{LinearSol}
	p_1(t) = p_1^0\mathrm{e}^{-\lambda t}, \quad q_1(t) = q_1^0\mathrm{e}^{\lambda t}, \quad p_2(t) + \mathrm{i}q_2(t) = (p_2^0 + \mathrm{i}q_2^0)\mathrm{e}^{-\mathrm{i}\omega t}, 
\end{equation}
for any constants $p_1^0,q_1^0,p_2^0,q_2^0\in\mathbb{R}$. Importantly,~\eqref{Linear} has three constants of motion
\begin{equation}
	C_1 = p_1q_1, \quad C_2 = p_2^2 + q_2^2, \quad C_3 = \mathcal{H}_\mathrm{loc}
\end{equation}
which we will use to understand the flow near the saddle point.

From the above, we can see that projecting the dynamics into the $(p_1,q_1)$-plane results in a standard planar saddle system. Similarly, projecting into the $(p_2,q_2)$-plane results in a linear center consisting of harmonic oscillator motion. See Figure~\ref{fig:SaddleComparison} for an illustration. For some fixed $h \in \mathbb{R}^+$ we can then consider the dynamics of~\eqref{Linear} inside the level set $\mathcal{H}_\mathrm{loc} = h$. First, rearranging~\eqref{LinearHam} gives
\begin{equation}
	p_2^2 + q_2^2 = \frac{2}{\omega}(h - \lambda p_1q_2),
\end{equation}
so that in the bounding scenario that $C_1 = h/\lambda$ we have that the dynamics in the $(p_2,q_2)$-plane collapse to the point $(0,0)$ and trajectories lie along the hyperbolas $p_1q_1 = h/\lambda$. For $0 < C_1 < h/\lambda$ the dynamics in the $(p_2,q_2)$-plane lie along a circle and the trajectories are diffeomorphic to a circle crossed with a line. 

When $p_1 = q_1 = 0$ the resulting dynamics are confined to the circles
\begin{equation}
	S^1_h := \bigg\{(p_2,q_2):\ p_2^2 + q_2^2 = \frac{2h}{\omega}\bigg\}
\end{equation}
for each $h > 0$. These invariant circles constitute periodic orbits of the system, referred to as {\em Lyapunov orbits} in the context of celestial dynamics governed by many body equations. The above circles are only one of three nontrivial examples of trajectories with the constant of motion $C_1 = 0$. Simply restricting to the invariant manifold $q_1 = 0$ again gives $C_1 = 0$, but now the dynamics of $p_1$ need not be trivial. Using the solution~\eqref{LinearSol} we find that $p_1(t) \to 0$ as $t \to \infty$ while the dynamics in the $(p_2,q_2)$-plane are still restricted to $S^1_h$. Therefore, the set $q_1 = 0$ is the stable manifold of the periodic orbit $S^1_h$, denoted $W^s(S^1_h)$. Similarly, taking $p_1 = 0$ and using the fact that $q_1(t) \to 0$ as $t \to -\infty$, we have that the set $p_1 = 0$ is the unstable manifold of the period orbit $S^1_h$, denoted $W^u(S^1_h)$. Both the stable and unstable manifolds of $S^1_h$ are homeomorphic to circles crossed with lines, which we henceforth refer to as tubes. Furthermore, the presence of both stable and unstable manifolds for the periodic orbits implies they are all saddles like the equilibrium they surround, and therefore are unstable.

\subsection{Nonlinear Dynamics Near an Index-1 Saddle}

In the previous subsection we described the linearized dynamics near an index-1 saddle point in a Hamiltonian system. We now extend this discussion to the nonlinear dynamics near the index-1 saddle. The same invertible change of variable that brings the linear dynamics to the system~\eqref{Linear} transforms the Hamiltonian of the nonlinear system into the form
\begin{equation}\label{LinearHam2}
	\mathcal{H}(p_1,q_1,p_2,q_2) = \mathcal{H}_\mathrm{loc}(p_1,q_1,p_2,q_2) + h.o.t.,
\end{equation}
where {\em h.o.t.} denotes the `higher order terms' which in this case are at least cubic and $\mathcal{H}_\mathrm{loc}$ is given in~\eqref{LinearHam}. Similarly, the dynamics in a neighborhood of the equilibrium (shifted to the origin in the $p_1,q_1,p_2,q_2$ variables) takes the form 
\begin{equation}\label{Nonlinear}
	\begin{bmatrix}
		\dot p_1 \\ \dot q_1 \\ \dot p_2 \\ \dot q_2
	\end{bmatrix} = \begin{bmatrix}
		-\lambda & 0 & 0 & 0 \\
		0 & \lambda & 0 & 0 \\
		0 & 0 & 0 & -\omega \\
		0 & 0 & \omega & 0 \\
	\end{bmatrix}\cdot\begin{bmatrix}
		p_1 \\ q_1 \\ p_2 \\ q_2
	\end{bmatrix} + h.o.t.
\end{equation} 
where now the higher order terms are at least quadratic. It is apparent from the fact that the linearization includes a complex conjugate pair of purely imaginary eigenvalues that one cannot apply results such as the Hartman--Grobman theorem to conclude that the dynamics~\eqref{Nonlinear} are locally conjugate to the dynamics of~\eqref{Linear}. Fortunately, the works~\cite{Moser,Wiggins} exploit the Hamiltonian structure of the underlying dynamical system to provide the necessary results that demonstrate this. Precisely, there exists a small ball around $(p_1,q_1,p_2,q_2) = (0,0,0,0)$ for which the dynamics of~\eqref{Nonlinear} are equivalent to the dynamics of~\eqref{Linear}. 

The consequence of the above is that in a neighborhood of any index-1 saddle point comprises a continuum of UPOs. Furthermore, these UPOs have two-dimensional stable and unstable manifolds resembling tubes that eventually leave the region of validity for the conjugacy with the linear system~\eqref{Linear}. Using~\eqref{LinearHam2}, it follows that these UPOs can only be guaranteed to exist for $\mathcal{H} = h$ with $0 < h \ll 1$ since the higher order terms become more important as $h$ increases. By abuse of notation we will refer to these UPOs again as $S^1_h$ since the UPOs of~\eqref{Linear} and~\eqref{Nonlinear} lie in one-to-one correspondence when $h$ is small. Moreover, the local components of the stable and unstable manifolds, $W^s_\mathrm{loc}(S^1_h)$ and $W^u_\mathrm{loc}(S^1_h)$, respectively, are approximated by the sets $\{q_1 = 0\}$ and $\{p_1 = 0\}$, respectively. These points become important in the following sections for numerically following the global stable and unstable manifolds of different UPOs.

\subsection{Global Invariant Manifold Tube Dynamics}

The previous subsections focus on local aspects of the Hamiltonian phase space in neighborhoods of index-1 saddles. Per the previous discussion, the saddles and their nearby UPOs all have stable and unstable manifolds that extend beyond these neighborhoods. Importantly, intersections of these stable and unstable manifolds give the existence of global structures such as homoclinic and heteroclinic orbits in the phase space. These homoclinic and heteroclinic orbits organize the global geometry of phase space and enable transport over vast regions of space with no extra energy expenditure. 

Note that with a 2 degree of freedom Hamiltonian system, we are restricted to the (generically) three-dimensional level sets of the Hamiltonian function. In one of these three-dimensional level sets, the index-1 saddles have one-dimensional stable and unstable manifolds, meaning that the existence of a homoclinic orbit is unlikely since these stable and unstable manifolds would have to coincide. Such a phenomenon is not robust with respect to generic parameter manipulation, and so this situation should not be expected without additional assumptions such as symmetries. Alternatively, for a connected continuum of values of the Hamiltonian slightly above the value at the index-1 saddle we have the UPOs, each of which have two-dimensional stable and unstable manifolds inside the three-dimensional level sets. In this case an intersection of these manifolds is robust and expected to be one-dimensional, leading to the existence of a homoclinic orbit. The robustness of these intersections implies that the existence of a homoclinic orbit for one value of the Hamiltonian is expected to imply the existence for values of the Hamiltonian in an open neighborhood of said value. Heteroclinic connections between index-1 saddles are similarly robust and are made up of intersections of one UPO's stable manifold intersecting along a generically one-dimensional curve with another UPO's unstable manifold. We again expect the existence of such a heteroclinic connection to exist for a range of values of the Hamiltonian, if it exists at all.      

\begin{figure}[t]
    \centering
    \includegraphics[width=0.85\textwidth]{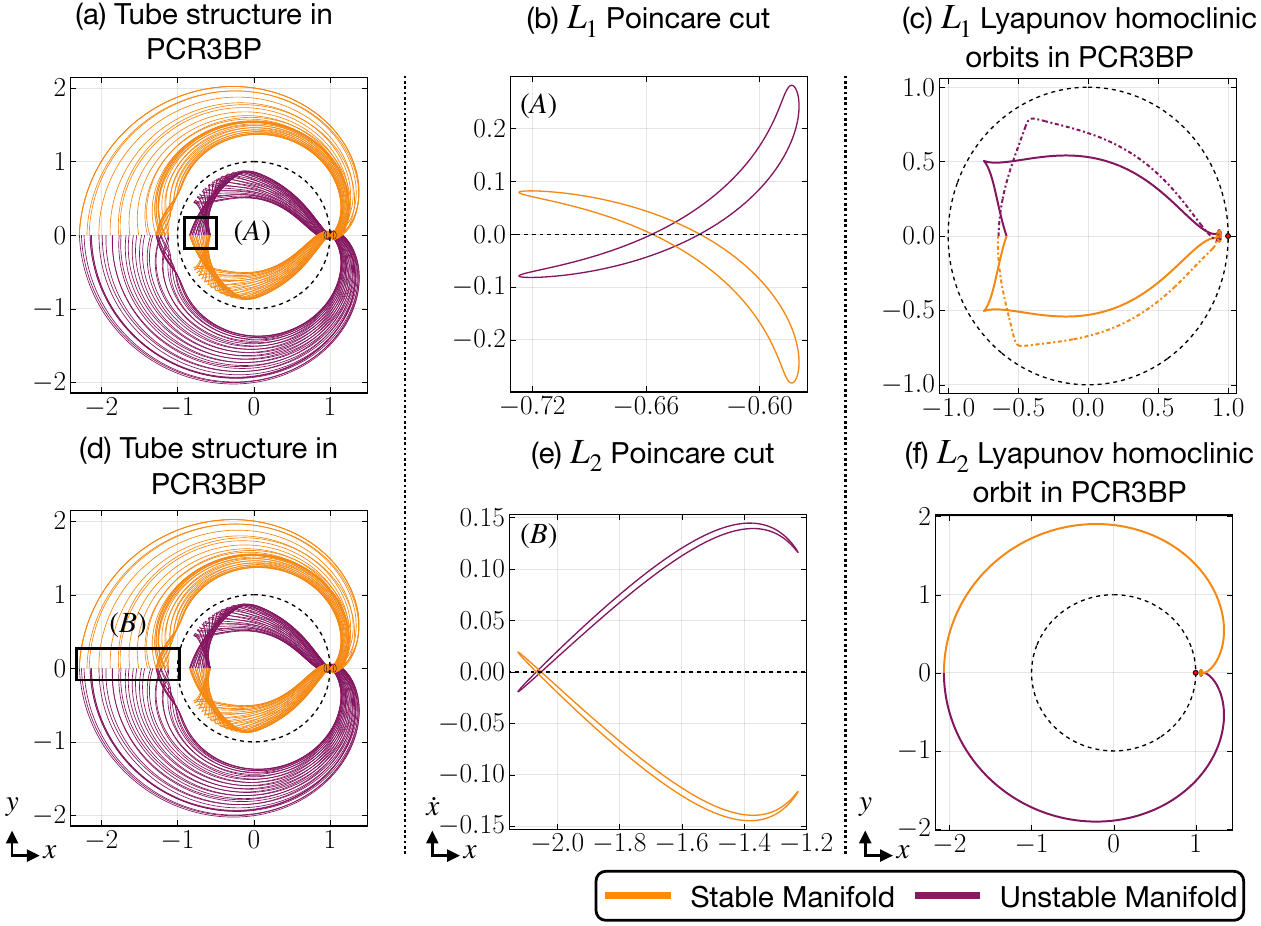}
    \caption{Numerically confirming the existence of homoclinic trajectories to the $L_1$ and $L_2$ Lyapunov orbits in the PCR3BP is done by finding intersections between their stable and unstable manifolds in the Poincar\'e section $y = 0$ with $x < 0$. The manifolds are shown in (a) and (d), while their intersections with the Poincar\'e section are illustrated in (b) and (e). Intersections of the stable and unstable manifolds of each Lyapunov orbit results in a homoclinic orbit, some of which are depicted in (c) and (f). In (c), both symmetric (solid) and asymmetric (dashed) $L_1$ homoclinic orbits are shown. }
    \label{fig:TBP_HomoclinicOrbit}
\end{figure}

The existence of homoclinic orbits to the UPOs near index-1 saddles has notably been proven for the Lyapunov orbits around the $L_1$ Lagrange point in the PCR3BP~\cite{Conley,McGehee}. These results were improved by~\cite{LMS} to show that under appropriate conditions one could conclude that these intersections were transverse, meaning that they are indeed robust. More in line with our work herein, the work~\cite{Koon} followed the ideas put forth by~\cite{Barcelona} to explore the system numerically and detect the existence of these homoclinic orbits to not only Lyapunov orbits of the $L_1$ Lagrange point, but also the $L_2$ point. In a process that will be described in more detail below, one simulates the unstable manifold of the Lyapunov orbit forward in time and the stable manifold backward in time until they meet at an appropriately defined Poinacar\'e section. The intersection of these manifolds with the Poincar\'e section will be diffeomorphic to a circle and the intersection of these circles represent the desired homoclinic trajectories. This process and the found homoclinic orbits are illustrated in Figure~\ref{fig:TBP_HomoclinicOrbit}. We refer to the appendix for details on the PCR3BP including the equations of motion, the location of the Lagrange points $L_1$ and $L_2$, and the parameter values used in our computations.

\begin{figure}[t]
    \centering
    \includegraphics[width=0.75\textwidth]{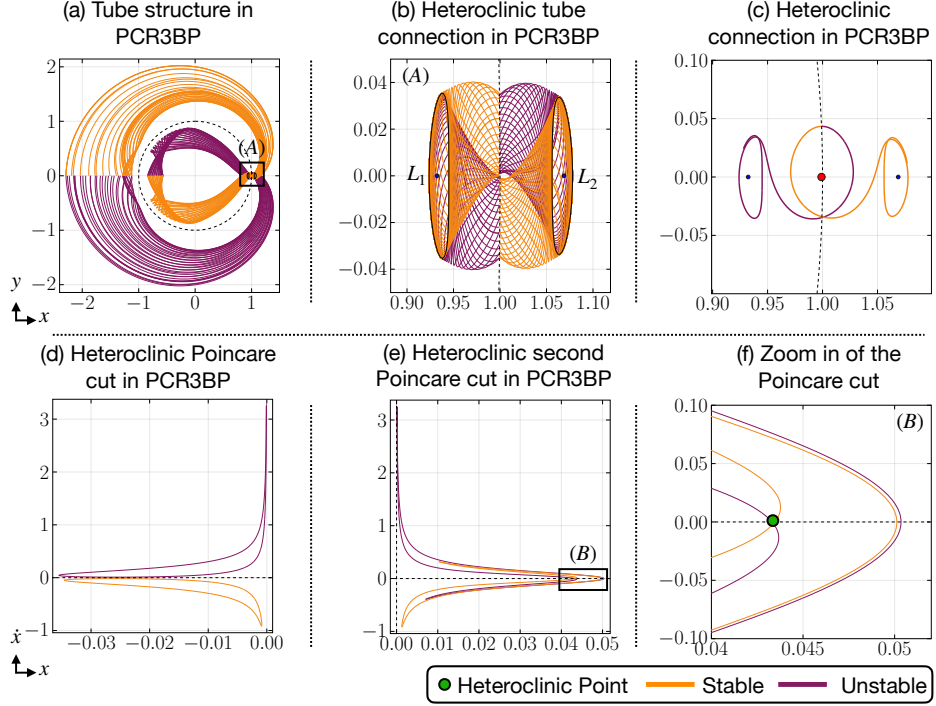}
    \caption{Heteroclinic orbits between the $L_1$ and $L_2$ Lyapunov orbits of the PCR3BP are found numerically in the same way as the homoclinic orbits in Figure~\ref{fig:TBP_HomoclinicOrbit}, as shown in (a) and (b). A heteroclinic orbit from an $L_1$ Lyapunov orbit to an $L_2$ Lyapunov orbit is shown in (c) and is identified from manifold intersections in the Poincar\'e section in (d). Similarly, (e) and (f) show manifold intersections in the Poincar\'e section resulting in heteroclinic orbits from an $L_2$ to an $L_1$ Lyapunov orbit.}
    \label{fig:TBP_HeteroclinicOrbit}
\end{figure}

Beyond just the homoclinic orbits,~\cite{Koon} also numerically demonstrates the existence of heteroclinic orbits between the $L_1$ and $L_2$ Lyapunov orbits. These are obtained in the same way as the homoclinic orbits in that stable and unstable manifolds are simulated backward and forward in time, respectively, to meet an appropriate Poincar\'e section to find intersections of these manifolds. This process and the results are summarized in Figure~\ref{fig:TBP_HeteroclinicOrbit}. The effect this has is that for values of the Hamiltonian slightly above that of the index-1 saddles there exist level sets containing Lyapunov orbits to both $L_1$ and $L_2$ which each have homoclinic orbits and are connected by a heteroclinic orbit. These orbits are used in~\cite{Koon} to construct itineraries through phase space that allow one to jump from following the homoclinics to the heteroclinics and back, thus providing trajectories that allow the lightweight object to fully explore the phase space with no additional energy expenditure. In what follows we will perform a similar investigation for the double pendulum by identifying homoclinic and heteroclinic trajectories to index-1 saddles. We further prove a general result that uses these orbits that we identify numerically to construct arbitrarily long trajectories in phase space that are analogous to the itineraries identified.

\section{Equations of Motion for the Physical Double Pendulum} \label{sec:EOM}

In this section we provide the relevant equations of motion for the physical double pendulum. We note that many mathematical investigations of the double pendulum use the {\em point-mass double pendulum}, which ignores the mass of the pendulum rod. To make our investigation applicable to future laboratory experiments, we will consider the {\em physical double pendulum} which includes the mass of the pendulum arms. In Figure~\ref{Fig:DP_Illustrate} we provide a comparison between the theoretical and physical double pendulum models and we direct the reader to the appendix for the equations of motion for the theoretical double pendulum. 

To begin, let us denote the mass of the first pendulum arm by $m_1$ and the mass of second pendulum arm by $m_2$. 
Then, let $\ell_1$ and $\ell_2$ be the lengths of the first and second arms, respectively. Next, define $a_1$ and $a_2$ as the positions of the centers of mass of the first and second pendulum arm, respectively. Let $J_1$ and $J_2$ denote the moment of inertia of the first and second pendulum arm. The constant of gravitational acceleration is $g$. Finally, we take $\theta_1$ and $\theta_2$ to be the time-dependent rotational angles of pendulum arms, as illustrated in Figure~\ref{Fig:DP_Illustrate}.

\begin{figure}[t]
    \centering
    \includegraphics[width=0.7\textwidth]{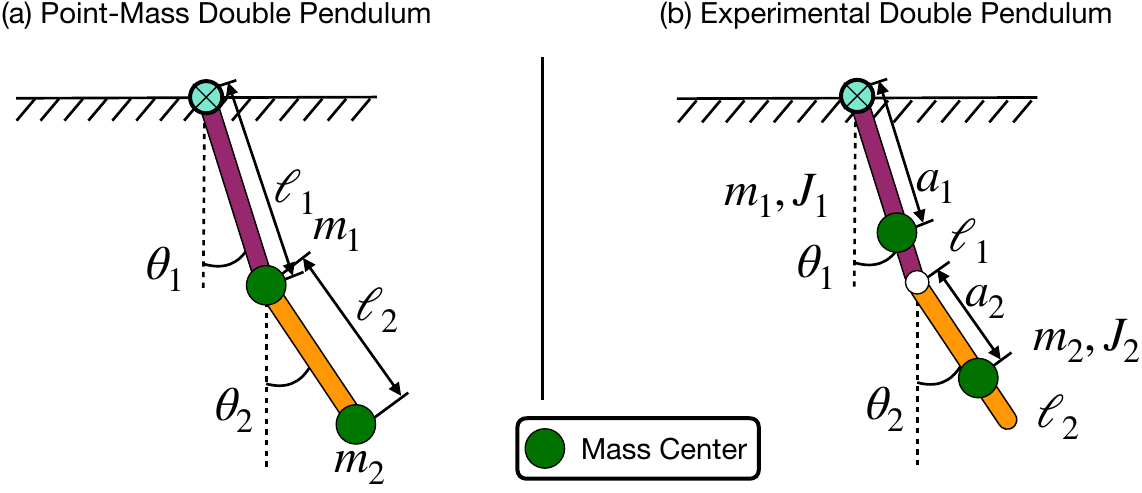}
    \caption{This figure illustrates the difference between the point-mass model of the double pendulum (left) and the more realistic physical double pendulum (right) studied herein. The physical model includes both the inertial ($J_i$) and mass $(m_i)$ of the pendulum arms.}
    \label{Fig:DP_Illustrate}
\end{figure}

The kinetic energy of the physical double pendulum is given by
\begin{equation}\label{eq:edp_kinetic_energy}
    T=\frac{1}{2}\left(m_1a_1\dot{\theta}_1^2 +m_2\left(\ell_1^2\dot{\theta}_1^2+a_2^2\dot{\theta}_2^2+2\ell_1a_2\dot{\theta}_1\dot{\theta}_2\cos(\theta_1-\theta_2)\right)+J_1\dot{\theta}_1^2+J_2\dot{\theta}_2^2\right)
\end{equation}
and the potential energy is
\begin{equation}\label{eq:edp_potential_energy}
    \begin{split}
        V&=-g\left(m_1a_1\cos(\theta_1)+m_2(\ell_1\cos(\theta_1)+a_2\cos(\theta_2))\right).
    \end{split}
\end{equation}
Using~\eqref{eq:edp_kinetic_energy} and~\eqref{eq:edp_potential_energy}, the Lagrangian of the physical double pendulum can be calculated as
\begin{equation}\label{eq:edp_lagrangian}
    L=T-V,
\end{equation}
and the Hamiltonian of the system is given by
\begin{equation}\label{eq:edp_hamiltonian}
    \mathcal{H} =T+V.
\end{equation}
Throughout our investigation in this manuscript we will neglect the effect of friction, although this is important for real experimental control.

Using the Lagrangian~\eqref{eq:edp_lagrangian}, the equations of motion of the double pendulum are given by
\begin{subequations}
\label{eq:edp_eom}
\begin{align}
        \frac{\partial}{\partial t}\frac{\partial L}{\partial \dot{\theta}_1}-\frac{\partial L}{\partial \theta_1}&=0, \label{eq:edp_eom1}\\
        \frac{\partial}{\partial t}\frac{\partial L}{\partial \dot{\theta}_2}-\frac{\partial L}{\partial \theta_2}&=0. \label{eq:edp_eom2}
\end{align}
\end{subequations}
Equation~\eqref{eq:edp_eom1} yields
\begin{equation}\label{eq:edp_eom_cal1}
\begin{split}
0 &= J_1\ddot{\theta}_1 + \ell_1^2\ddot{\theta}_1m_2
        + a_1^2\ddot{\theta}_1m_1 +\ell_1gm_2\sin(\theta_1) + a_1gm_1\sin(\theta_1) \\ &\qquad + \ell_1a_2\dot{\theta}_2^2m_2\sin(\theta_1 - \theta_2)
        + \ell_1a_2\ddot{\theta}_2m_2\cos(\theta_1 - \theta_2),\\
\end{split}
\end{equation}
while equation~\eqref{eq:edp_eom2} yields
\begin{align}\label{eq:edp_eom_cal2}
        0 &= J_2\ddot{\theta}_2 - \ell_1a_2\dot{\theta}_1^2m_2\sin(\theta_1 - \theta_2) + \ell_1a_2\ddot{\theta}_1m_2\cos(\theta_1 - \theta_2).
\end{align}
From~\eqref{eq:edp_eom_cal1} and~\eqref{eq:edp_eom_cal2} we can solve for $\ddot\theta_1$ and $\ddot\theta_2$, giving
\begin{subequations}\label{eq:edp_ddtheta}
\begin{align}
        \ddot{\theta}_1&=\frac{A_{10}\sin(\theta_1)+A_{11}\sin(\theta_1-\theta_2)+A_{12}\sin(\theta_1-2\theta_2)+A_{22}\sin(2\theta_1-2\theta_2)}{D(\theta_1 - \theta_2)}\\
            \ddot{\theta}_2&=\frac{B_{01}\sin(\theta_2)+B_{11}\sin(\theta_1-\theta_2)+B_{21}\sin(2\theta_1-\theta_2)+B_{22}\sin(2\theta_1-2\theta_2)}{D(\theta_1 - \theta_2)}
\end{align}
\end{subequations}
where
\begin{equation}
    \begin{split}
        A_{10} &= - \ell_1ga_2^2m_2^2 - 2a_1gm_1a_2^2m_2 - 2J_2l_1gm_2 - 2J_2a_1gm_1 \\
        A_{11} &= - 2\ell_1a_2^3\dot{\theta}_2^2m_2^2 - 2J_2\ell_1a_2\dot{\theta}_2^2m_2 \\
        A_{12} &= -\ell_1a_2^2gm_2^2 \\
        A_{22} &= -\ell_1^2a_2^2\dot{\theta}_1^2m_2^2 \\
        B_{01} &= -a_2m_2(gm_2\ell_1^2 - gm_1\ell_1a_1 + 2gm_1a_1^2 + 2J_1g)\\
        B_{11} &= a_2m_2(2m_2\ell_1^3\dot{\theta}_1^2 + 2m_1\ell_1a_1^2\dot{\theta}_1^2 + 2J_1\ell_1\dot{\theta}_1^2)\\
        B_{21} &= a_2m_2(gm_2\ell_1^2 + a_1gm_1\ell_1)\\
        B_{22} &= 2m_2(gm_2\ell_1^2 + a_1gm_1\ell_1)\\
        D(\theta_1 - \theta_2) &=-\ell_1^2a_2^2m_2^2cos(\theta_1-\theta_2)^2 + \ell_1^2a_2^2m_2^2 +  J_2\ell_1^2m_2 + m_1a_1^2a_2^2m_2 \\&\qquad + J_2m_1a_1^2 + J_1a_2^2m_2 + J_1J_2.
    \end{split}
\end{equation}
Here we use the convention that $A_{ij},B_{ij}$ are the coefficients of the energy preserved Hamiltonian system for the $\ddot\theta_1$ and $\ddot\theta_2$ equations, respectively, while the subscripts denote the coefficients on the $\theta_1$ and $\theta_2$ terms inside the sine function they are multiplied against. The notation $D(\theta_1 - \theta_2)$ represents the denominator of both equations. Writing~\eqref{eq:edp_ddtheta} as a first-order ODE gives
\begin{subequations}\label{eq:edp_ode}
\begin{align}
        \dot{\theta}_1&=\omega_1,\\
        \dot{\theta}_2&=\omega_2,\\
        \dot{\omega}_1&=\frac{A_{10}\sin(\theta_1)+A_{11}\sin(\theta_1-\theta_2)+A_{12}\sin(\theta_1-2\theta_2)+A_{22}\sin(2\theta_1-2\theta_2)}{D(\theta_1 - \theta_2)},\\
        \dot{\omega}_2&=\frac{B_{01}\sin(\theta_2)+B_{11}\sin(\theta_1-\theta_2)+B_{21}\sin(2\theta_1-\theta_2)+B_{22}\sin(2\theta_1-2\theta_2)}{D(\theta_1 - \theta_2)}.
\end{align}
\end{subequations}
System~\eqref{eq:edp_ode} therefore represents the full equations of motion for the double pendulum on a cart that we investigate for the remainder of this manuscript. 

Throughout this manuscript our numerical experiments will use the parameter values
\begin{equation}\label{ParamVals} 
	\begin{split}
		&m_1 = 0.0938\mathrm{kg}, \quad m_2 = 0.1376\mathrm{kg}, \quad a_1 = 0.1086\mathrm{m}, \quad a_2 = 0.1168\mathrm{m}, \\
		&\ell_1= 0.1727\mathrm{m}, \quad \quad J_1 = 10^{-4}\mathrm{kgm}^2, \quad J_2 = 10^{-4}\mathrm{kgm}^2, \quad g = 9.808\mathrm{m}/\mathrm{s}^2.
	\end{split}
\end{equation}
All of the above values were obtained via parameter estimation from the physical model constructed in~\cite{kaheman2022experimental} and chosen to demonstrate that these results are physically realizable. We have also found that other similar choices of parameter values lead to nearly identical results and so we do not expect that the restriction to these parameter values is a limitation of this work. Finally, notice that the length of the second arm, $\ell_2$, is notably absent from the equations of motion~\eqref{eq:edp_ode} and so its value is not necessary for our investigation of the physical double pendulum.

\section{Tube Dynamics of the Double Pendulum}\label{sec:TubesDP}

In this section we will consider the tube structure of transport between index-1 saddles of system~\eqref{eq:edp_ode}. Over the following subsections we seek to identify its index-1 saddles and their global tube dynamics. This will help us to identify the macroscopic transport mechanisms of the physical double pendulum modelled by equations~\eqref{eq:edp_ode} in a similar manner to what has been done for a variety of other Hamiltonian systems, including the PCR3BP. 

We begin in~\S\ref{sec:Symmetry} by describing the reversible symmetry of~\eqref{eq:edp_ode}. This reversible symmetry will help us to classify the trajectories between neighborhoods of index-1 saddles as either symmetric or asymmetric, with those that are asymmetric coming in pairs related by applying the symmetry transformation. We then proceed to identify the index-1 saddles of~\eqref{eq:edp_ode} in~\S\ref{sec:linearanalysis} by analyzing the linearization of the system about each of its steady-states. We identify two index-1 saddles given by one of the pendulum arms standing straight up (unstable) and the other hanging down (stable). We provide a numerical justification for the existence of homoclinic orbits to UPOs in the neighborhoods of these saddles in~\S\ref{sec:L1Hom} and~\S\ref{sec:L2Hom}. We then conclude this section in~\S\ref{sec:Heteroclinic} by numerically demonstrating the existence of heteroclinic orbits between UPOs in the neighborhoods of each of the index-1 saddles.


\subsection{Reversible Symmetry}\label{sec:Symmetry}

Beyond the Hamiltonian structure of~\eqref{eq:edp_ode}, we note that the system is also reversible in the sense that if $(\theta_1(t),\theta_2(t),\omega_1(t),\omega_2(t))$ is a solution, then so is $(\theta_1(-t),\theta_2(-t),-\omega_1(-t),-\omega_2(-t))$. Precisely, we define the reverser 
\begin{equation}\label{Reverser}
	\mathcal{R} = \begin{bmatrix}
		1 & 0 & 0 & 0 \\ 0 & 1 & 0 & 0 \\ 0 & 0 & -1 & 0 \\ 0 & 0 & 0 & -1 \\
	\end{bmatrix}
\end{equation} 
and say that $\Theta(t) := (\theta_1(t),\theta_2(t),\omega_1(t),\omega_2(t))$ and $\mathcal{R}\Theta(-t)$ are solutions. Using this property, we will refer to a solution as {\em symmetric} if $\Theta(t) = \mathcal{R}\Theta(-t)$ for all $t \in \mathbb{R}$. Importantly, if $\Theta_*$ is a symmetric equilibrium solution of~\eqref{eq:edp_ode}, i.e. $\mathcal{R}\Theta_* = \Theta_*$, then the associated stable manifold $W^s(\Theta_*)$ and unstable manifold $W^u(\Theta_*)$ are such that $W^s(\Theta_*) = \mathcal{R}W^u(\Theta_*)$. Therefore, homoclinic orbits of symmetric equilibria are either symmetric or asymmetric, for which the latter case implies that they come in pairs, related by applying $\mathcal{R}$ and reversing the flow of time. The preceding discussion also holds for the symmetric UPOs near a symmetric index-1 saddle which are the focus of much of our discussion going forward. We will therefore emphasize the symmetric and asymmetric homoclinic orbits that can be observed in our numerical findings.

\subsection{Linear Analysis}\label{sec:linearanalysis}

From the equations of motion~\eqref{eq:edp_ode} we can see that equilibria come in the form $(\theta_1,\theta_2,\omega_1,\omega_2) = (\pi k_1,\pi k_2,0,0)$ for every pair of integers $(k_1,k_2) \in \mathbb{Z}^2$. All of these equilibria are symmetric with respect to the action of the reverser $\mathcal{R}$, defined in~\eqref{Reverser}. Although we have obtained a lattice of equilibria, periodicity of the components $\theta_{1,2}$ implies that we may restrict our analysis to the four distinct steady-states
\begin{subequations}\label{DPsteady}
\begin{align}
	\mathrm{Down-Down}: \quad &(\theta_1,\theta_2,\omega_1,\omega_2) = (0,0,0,0) \\
	\mathrm{Down-Up}: \quad &(\theta_1,\theta_2,\omega_1,\omega_2) = (0,\pi,0,0) \\
	\mathrm{Up-Down}: \quad &(\theta_1,\theta_2,\omega_1,\omega_2) = (\pi,0,0,0) \\
	\mathrm{Up-Up}: \quad &(\theta_1,\theta_2,\omega_1,\omega_2) = (\pi,\pi,0,0).
\end{align}  
\end{subequations}
The naming of the states represents the vertical position of the arms of the double pendulum. Precisely, Down-Down has both arms hanging straight down, Down-Up has the first arm with parameters $(\ell_1,m_1)$ hanging down and the other standing straight up, Up-Down has the first arm standing straight up and the other hanging down, and the Up-Up state has both arms standing straight up. We note that we have listed the steady-states in~\eqref{DPsteady} in order of ascending energy values, according to the Hamiltonian~\eqref{eq:edp_hamiltonian}. Our results that follow apply for all choices of parameters, not just those provided in~\eqref{ParamVals}.  

Linearizing~\eqref{eq:edp_ode} about an equilibrium $(\theta_1,\theta_2,\omega_1,\omega_2) = (\pi k_1,\pi k_2,0,0)$ with $k_1,k_2\in\mathbb{Z}$ results in a Jacobian matrix of the form
\begin{equation}\label{DPLinMat}
	\begin{bmatrix}
		0 & 0 & 1 & 0 \\
		0 & 0 & 0 & 1 \\
		\frac{\sigma_2\sigma_3}{\sigma_1^2}-\frac{\sigma_4}{2\sigma_1} & \frac{\sigma_5}{\sigma_1}-\frac{\sigma_2\sigma_3}{\sigma_1^2} & 0 & 0 \\
		\frac{\sigma_6}{\sigma_1}-\frac{\sigma_7\sigma_8}{\sigma_1^2} & -\frac{\sigma_{9}}{\sigma_1}-\frac{\sigma_7\sigma_8}{\sigma_1^2} & 0 & 0
	\end{bmatrix}
\end{equation} 
with 
\begin{equation}
    \begin{split}
        \sigma_1&=- \ell_1^2a_2^2m_2^2\cos(k_1\pi - k_2\pi)^2 + \ell_1^2a_2^2m_2^2 \\
            & \quad + J_2\ell_1^2m_2 + m_1a_1^2a_2^2m_2 + J_2m_1a_1^2 + J_1a_2^2m_2 + J_1J_2, \\
        \sigma_2 &=\ell_1^2 a_2^2 g m_2^2 \cos(k_1\pi - k_2\pi) \sin(k_1\pi - k_2\pi), \\
       \sigma_3 &=(\ell_1 a_2^2 m_2^2 \sin(k_1\pi) + 2 J_2 \ell_1 m_2 \sin(k_1\pi) \\
            &\quad + \ell_1 a_2^2 m_2^2 \sin(k_1\pi - 2 k_2\pi) + 2 J_2 a_1 m_1 \sin(k_1\pi) + 2 a_1 a_2^2 m_1 m_2 \sin(k_1\pi)), \\
        \sigma_4 &=(\ell_1 a_2^2 m_2^2 \cos(k_1\pi) + 2 J_2 \ell_1 m_2 \cos(k_1\pi)\\
            &\quad + \ell_1 a_2^2 m_2^2 \cos(k_1\pi - 2 k_2\pi) + 2 J_2 a_1 m_1 \cos(k_1\pi) + 2 a_1 a_2^2 m_1 m_2 \cos(k_1\pi)),\\
        \sigma_5 &=\ell_1 a_2^2 g m_2^2 \cos(k_1\pi - 2 k_2\pi), \\
        \sigma_6 &=\ell_1 a_2 g m_2 \cos(2 k_1\pi - k_2\pi) (\ell_1 m_2 + a_1 m_1), \\
        \sigma_7 &=2 \ell_1^2 a_2^3 g m_2^3 \cos(k_1\pi - k_2\pi) \sin(k_1\pi - k_2\pi), \\
        \sigma_8 &=(J_1 \sin(k_2\pi) + \ell_1^2 m_2 \sin(k_2\pi) + a_1^2 m_1 \sin(k_2\pi)\\
            &\quad - \ell_1^2 m_2 \cos(k_1\pi - k_2\pi) \sin(k_1\pi) - \ell_1 a_1 m_1 \cos(k_1\pi - k_2\pi) \sin(k_1\pi)), \\
        \sigma_9 &=a_2 g m_2 (J_1 \cos(k_2\pi) + \ell_1^2 m_2 \cos(k_2\pi) + a_1^2 m_1 \cos(k_2\pi) \\
            &\quad - \ell_1^2 m_2 \sin(k_1\pi - k_2\pi) \sin(k_1\pi) - \ell_1 a_1 m_1 \sin(k_1\pi - k_2\pi) \sin(k_1\pi)).
    \end{split}
\end{equation}
Due to the block structure of~\eqref{DPLinMat}, the square of its eigenvalues are equal to those of the lower left $2\times 2$ matrix 
\begin{equation}\label{MiniMatEDP}
	\begin{bmatrix}
		\frac{\sigma_2\sigma_3}{\sigma_1^2}-\frac{\sigma_4}{2\sigma_1} & \frac{\sigma_5}{\sigma_1}-\frac{\sigma_2\sigma_3}{\sigma_1^2} \\
		\frac{\sigma_6}{\sigma_1}-\frac{\sigma_7\sigma_8}{\sigma_1^2} & -\frac{\sigma_{9}}{\sigma_1}-\frac{\sigma_7\sigma_8}{\sigma_1^2}
	\end{bmatrix}.
\end{equation}
We will now proceed to classify the stability of the four equilibria in~\eqref{DPsteady} using the matrix~\eqref{MiniMatEDP}.\\

\noindent{\bf Down-Down.} The Down-Down state is obtained by setting $k_1 = k_2 = 0$. In this case~\eqref{MiniMatEDP} becomes
\begin{equation}
	\begin{bmatrix}
		-\frac{g(m_2a_2^2+J_2)(\ell_1m_2+a_1m_1)}{J_2m_2\ell_1^2+m_1m_2a_1^2a_2^2+J_2m_1a_1^2+J_1m_2a_2^2+J_1J_2} & \frac{\ell_1a_2^2gm_2^2}{J_2m_2\ell_1^2+m_1m_2a_1^2a_2^2+J_2m_1a_1^2+J_1m_2a_2^2+J_1J_2} \\
		\frac{\ell_1a_2gm_2(\ell_1m_2+a_1m_1)}{J_2m_2\ell_1^2+m_1m_2a_1^2a_2^2+J_2m_1a_1^2+J_1m_2a_2^2+J_1J_2} & -\frac{a_2gm_2(m_2\ell_1^2+m_1a_1^2+J_1)}{J_2m_2\ell_1^2+m_1m_2a_1^2a_2^2+J_2m_1a_1^2+J_1m_2a_2^2+J_1J_2}
	\end{bmatrix}.
\end{equation} 
Since the trace of the above matrix is negative and the determinant is positive for all the physical parameters $m_1$,$m_2$,$a_1$,$a_2$,$J_1$,$J_2$,$g$,$\ell_1$,$\ell_2$, it follows that both eigenvalues are distinct and strictly negative. Therefore, the linearization~\eqref{DPLinMat} about the Down-Down state has two pairs of purely complex eigenvalues, making it a linear center. \\

\noindent{\bf Down-Up.} The Down-Up equilibrium is obtained by setting $k_1 = 0$ and $k_2 = 1$. The matrix~\eqref{MiniMatEDP} becomes 
\begin{equation}
	\begin{bmatrix}
		-\frac{g(m_2a_2^2+J_2)(\ell_1m_2+a_1m_1)}{J_2m_2\ell_1^2+m_1m_2a_1^2a_2^2+J_2m_1a_1^2+J_1m_2a_2^2+J_1J_2} & \frac{\ell_1a_2^2gm_2^2}{J_2m_2\ell_1^2+m_1m_2a_1^2a_2^2+J_2m_1a_1^2+J_1m_2a_2^2+J_1J_2} \\
		-\frac{\ell_1a_2gm_2(\ell_1m_2+a_1m_1)}{J_2m_2\ell_1^2+m_1m_2a_1^2a_2^2+J_2m_1a_1^2+J_1m_2a_2^2+J_1J_2} & \frac{a_2gm_2(m_2\ell_1^2+m_1a_1^2+J_1)}{J_2m_2\ell_1^2+m_1m_2a_1^2a_2^2+J_2m_1a_1^2+J_1m_2a_2^2+J_1J_2}
	\end{bmatrix}.
\end{equation} 
which has a negative determinant for all relevant parameter values. Therefore, the above matrix has one positive and one negative real eigenvalue along with a pair of purely imaginary eigenvalues. Hence, the Down-Up equilibrium is an index-1 saddle.  \\

\noindent{\bf Up-Down.} We obtain the Up-Down equilibrium by setting $k_1 = 1$ and $k_2 = 0$. The matrix~\eqref{MiniMatEDP} is then given by
\begin{equation}
	\begin{bmatrix}
		\frac{g(m_2a_2^2+J_2)(\ell_1m_2+a_1m_1)}{J_2m_2\ell_1^2+m_1m_2a_1^2a_2^2+J_2m_1a_1^2+J_1m_2a_2^2+J_1J_2} & -\frac{\ell_1a_2^2gm_2^2}{J_2m_2\ell_1^2+m_1m_2a_1^2a_2^2+J_2m_1a_1^2+J_1m_2a_2^2+J_1J_2} \\
		\frac{\ell_1a_2gm_2(\ell_1m_2+a_1m_1)}{J_2m_2\ell_1^2+m_1m_2a_1^2a_2^2+J_2m_1a_1^2+J_1m_2a_2^2+J_1J_2} & -\frac{a_2gm_2(m_2\ell_1^2+m_1a_1^2+J_1)}{J_2m_2\ell_1^2+m_1m_2a_1^2a_2^2+J_2m_1a_1^2+J_1m_2a_2^2+J_1J_2}
	\end{bmatrix}.
\end{equation} 
and following as in the Down-Up equilibrium, we find that the Up-Down equilibrium is also an index-1 saddle. \\

\noindent{\bf Up-Up.} The Up-Up state has $k_1 = k_2 = 1$, resulting in~\eqref{MiniMatEDP} taking the form
\begin{equation}
	\begin{bmatrix}
		\frac{g(m_2a_2^2+J_2)(\ell_1m_2+a_1m_1)}{J_2m_2\ell_1^2+m_1m_2a_1^2a_2^2+J_2m_1a_1^2+J_1m_2a_2^2+J_1J_2} & -\frac{\ell_1a_2^2gm_2^2}{J_2m_2\ell_1^2+m_1m_2a_1^2a_2^2+J_2m_1a_1^2+J_1m_2a_2^2+J_1J_2} \\
		-\frac{\ell_1a_2gm_2(\ell_1m_2+a_1m_1)}{J_2m_2\ell_1^2+m_1m_2a_1^2a_2^2+J_2m_1a_1^2+J_1m_2a_2^2+J_1J_2} & \frac{a_2gm_2(m_2\ell_1^2+m_1a_1^2+J_1)}{J_2m_2\ell_1^2+m_1m_2a_1^2a_2^2+J_2m_1a_1^2+J_1m_2a_2^2+J_1J_2}
	\end{bmatrix}.
\end{equation} 
The above matrix is the result of negating all of the entries of the Down-Down analysis. It follows that the linearization~\eqref{DPLinMat} about the Up-Up state has two positive eigenvalues and two negative eigenvalues. From the nomenclature above, the Up-Up state is an {\em index-2 saddle}. \\

\begin{figure}
    \centering
    \includegraphics[height=0.55\textwidth]{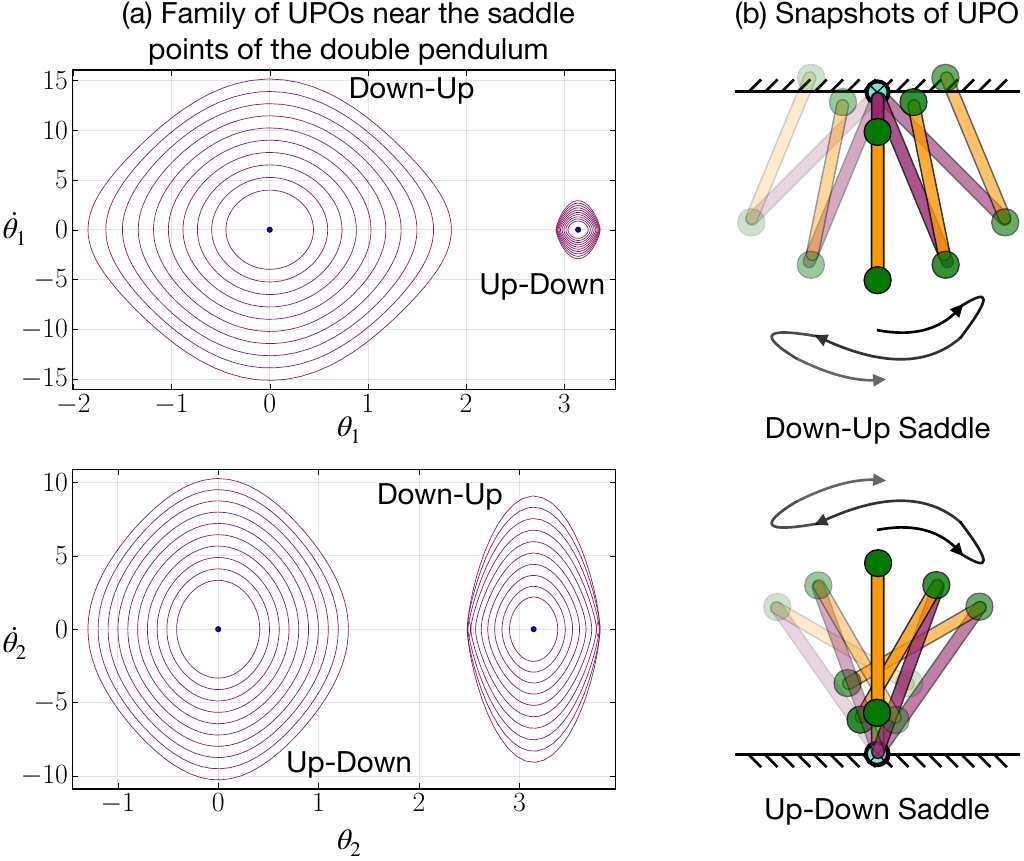}
    \caption{Near the index-1 saddles of the double pendulum there are infinitely many UPOs. These UPOs are projected into (a) the $(\theta_1,\dot\theta_1)$-plane and the $(\theta_2,\dot\theta_2)$-plane with a cartoon of their physical motion in the double pendulum shown in (b).}
    \label{fig:DP_UPO}
\end{figure}

From the above we find that the Down-Up and Up-Down equilibria are index-1 saddles. Importantly, our work in Section~\ref{sec:Background} gives that in a neighborhood of these index-1 saddles there exist infinitely many UPOs. These UPOs are analogous to the Lyapunov orbits of the $L_1$ and $L_2$ Lagrange points of the PCR3BP. However, in the case of the double pendulum these UPOs may be easier to visualize. Precisely, the `Down' arm of these index-1 saddles should be understood as being stable, while the `Up' arm is unstable, both due to the force of gravity. Therefore, the UPOs near these Down-Up and Up-Down saddles have the `Up' arm undergoing a slight wobble, while the `Down' arm exhibits little variation since it is stable. This is demonstrated in Figure~\ref{fig:DP_UPO} where one can see a collection of these UPOs projected into both the $(\theta_1,\dot\theta_1)$ and $(\theta_2,\dot\theta_2)$ planes. In the following subsections we will work to numerically identify homoclinic and heteroclinic trajectories associated to these UPOs near the index-1 saddles of the double pendulum.

\subsection{Tube Dynamics Near the Down-Up State}\label{sec:L1Hom}

We begin by describing our process of identifying homoclinic trajectories to the UPOs near the Down-Up equilibrium of the double pendulum. We emphasize that although we refer to these orbits throughout as 'homoclinic' they are actually heteroclinic orbits in phase space. This potential source of confusion comes from the periodicity of the $\theta_1$ and $\theta_2$ components in~\eqref{eq:edp_ode}. Precisely, a trajectory that connections the UPOs near the Down-Up equilibrium $(\theta_1,\theta_2,\omega_1,\omega_2) = (0,\pi,0,0)$ to those near another Down-Up equilibrium $(\theta_1,\theta_2,\omega_1,\omega_2) = (0,\pi\pm 2\pi,0,0)$ comes as a heteroclinic orbit in phase space if one does not quotient by the periodicity of the $\theta_1$ and $\theta_2$ components, but in physical space such a connection appears to return to where it started while having the second arm undergoing a full clockwise or counterclockwise rotation. As the motion of the double pendulum is best understood in physical space, we will hereby refer to trajectories that connect UPOs near any of the Down-Up equilibriums, $(0,\pi \pm 2\pi k,0,0)$, $k \in \mathbb{Z}$, as homoclinic. Finally, we comment that our numerical investigations have revealed that true homoclinic trajectories that asymptotically approach one of the UPOs near $(0,\pi,0,0)$ do not exist, meaning that only 'physical' homoclinic trajectories described above can exist.

We can numerically obtain the UPOs near the Down-Up equilibrium by searching for symmetric periodic solutions with the value of the Hamiltonian~\eqref{eq:edp_hamiltonian} slightly above that of the Down-Up equilibrium. The search for these periodic solutions can be posed as a root-finding problem constrained by the fact that we must remain within a Hamiltonian level set everywhere on the trajectory. This process is identical to that used in~\cite{RossBook} to identify homoclinic trajectories in the PCR3BP and to implement this process numerically we use the Julia package $Zygote$~\cite{liao2019differentiable}. A demonstration of these techniques is included in the repository associated with this manuscript. With these UPOs, as displayed in Figure~\ref{fig:DP_UPO}, we then use their Floquet spectrum to locally approximate the stable and unstable directions associated to the UPO. By simulating these approximate stable and unstable directions backward and forward in time, respectively, according to the ODE~\eqref{eq:edp_ode} we obtain accurate representations of the stable and unstable manifolds, i.e. the tubes, associated to each UPO.

\begin{figure}
    \centering
    \includegraphics[width=0.8\textwidth]{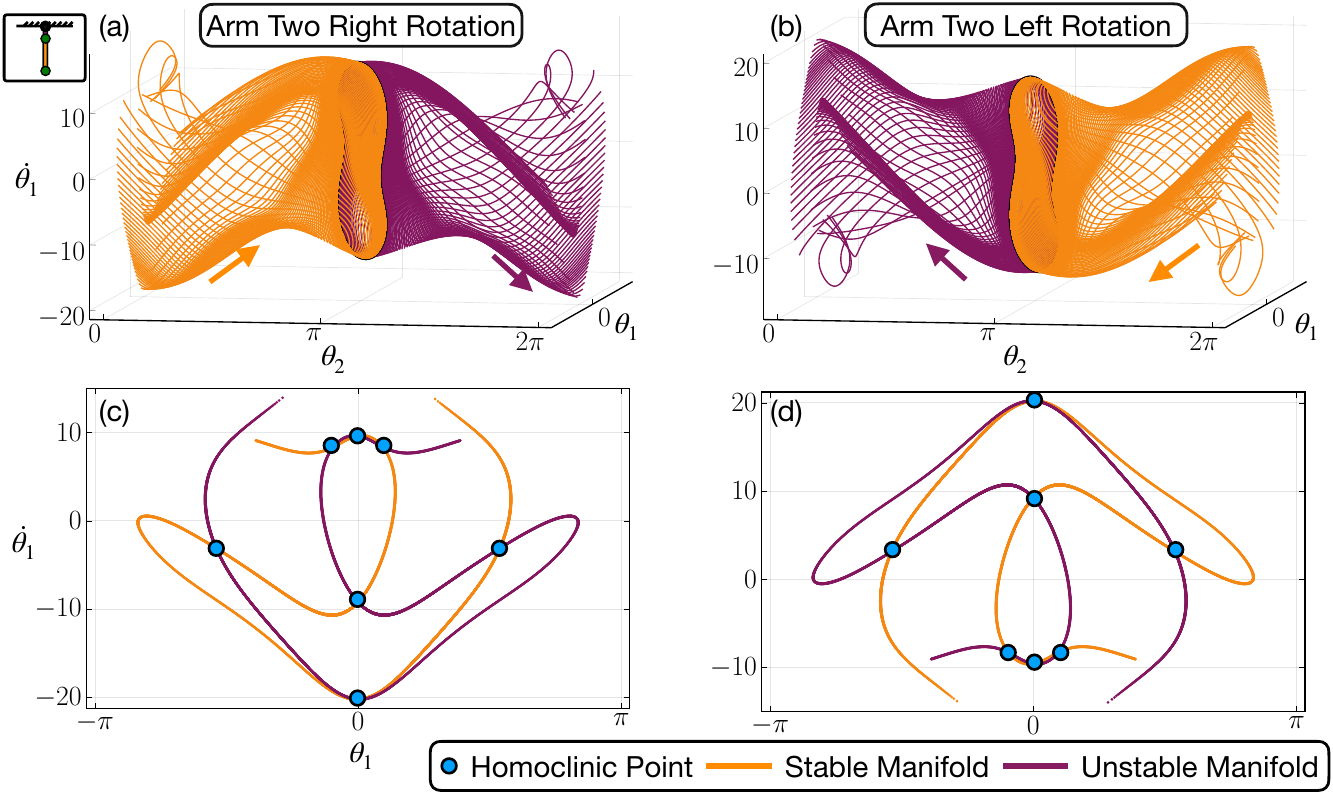}
    \caption{(a) The clockwise rotating portion of the tube structure of the stable and unstable manifolds associated to a UPO near the Down-Up equilibrium with Hamiltonian value $\mathcal{H} = 0.2$. In (b) we show the counterclockwise portion of the tube structure, related to (a) by applying the reverser~\eqref{Reverser}. In (c) and (d) we present the intersection of the tubes with the Poincar\'e section $\theta_2 = 2\pi k$, $k \in \mathbb{Z}$. Intersections of the stable and unstable manifolds in the Poincar\'e section represent 'physical' homoclinic trajectories that connect the UPOs near the Down-Up equilibrium.}
    \label{fig:EDP_TubeStructure_L1}
\end{figure}

To identify homoclinic trajectories we flow the unstable manifold forward in time until $\theta_2 = 2\pi$. Similarly, we flow the stable manifold backward in time until $\theta_2 = 0$. Notice that upon quotienting for the periodicity in $\theta_2$ both the stable and unstable manifolds have been simulated to the same region in physical space, given by the second arm hanging straight down. We use the hyperplane $\theta_2 = 2\pi k$ with $k \in \mathbb{Z}$ as our Poincar\'e section to identify homoclinic trajectories. We further note that the reversible symmetry of~\eqref{eq:edp_ode} guarantees that we could equivalently flow the unstable manifold forward in the direction that decreases the value of $\theta_2$ until it reaches $\theta_2 = 0$ to meet the Poincar\'e section, and similarly for the stable manifold meeting $\theta_2 = 2\pi$. The tube structure of these stable and unstable manifolds are given in Figure~\ref{fig:EDP_TubeStructure_L1}(a) and their reversible counterparts are shown in Figure~\ref{fig:EDP_TubeStructure_L1}(b) for $\mathcal{H} = 0.2$, above the value of the Down-Up equilibrium at $\mathcal{H} = -0.1754$. Having the unstable manifold meet $\theta_2 = 2\pi$ represents the second arm moving clockwise, while having the unstable manifold meet $\theta_2 = 0$ represents the second arm moving counterclockwise.

We can identify `physical' homoclinic trajectories near the Down-Up equilibrium of the double pendulum using the Poincar\'e section $\theta_2 = 2\pi k$, $k \in \mathbb{Z}$. Indeed, Figure~\ref{fig:EDP_TubeStructure_L1}(c) and (d) plot the $(\theta_1,\dot\theta_1)$ plane corresponding to the tube structures presented in (a) and (b), respectively, of the same figure. Intersections of the stable and unstable manifolds in the Poincar\'e section represent the homoclinic orbits that are the focus of this section. Such intersections are homoclinic trajectories since $\theta_1,\theta_2$, and $\dot\theta_1$ are equal at these points, and it can be verified numerically that the Hamiltonian structure of system~\eqref{eq:edp_ode} confines the value of $\dot\theta_2$ to be the same here as well. As one can see, for this value of the Hamiltonian there are numerous intersections. Those with $\theta_1 = 0$ represent symmetric, or reversible, homoclinic orbits while those intersections with $\theta_1 \neq 0$ represent the asymmetric homoclinic orbits. As discussed in \S~\ref{sec:Symmetry}, the asymmetric orbits come in pairs, represented in the Poincar\'e section by the symmetry over the $\theta_1 = 0$ line. From our choice of the Poincar\'e section, all such homoclinic trajectories represent the double pendulum starting near the Down-Up equilibrium and having the second arm completing a full rotation before returning to a neighborhood of the Down-Up equilibrium.\footnote{To visualize the behavior of those homoclinic orbits presented in Fig.~\ref{fig:EDP_TubeStructure_L1} in physical space see \href{https://github.com/dynamicslab/Saddle-Mediated-Transport-of-Double-Pendulum}{https://github.com/dynamicslab/Saddle-Mediated-Transport-of-Double-Pendulum}}.

\begin{figure}[t]
    \centering
    \includegraphics[width=0.8\textwidth]{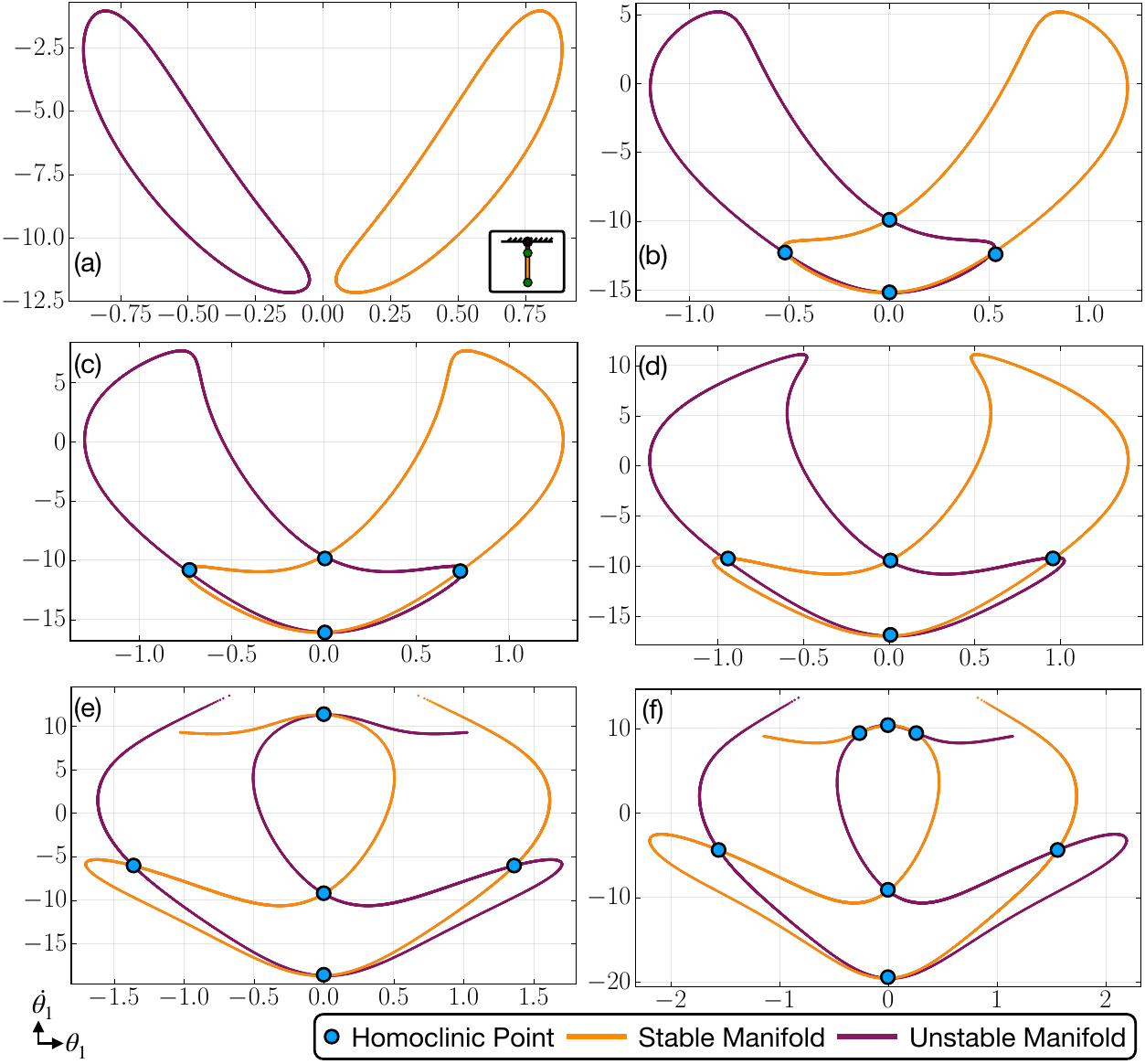}
    \caption{Increasing the energy of the double pendulum up from the energy of the Down-Up steady-state results in an increasing number of homoclinic trajectories of the UPOs near the Down-Up equilibrium. Energy values are (a) $\mathcal{H} = -0.147$, (b) $\mathcal{H} = -0.07$, (c) $\mathcal{H} = -0.034$, (d) $\mathcal{H} = 0.06$, (e) $\mathcal{H} = 0.102$, and (f) $\mathcal{H} = 0.157$. The Down-Up equilibrium has energy $\mathcal{H} = -0.1754$.}
    \label{fig:EDP_PoincareCut_L1_VariesEng}
\end{figure}

Figure~\ref{fig:EDP_PoincareCut_L1_VariesEng} shows that increasing the Hamiltonian energy~\eqref{eq:edp_hamiltonian}, from that of the Down-Up equilibrium,  results in a steadily increasing number of homoclinic trajectories. These homoclinic orbits arise via saddle-node and symmetry-breaking pitchfork bifurcations as the energy is increased, coming from further manifold intersections. For relatively large values of the energy, as in panels (e) and (f), flowing some parts of stable and unstable manifolds towards the Poincar\'e section takes a long time. The result is that some of our intersections of the stable and unstable manifolds with the Poincar\'e section do not present themselves as closed loops. We expect that simulating the tubes for a significantly longer time will complete their intersections with the Poincar\'e section, but this would require significant computational time and is not necessary for our exploration here.

\begin{figure}[t]
    \centering
    \includegraphics[width=0.8\textwidth]{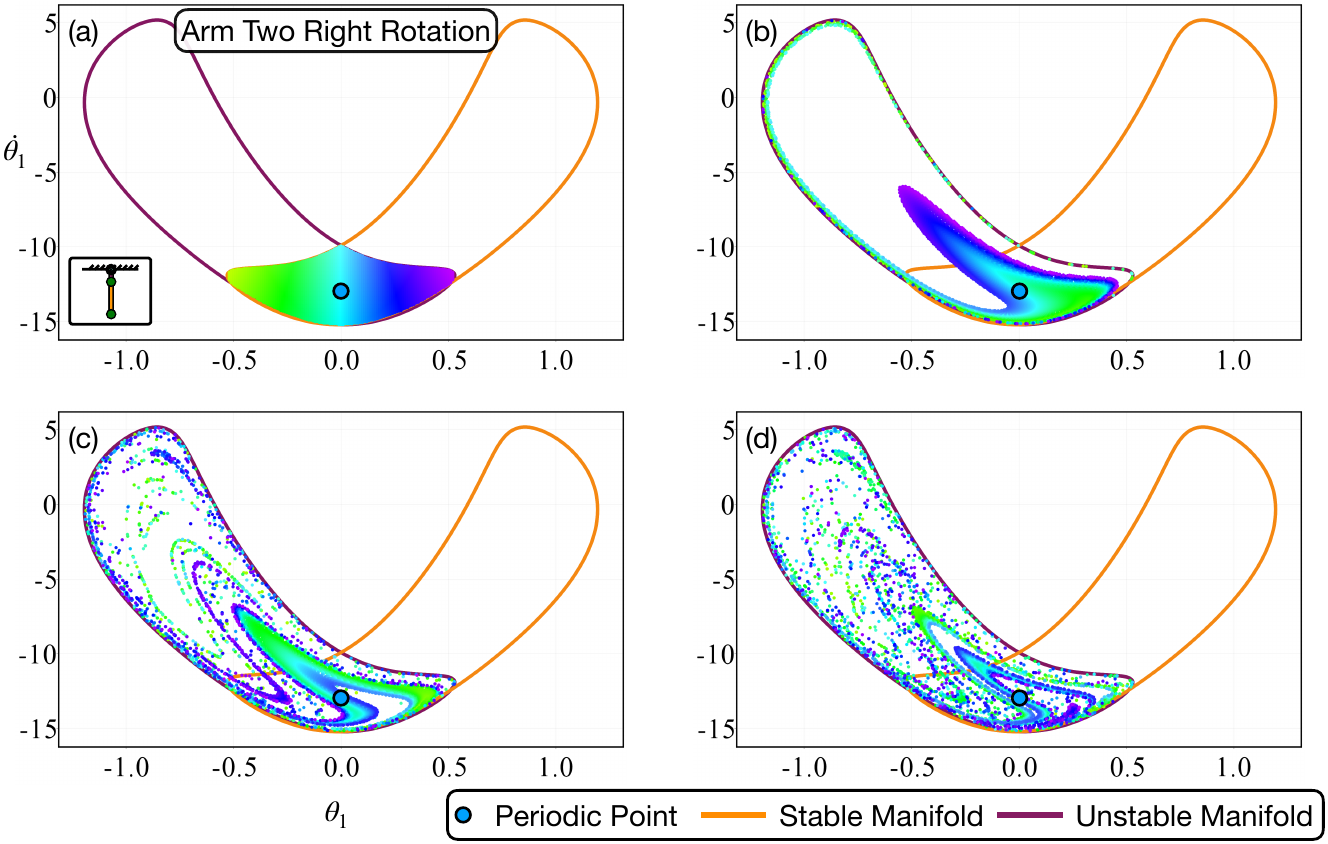}
    \caption{Trajectories that start inside of the unstable tube will remain inside of it for all forward time. This is demonstrated by flowing numerous points shown in (a) which lie inside both the stable and unstable tubes of a UPO near the Down-Up equilibrium with energy $\mathcal{H}=-0.06985$ forward until they intersect the Poincar\'e section again at (b) $\theta_2 = 4\pi$, (c) $\theta_2 = 6\pi$, and (d) $\theta_2 = 8\pi$. Flowing these same points backward in time produces results that are mirrored over $\theta_1 = 0$, likely resulting in a fractal set that remains inside both the stable and unstable tubes for all forward and backward time. The location of a symmetric periodic orbit that remains inside both tubes for all times is also plotted.}
    \label{fig:EDP_L1_ColorWheel_LowEng}
\end{figure}

The fact that these stable and unstable tubes intersect means that the system exhibits macroscopic transport mechanisms that allow one to traverse large distances in phase space, eventually returning to the region where it started. Precisely, trajectories that are initialized inside of an unstable tube will remain inside it for all forward time, thus repeatedly coming back to the Poincar\'e section $\theta_2 = 2\pi k$ and intersecting within the closed curve created by the intersection with the unstable tube. This is demonstrated in Figure~\ref{fig:EDP_L1_ColorWheel_LowEng} where we simulate numerous trajectories with initial conditions from the Poincar\'e section that lie in the intersection of the stable and unstable tubes for energy $\mathcal{H}=-0.06985$. After each pass from the Poincar\'e section up to the UPO near the Down-Up equilibrium and back, one sees that these trajectories intersect the Poincar\'e section inside the closed curve created by the unstable tube, while spreading out over the interior of the tube. One can achieve a similar phenomenon by simulating these initial conditions backward in time, forcing them to remain inside of the stable tube forever. Visualizing this can be achieved by simply mirroring the forward iterates in Figure~\ref{fig:EDP_L1_ColorWheel_LowEng} over $\theta_1 = 0$. Thus, it appears that those trajectories that remain inside both tubes for all forward and backward time become a fractal set, which we suspect forms a horseshoe set~\cite{Smale}. In fact, our theoretical results below in Theorem~\ref{thm:Main} demonstrate that this is indeed the case in a neighborhood of each homoclinic orbit, while Figure~\ref{fig:EDP_L1_ColorWheel_LowEng} provides evidence for a more global horseshoe structure.     

This apparent horseshoe structure means that there exist a plethora of bounded solutions that return infinitely often to a neighborhood of the UPO near the Down-Up equilibrium. As an example, we have identified a simple symmetric periodic orbit that in physical space has the second arm of the double pendulum completing a full $2\pi$ rotation over each period. 
The intersection of this periodic orbit with the Poincar\'e section is plotted in Figure~\ref{fig:EDP_L1_ColorWheel_LowEng}. Interestingly, it lies almost perfectly between the location of the two symmetric homoclinic orbits that form the boundary of the tubes restricted to $\theta_1 = 0$. Although we do not have an explanation for this, we do not believe it to be a coincidence and expect it to be related to the precise horseshoe structure generated by the intersection of the interiors of the stable and unstable tubes.

\subsection{Tube Dynamics Near the Up-Down State}\label{sec:L2Hom}

The process of identifying homoclinic trajectories to the UPOs near the Up-Down equilibrium is similar to what was done in the previous subsection. Again, we were not able to identify homoclinic trajectories in phase space but have sought to identify the `physical' homoclinic trajectories that connect UPOs near the Up-Down equilibrium after accounting for the periodicity of the $\theta_1$ and $\theta_2$ variables. For the UPOs near the Up-Down equilibrium we find that these physical homoclinic trajectories have the first arm making a full clockwise or counterclockwise rotation before returning to where they started. In phase space this means that $\theta_1$ increases or decreases by exactly $2\pi$ over the course of the homoclinic trajectory.

\begin{figure}
    \centering
    \includegraphics[width=0.85\textwidth]{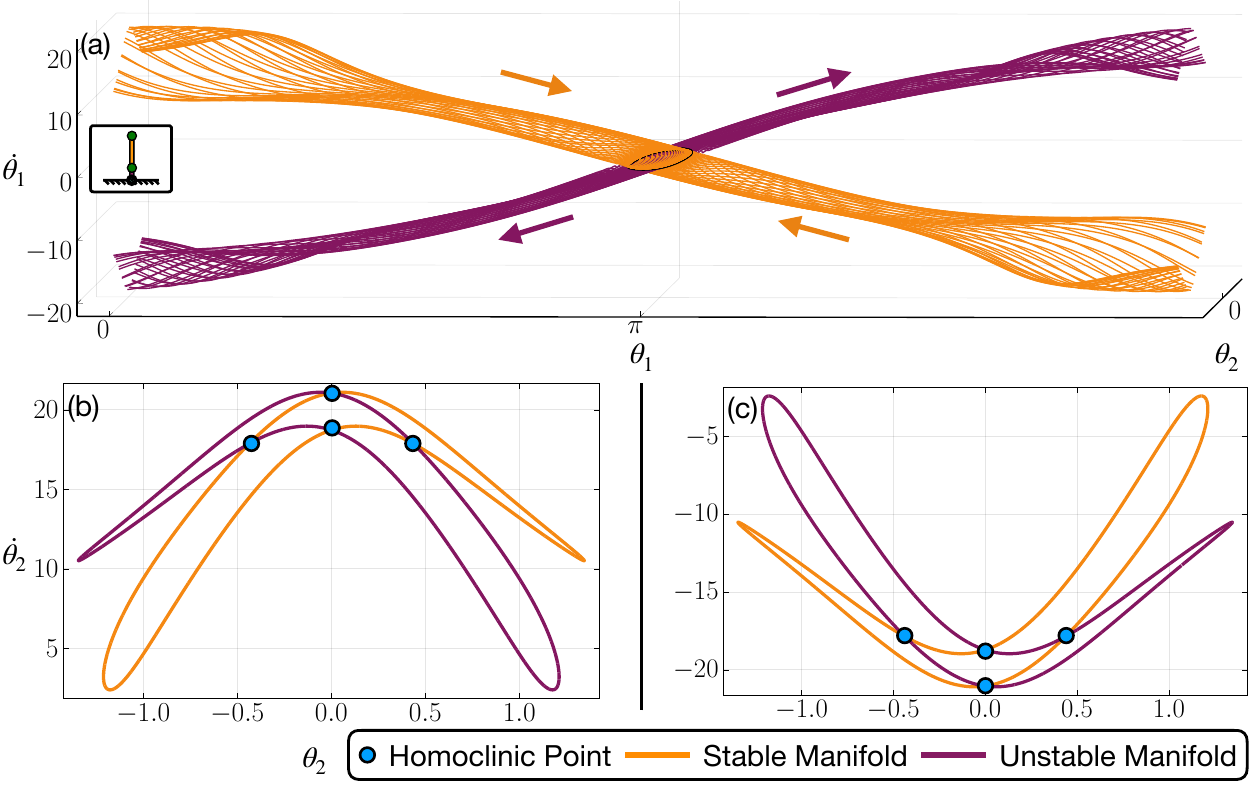}
    \caption{Panel (a) presents the stable and unstable tubes of a UPO near the Up-Down steady-state in the variables $(\theta_1,\theta_2,\dot\theta_1)$ with energy $\mathcal{H} = 0.2$. In (b) and (c) we present the intersection of the tubes with the Poincar\'e section $\theta_1 = 2\pi k$, $k \in \mathbb{Z}$. Intersections of the stable and unstable manifolds in the Poincar\'e section represent 'physical' homoclinic trajectories that connect the UPOs near the Up-Down state.}
    \label{fig:EDP_TubeStructure_L2}
\end{figure}

We can identify the UPOs near the Up-Down equilibrium and use their Floquet spectrum to follow the stable and unstable tubes backward and forward in time, respectively. We flow these stable and unstable tubes associated to these UPOs according to~\eqref{eq:edp_ode} until they meet the Poincar\'e section $\theta_1 = 2\pi k$, $k \in \mathbb{Z}$. In Figure~\ref{fig:EDP_TubeStructure_L2}(a) we present the stable and unstable manifolds of a UPO near the Up-Down equilibrium with energy $\mathcal{H} = 0.2$. The resulting intersections with the Poincar\'e section are further plotted in panels (b) and (c), for which the former represents clockwise rotations of the first arm, while the latter represents counterclockwise rotations. Again we are able to identify multiple symmetric and asymmetric homoclinic orbits that enable macroscopic transport throughout the phase space of the double pendulum\footnote{To visualize the behavior of those homoclinic orbits presented in Fig.~\ref{fig:EDP_TubeStructure_L2} in physical space see \href{https://github.com/dynamicslab/Saddle-Mediated-Transport-of-Double-Pendulum}{https://github.com/dynamicslab/Saddle-Mediated-Transport-of-Double-Pendulum}.}. As with the UPOs near the Down-Up equilibrium, we find that increasing the energy above that of the Up-Down equilibrium results in sequences of saddle-node and symmetry-breaking pitchfork bifurcations to more and more homoclinic orbits. We do not present a figure with these results for brevity.  

\begin{figure}
    \centering
    \includegraphics[width=0.8\textwidth]{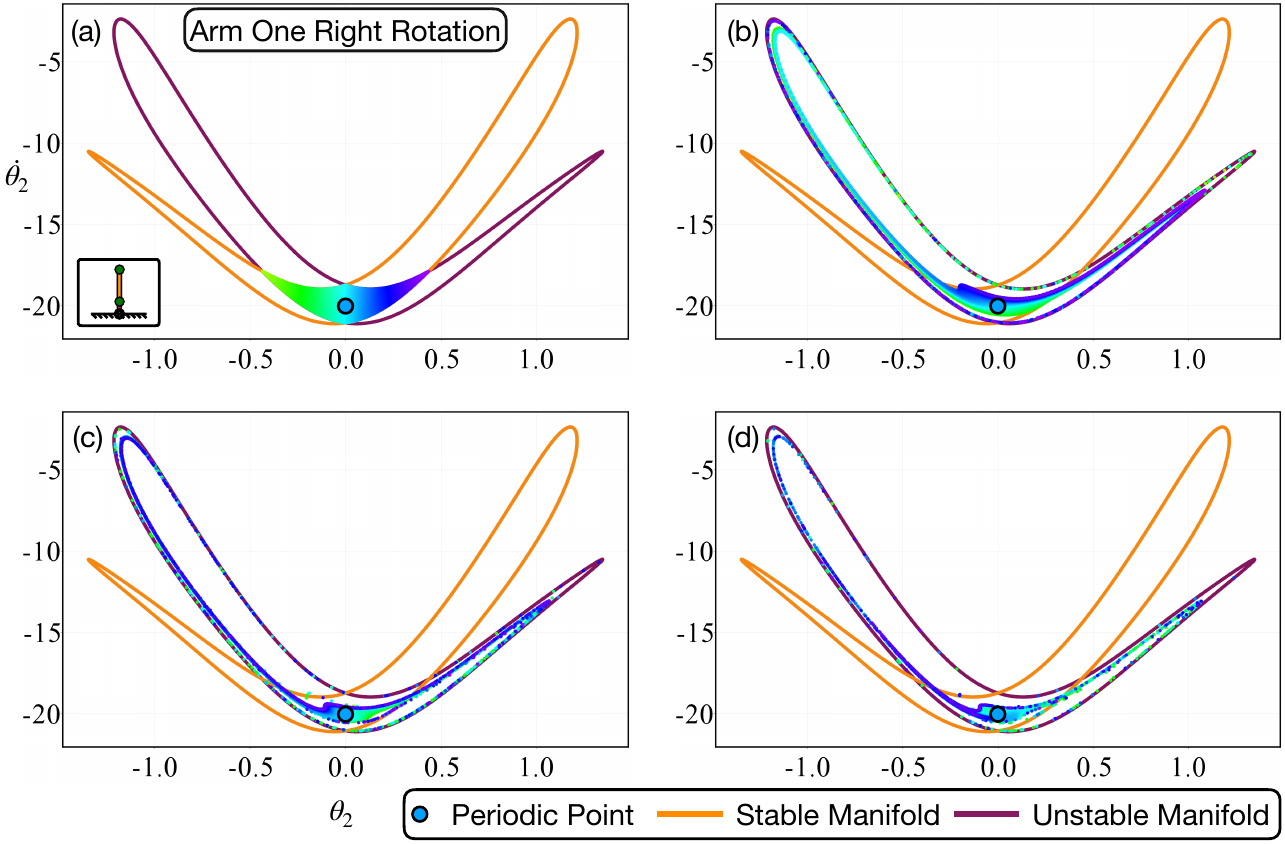}
    \caption{The apparent fractal structure generated by those trajectories that remain inside of the unstable tube of a UPO near the Up-Down equilibrium with energy $\mathcal{H} = 0.2$. As in Figure~\ref{fig:EDP_L1_ColorWheel_LowEng} we flow numerous initial conditions shown in (a) forward in time until they intersect the Poincar\'e section for a (b) first, (c) second, and (d) third time. The location of a symmetric periodic orbit that remains inside both tubes for all times is also plotted.}
    \label{fig:EDP_L2_ColorWheel}
\end{figure}

As with the tubes associated to UPOs near the Down-Up equilibrium, we find a fractal/horseshoe structure of trajectories that remain inside the tubes associated to UPOs near the Up-Down equilibrium for all forward and backward time. This structure is illustrated in Figure~\ref{fig:EDP_L2_ColorWheel} for energy $\mathcal{H} = 0.2$, which is analogous to Figure~\ref{fig:EDP_L1_ColorWheel_LowEng} presented previously. We again exploit this complex structure to identify an exemplary symmetric periodic orbit that represents the first arm making a full $2\pi$ rotation in physical space. These symmetric orbits intersect the Poincar\'e section at $\theta_2 = 0$ and again we find this point of intersection lies almost perfectly in the center of the intersections of two symmetric homoclinic orbits.

\subsection{Heteroclinic Transitions} \label{sec:Heteroclinic}

The approach to identifying heteroclinic connections between neighborhoods of the Down-Up and Up-Down equilibria is similar to that outlined above for identifying homoclinic trajectories. We fix the energy and flow the unstable manifold of the a UPO near the Down-Up equilibrium forward in time until it reaches the Poincar\'e section $\theta_1 = \theta_2$ in phase space. We then flow the stable manifold of a UPO with the same energy near the Up-Down equilibrium backwards in time until it also reaches the Poincar\'e section $\theta_1 = \theta_2$. Intersections of these stable and unstable manifolds in the Poincar\'e section therefore represent the desired heteroclinic orbits that transport one through phase space between neighborhoods of index-1 saddles. The process we have just described allows one to obtain orbits from near the Down-Up equilibrium to near the Up-Down state, although applying the reverser~\eqref{Reverser} reverses the direction of these orbits and therefore gives orbits that originate near the Up-Down equilibrium and move to a neighborhood of the Down-Up equilibrium as well.

\begin{figure}[t]
    \centering
    \includegraphics[width=0.95\textwidth]{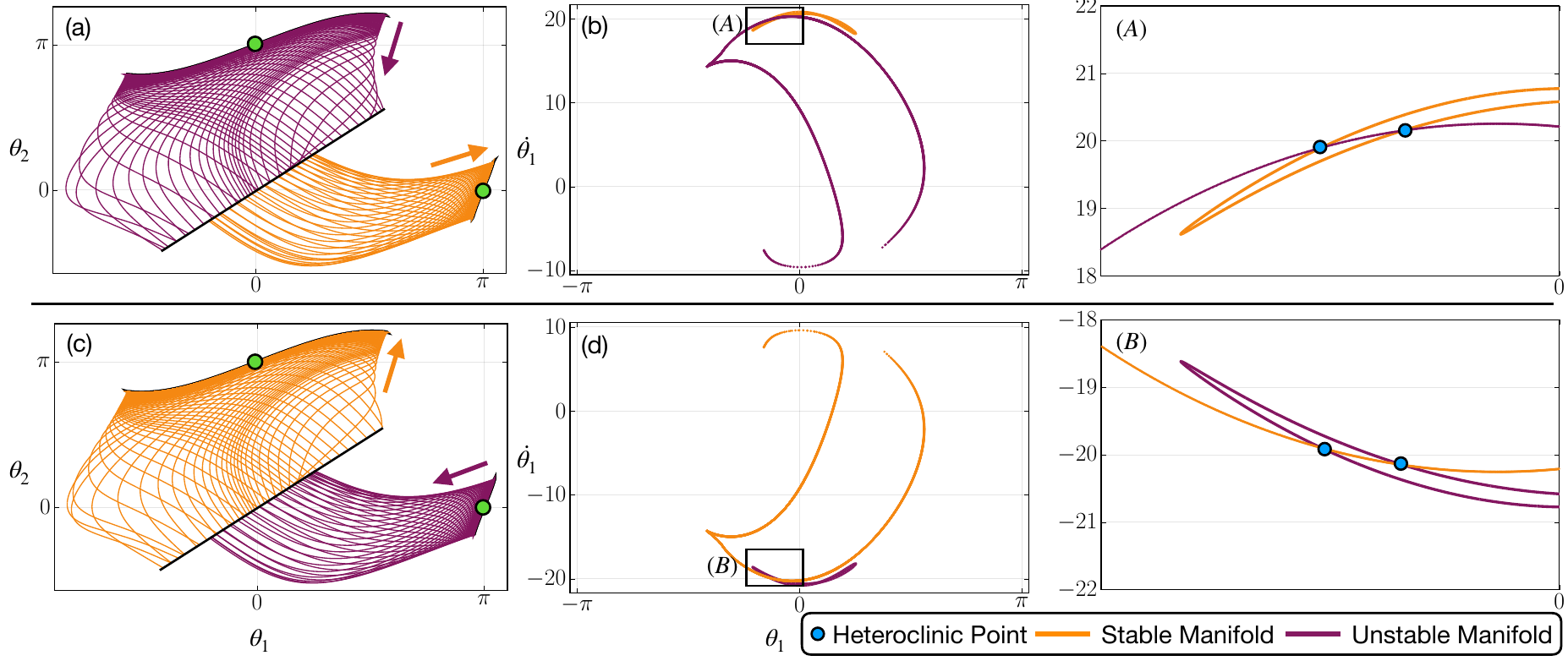}
    \caption{Heteroclinic connections between UPOs in the neighborhoods of the Down-Up and Up-Down equilibria are found by inspecting where their tubes meet the Poincar\'e section $\theta_1 = \theta_2$. (a) and (c) show the tubes flowing towards the section in forward and backward time for  $\mathcal{H}=0.2$. Panels (b) and (d) demonstrate the intersection of the tubes with the Poincar\'e section and panels (A) and (B) provide a zoom of the intersection of these tubes, representing heteroclinic trajectories.}
    \label{fig:EDP_HeteroclinicPoints}
\end{figure}

Our results are illustrated in Figure~\ref{fig:EDP_HeteroclinicPoints}. Panel (a) demonstrates the unstable manifold of a UPO near the Down-Up equilibrium moving toward the Poincar\'e section, while panel (c) demonstrates the time reversed flow of the stable manifold of the same UPO moving backward in time toward the Poincar\'e section. Panels (b) and (d) show the intersection of the UPO tubes with the Poincar\'e section for panels (a) and (c), respectively. We further present a zoom-in of the Poincar\'e section near the intersections, illustrating two distinct heteroclinic connections. These heteroclinic orbits are plotted in $(\theta_1,\theta_2,\dot\theta_1)$ space in Figure~\ref{fig:EDP_HeteroclinicOrbits3D} with their asymptotic UPOs shown in solid black. Our numerical experiments have shown that such heteroclinic connections exist for $\mathcal{H} \geq 0.1754$. Like the homoclinic orbits of the previous subsections, the existence of these heteroclinic orbits demonstrates a transport mechanism that allows one to move between neighbourhoods of the index-1 saddles in phase space. 

\begin{figure}[t]
    \centering
    \includegraphics[width=0.65\textwidth]{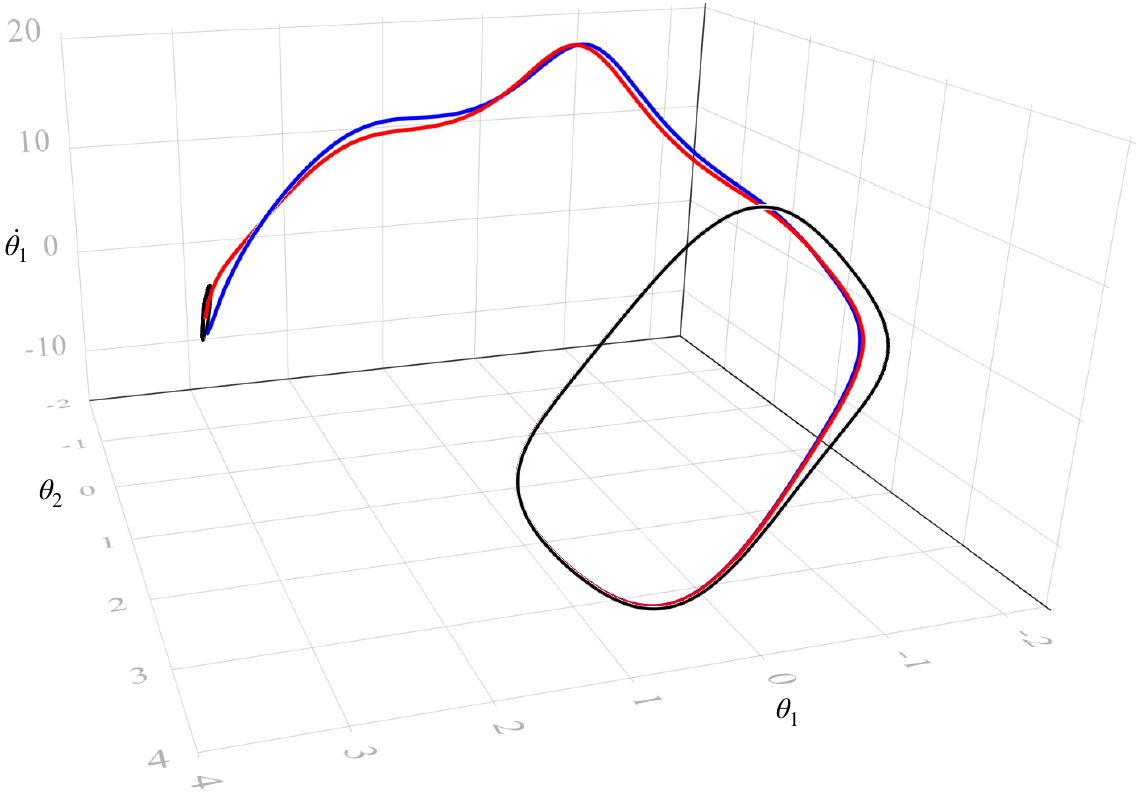}
    \caption{A visualization in $(\theta_1,\theta_2,\dot\theta_1)$ space of the two distinct heteroclinic orbits (red and blue) found as intersections in the Poincar\'e section $\theta_1 = \theta_2$ in Figure~\ref{fig:EDP_HeteroclinicPoints}. The black closed loops represent the asymptotic UPOs near each of the index-1 saddles.}
    \label{fig:EDP_HeteroclinicOrbits3D}
\end{figure}

\section{Existence of Long Periodic and Connecting Orbits}\label{sec:Theory}

Here we seek to leverage the numerical findings of the previous section to establish the existence of periodic and homo-/heteroclinic orbits of the double pendulum that transit over vast regions of phase space. Our main result, Theorem~\ref{thm:Main}, in~\S\ref{sec:Theorem} provides this existence for completely general four dimensional Hamilton systems. In~\S\ref{subsec:DPapplication} we will apply our result to the double pendulum, thus extending the numerical results of the previous section. The generality of this result means that it is applicable to a wide variety of Hamiltonian systems with two degrees of freedom. In fact, it provides similar results to the itinerary construction for the PCR3BP in~\cite{Koon} and we briefly comment on how our results extend the known structure of phase space for the PCR3BP in~\S\ref{subsec:TBPapplication}. Further, our result is general enough to apply to similar studies where homoclinic and heteroclinic orbits near index-1 saddles are known to exist without needing to know the specific structure of the system. Therefore, these results may be applied to chemical reaction models~\cite{Chemical,Chemical2,Chemical3}, models for ship motion~\cite{Ship1,Ship2,Naik}, snap-through buckling of shallow arches~\cite{Arch}, and the motion of a ball rolling on a saddle surface~\cite{Ball}.

\subsection{Abstract Setting and Results}\label{sec:Theorem}

Let us consider a general ODE 
\begin{equation}\label{ODE}
	\dot{u} = F(u),
\end{equation}
with $u \in \R^4$ and a smooth function $F:\R^4 \to \R^4$. To mimic the setting of the double pendulum, we first assume the existence of a conserved quantity, corresponding to a Hamiltonian of the system~\eqref{ODE}.

\begin{hyp}\label{hyp:Ham} 
There exists a smooth function $\mathcal{H}:\R^4\to\R$ such that $\langle F(u),\nabla \mathcal{H}(u) \rangle = 0$ for all $u \in \R^4$.
\end{hyp}

The existence of a conserved quantity allows us to reduce ourselves to its (generically) three-dimensional level sets. Next we assume that there exists a collection of UPOs that reside in the same level set of $\mathcal{H}$.

\begin{hyp}\label{hyp:Periodic} 
There exists $p \geq 1$ such that~\eqref{ODE} has periodic orbits $\gamma_1(t), \gamma_2(t), \dots, \gamma_p(t)$ satisfying 
\begin{equation}
	\mathcal{H}(\gamma_i(t)) = \mathcal{H}(\gamma_j(t))
\end{equation}
for all $1 \leq i,j \leq p$ and all $t \in \R$. Moreover, each $\gamma_j$ is hyperbolic and so has precisely two Floquet multipliers at one and no others on the unit circle.
\end{hyp}

Note that each periodic orbits has exactly two Floquet multipliers at one under Hypothesis~\ref{hyp:Ham}. Indeed, one multiplier comes from the Floquet exponent at zero related to translations around the periodic orbit, while the other comes from the fact that the conserved quantity reduces the flow of system~\eqref{ODE} to its level sets. Furthermore, the Hamiltonian structure forces the symmetry relationship that if $\lambda \in \mathbb{C}$ is a Floquet multiplier of some periodic orbit, then so is $\lambda^{-1},\bar \lambda,$ and $\bar \lambda^{-1}$. Therefore, we find that the two remaining Floquet multipliers that are not on the unit circle are real, taking the form $\lambda,\lambda^{-1} \in \mathbb{R}\setminus\{0\}$.  

From the assumption that each $\gamma_i(t)$ is hyperbolic, each periodic orbit has a stable and unstable manifold associated to it. As was discussed in Section~\ref{sec:Background}, these stable and unstable manifolds are diffeomorphic to cylinders inside of the level set $\mathcal{H}(\gamma_i(t))$. We denote these stable and unstable manifolds for each periodic orbit $\gamma_i$ by $W^s(\gamma_i)$ and $W^u(\gamma_i)$, respectively. Then we define 
\begin{equation}
	\W := \bigcup_{1 \leq i,j\leq p} W^s(\gamma_i) \cap W^u(\gamma_j)
\end{equation}
to be the set of all homoclinic and heteroclinic trajectories of~\eqref{ODE}. Of course, for a given system it is nearly impossible to know the full structure of the set $\W$, but our main result shows that we can use a finite collection of elements in $\W$ to construct infinitely many more. The following hypothesis makes our assumptions precise.

\begin{hyp}\label{hyp:HomHet} 
The set $\W$ is nonempty and the subset $\W_0 \subset \W$ is comprised of finitely many transverse homoclinic and heteroclinic orbits of~\eqref{ODE} belonging to $\W$.
\end{hyp}

The set $\W_0$ comprises the set of `base' homoclinic and heteroclinic orbits of~\eqref{ODE}. The assumption that these trajectories are transverse means that they lie along a transverse intersection of stable and unstable manifolds and is necessary to providing the results of this manuscript.

Using Hypotheses~\ref{hyp:Periodic} and \ref{hyp:HomHet} we will define a directed graph, written $\mathcal{G} = (\mathcal{V},\mathcal{E})$ as a collection of vertices $\mathcal{V}$ and edges $\mathcal{E}$. The vertices, $\mathcal{V}$, will be labelled as $\{1,\dots,p\}$ meant to represent each of the periodic orbits $\gamma_1(t),\dots,\gamma_p(t)$. The edges of the graph lie in one-to-one correspondence with the finitely many unique elements of $\W_0$. That is, the element $h(t) \in \W_0$ defines a directed edge on the graph $\mathcal{G}$ from vertex $i \in\mathcal{V}$ to $j\in\mathcal{V}$ if, and only if, $h(t) \in W^u(\gamma_i)$ and $h(t) \in W^s(\gamma_j)$ for all $t \in \mathbb{R}$. Informally, an edge connection from vertex $i$ to vertex $j$ exists if, and only if, there exists an orbit in $\W_0$ that goes asymptotically from $\gamma_i(t)$ to $\gamma_j(t)$. 

We note that our definition of the graph not only contains directed edges, but also allows for the possibility of loops (an edge that goes from one vertex back to itself) and multiple edges between two vertices. The loops come from the potential presence of homoclinic orbits in $\W_0$, while we also allow for the presence of multiple homoclinic and heteroclinic orbits to and from the same pair of periodic orbits, thus giving multiple edges between the same two vertices. The reader is referred to Figure~\ref{fig:Graphs} below for examples of these graphs coming from our numerical results in Section~\ref{sec:TubesDP}. We will denote a walk of length $k \geq 1$ on the directed graph $\mathcal{G}$ by the tuple of edges $(E_1,E_2,\dots,E_k) \in \mathcal{E}^k$. Such a walk represents starting at vertex $V_1 \in \mathcal{V}$ and moving along edge $E_1$ to vertex $V_2$, then moving along edge $E_2$ to vertex $V_3$, and continuing until finally moving along edge $E_k$ to finish the walk on $V_{k+1}$. This leads to our main result whose proof is left to Section~\ref{sec:Proofs}.

\begin{theorem}\label{thm:Main} 
Assume Hypotheses~\ref{hyp:Ham}-\ref{hyp:HomHet}. For any integer $k \geq 1$ and a walk $(E_1,E_2,\dots,E_k)$ on $\mathcal{G}$ moving through the sequence of vertices $(V_1,V_2,\dots,V_{k+1})$, there exists $M_k \geq 1$ such that the following is true:
\begin{enumerate}
	\item For each set of integers $N_2,\dots,N_k \geq M_k$ there exists a trajectory of~\eqref{ODE} belonging to $W^u(\gamma_{V_1})\cap W^s(\gamma_{V_{k+1}})$. 
	\item If $V_{k+1} = V_1$, for each set of integers $N_1,N_2,\dots,N_{k-1},N_k \geq M_k$ there exists a periodic trajectory of~\eqref{ODE}.   
\end{enumerate}
In both cases above the trajectory follows the walk on the graph, in that it leaves a neighborhood of $\gamma_{V_1}(t)$ following closely along the element of $\W_0$ corresponding to $E_1$, enters a neighborhood of $\gamma_{V_2}(t)$ and leaves, following closely to the element of $\W_0$ corresponding to $E_2$, repeating this process until finally entering a neighborhood of $\gamma_{V_{k+1}}(t)$ along the element of $\W_0$ corresponding to $E_k$. 
\end{theorem}

In the statement of Theorem~\ref{thm:Main} the integers $N_i$ roughly correspond to how many times the trajectory wraps around the periodic orbit $\gamma_{V_j}(t)$ while in its neighborhood. Since there is no upper bound on the number of wraps, the above theorem gives infinitely many trajectories that follow the prescribed walk on the graph. Furthermore, the lack of upper bound on the integers $N_i$ means that the trajectories described in Theorem~\ref{thm:Main} can be arbitrarily long. In this way, a trajectory can reside near one of the periodic orbits $\gamma_i(t)$ for an arbitrarily long time before jumping to another one. The reason we only have $(k-1)$ integers $N_i$ in the first case of homoclinic and heteroclinic trajectories is that going backwards or forwards in $t$, the trajectory asymptotically enters the neighbourhoods of $\gamma_{V_1}(t)$ and $\gamma_{V_{k+1}}(t)$, respectively, not to ever leave it again. In the second point of a periodic trajectory we have $k$ integers $N_i$ since the trajectory initially wraps around $\gamma_{V_1}(t)$, but since $V_{k+1} = V_1$, it follows that once the trajectory re-enters the neighborhood $\gamma_{V_1}(t)$ along the edge $E_k$, it completes one full period of the periodic orbit.    

Another useful application of Theorem~\ref{thm:Main} is that it can be applied iteratively. That is, one may use one of the homoclinic or heteroclinic connections from point (1) and add it into the set $\W_0$ to add another edge to the graph $\mathcal{G}$. The result is a new graph with more edges, to which we can then apply Theorem~\ref{thm:Main} again. This process can be continued in order to construct longer and more complex trajectories of the system~\eqref{ODE}. However, we emphasize that the most important aspect when applying these results is the original set $\W_0$, since all trajectories constructed using Theorem~\ref{thm:Main} can only `shadow' the base homoclinic and heteroclinic orbits in $\W_0$. 

\begin{rmk}
We remark that the results of Theorem~\ref{thm:Main} also hold for non-Hamiltonian systems since Hypothesis~\ref{hyp:Ham} only restricts the dimension of the phase space to the Hamiltonian level sets of~\eqref{ODE}. For non-Hamiltonian ODEs we can alternatively take~\eqref{ODE} to have a three-dimensional phase space and simply remove the assumption that the periodic orbits belong to the same energy level set in Hypothesis~\ref{hyp:Periodic}. Although such a relaxation is easily achieved, we have elected to include the Hamiltonian assumption on~\eqref{ODE} to emphasize the applicability of our results to the double pendulum, as well as other notable Hamiltonian systems such as the PCR3BP. 
\end{rmk}

\subsection{Application to the Double Pendulum}\label{subsec:DPapplication}

In this subsection we provide a number of applications of Theorem~\ref{thm:Main} to the double pendulum model~\eqref{eq:edp_ode}. From our work in Section~\ref{sec:EOM} we have that Hypothesis~\ref{hyp:Ham} is satisfied with Hamiltonian function~\eqref{eq:edp_hamiltonian}. In all examples we will take the periodic orbits in Hypothesis~\ref{hyp:Periodic} to be the UPOs near the index-1 saddles, thus satisfying the hyperbolicity assumptions as well. Unlike the work~\cite{Conley,McGehee} on the PCR3BP that proves the existence of homoclinic orbits to the Lyapunov orbits about the index-1 saddles, the double pendulum lacks such proofs and so we rely on our numerical work in Section~\ref{sec:TubesDP}. Thus, the set $\mathcal{W}_0$ is formed from collections of the numerically observed homoclinic and heteroclinic orbits detailed in Figures~\ref{fig:EDP_PoincareCut_L1_VariesEng}, \ref{fig:EDP_TubeStructure_L2}, and \ref{fig:EDP_HeteroclinicPoints}, with examples of the associated directed graphs given in Figure~\ref{fig:Graphs}. We will work exclusively with the periodic phase space that identifies $\theta_{1,2}$ with $\theta_{1,2} \pm 2\pi$ to ensure that our orbits are truly homoclinic. Beyond this numerical existence, we also lack a proof of transversality of the orbits, but again the figures appear to confirm these properties, as one can see they lie along transverse intersections between stable and unstable manifolds in the Poincar\'e section.

In what follows we will detail three separate applications of Theorem~\ref{thm:Main}. We begin with only the homoclinic orbits to the UPOs near the Down-Up saddle state. In this case we can use the base homoclinic orbits found in Section~\ref{sec:L1Hom} to create longer multi-homoclinic trajectories and periodic orbits. Although not explored here, these same results can be extended to the UPOs near the Up-Down saddle using the homoclinic orbits found in Section~\ref{sec:L2Hom}. Then, we use the heteroclinic orbits from Section~\ref{sec:Heteroclinic} to create new homoclinic and periodic orbits that bounce between neighborhoods of the index-1 saddles. Finally, we use all of our numerical work together to show that when the energy is large enough we can transit across vast regions of phase space by following all of the homoclinic and heteroclinic orbits found in Section~\ref{sec:TubesDP}.

\begin{figure}
    \centering
    \includegraphics[width=0.75\textwidth]{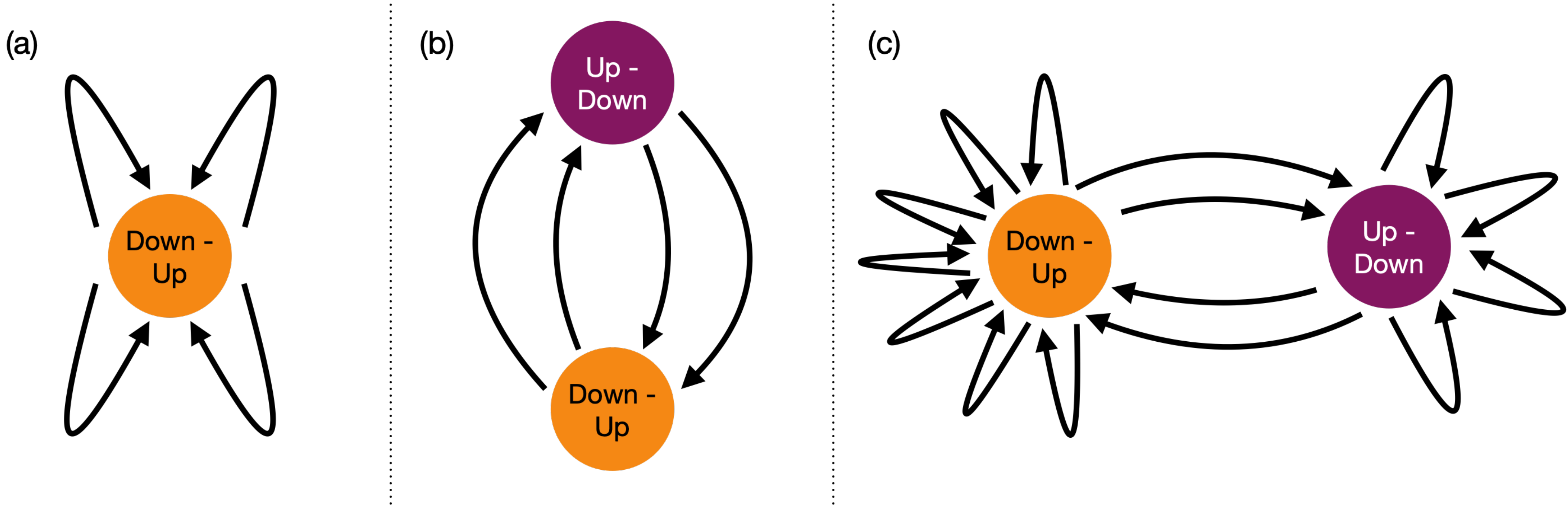}
    \vspace{-.1in}
    \caption{To apply Theorem~\ref{thm:Main} one uses the homoclinic and heteroclinic connections between UPOs to create a directed graph. (a) The graph whose single vertex corresponds to the UPO near the Down-Up equilibrium with $\mathcal{H} = -0.07$ for which Figure~\ref{fig:EDP_PoincareCut_L1_VariesEng} demonstrates the existence of four homoclinic orbits, resulting in four edges originating and terminating at the single vertex. (b) The heteroclinic orbits from Figure~\ref{fig:EDP_HeteroclinicPoints} and their reversed counterparts lead to a graph with two vertices corresponding to the UPOs near each of the index-1 saddles with $\mathcal{H} = 0.2$ and two pairs of directed edges from one edge to the other. (c) The resulting directed graph obtained by putting together the information from Figures~\ref{fig:EDP_TubeStructure_L1}, \ref{fig:EDP_TubeStructure_L2}, and \ref{fig:EDP_HeteroclinicPoints} with $\mathcal{H} = 0.2$.}
    \label{fig:Graphs}
\end{figure}

\subsubsection{Long Trajectories Near the Down-Up State}

An immediate corollary of Theorem~\ref{thm:Main} is that a single transverse homoclinic orbit to one of the periodic orbits $\gamma_i(t)$ gives rise to infinitely many more homoclinic and periodic orbits that `shadow' the original homoclinic orbit. In the language of the graph $\mathcal{G}$, we would have for any integer $k \geq 1$ a walk given by $E_1 = E_2 = \dots = E_k$ and $V_1 = V_2 = \dots = V_{k+1}$, where the edge $E_1$ is a loop on the vertex $V_1$. This result is a well-known consequence of having a transverse homoclinic orbit in three or more dimensional ODEs~\cite{Champneys,Haller,Homburg,Homburg2,Jens,2Pulse}.  Thus, Theorem~\ref{thm:Main} comes as a generalization of some of these results. In Figure~\ref{fig:MultiHom}(a) we present an illustration of a `base' homoclinic orbit that is asymptotic to a UPO near an index-1 saddle, along with cartoons of (b) a longer homoclinic orbit to the same UPO that shadows the base orbit twice ($k = 2$) and (c) a periodic orbit that shadows the base homoclinic once ($k = 1$), both of which are guaranteed by Theorem~\ref{thm:Main}. 

\begin{figure}[t]
    \centering
    \includegraphics[width=0.75\textwidth]{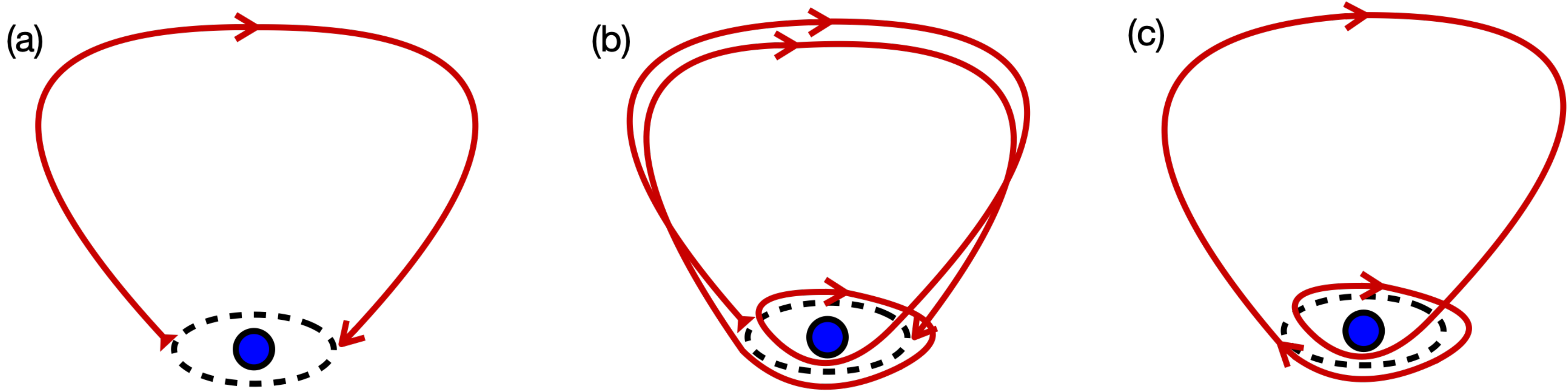}
    \caption{Theorem~\ref{thm:Main} gives the existence of multi-homoclinic and periodic solutions that shadow a single homoclinic orbit. (a) The `base' homoclinic orbit that is asymptotic to a single UPO near an index-1 saddle. (b) A double homoclinic orbit which shadows the base homoclinic twice with an intermediary transition wrapping around the UPO. (c) A periodic orbit that shadows the base homoclinic.}
    \label{fig:MultiHom}
\end{figure}

Beyond just a single homoclinic orbit emanating from a UPO near the Down-Up saddle, one may turn to our numerical findings in Figure~\ref{fig:EDP_PoincareCut_L1_VariesEng}, which demonstrate multiple homoclinic orbits. For example, inside the energy level set $\mathcal{H} = -0.07$ as in panel (b) of Figure~\ref{fig:EDP_PoincareCut_L1_VariesEng} we have four distinct homoclinic orbits. Thus, we can apply Theorem~\ref{thm:Main} with $\mathcal{W}_0$ comprised of these four homoclinic orbits. An illustration of the resulting directed graph is presented in Figure~\ref{fig:Graphs}(a). The result is an abundance of homoclinic and periodic trajectories formed by shadowing the elements of $\mathcal{W}_0$ for which each shadow causes the second pendulum arm to undergo a full rotation and the intermediary transition that switches between elements given by a long wobble of the second arm that shadows the UPO near the index-1 Down-Up equilibrium, as illustrated in Figure~\ref{fig:DP_UPO}. Finding these orbits is outlined in~\cite{Koon}, but we can also use similar numerical techniques from Section~\ref{sec:TubesDP} to identify longer homoclinic orbits. Figure~\ref{fig:L1DoubleHom}~(a) shows the result of continuing to flow the unstable manifold of the UPO near the Down-Up equilibrium forward to $\theta_2 = 4\pi$. Intersections in the Poincar\'e section now represent homoclinic orbits that have the second arm making two full rotations over the course of the trajectory. The curves are incomplete since many homoclinic orbits take a long time to reach the Poincar\'e section again, corresponding to large values of $N_i$ in Theorem~\ref{thm:Main}. One can see that the curve wraps around itself, appearing to give infinitely many intersections with the stable manifold near the location of the original homoclinic orbits from Figure~\ref{fig:EDP_PoincareCut_L1_VariesEng}(b). Figure~\ref{fig:L1DoubleHom}~(b) also includes the flow of the unstable manifold forward to $\theta_2 = 6\pi$, for which the resulting homoclinic orbits now have the second arm making three full rotations. Increasing the value of $\mathcal{H}$ results in more `base' trajectories with which to build $\mathcal{W}_0$, thus increasing the orbits to shadow for producing periodic and homoclinic trajectories. 

\begin{figure}[t]
    \centering
    \includegraphics[width=0.95\textwidth]{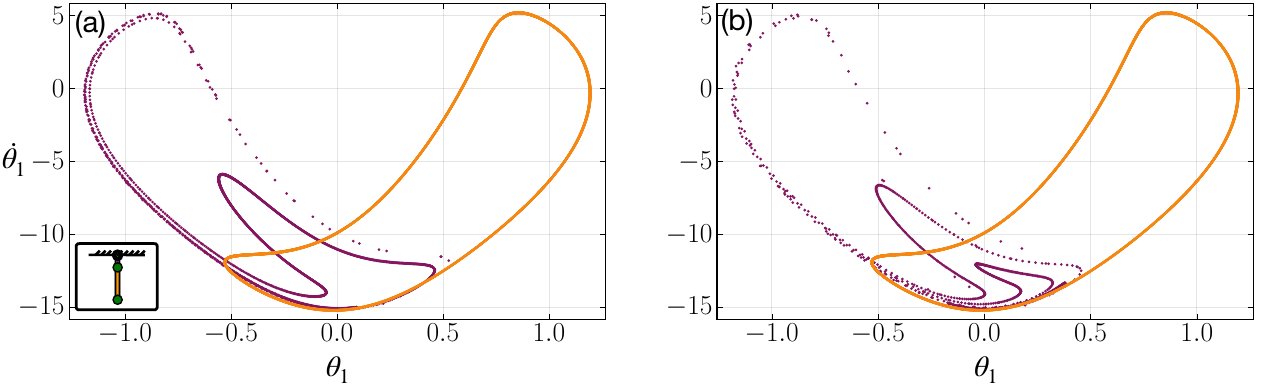}
    \vspace{-.15in}
    \caption{Flowing the unstable manifold of the UPO near the Down-Up equilibrium at $\mathcal{H} = -0.07$ forward to $\theta = 2\pi k$, $k \geq 2$, can numerically demonstrate the existence of longer homoclinic orbits to it. (a) Intersections in the Poincar\'e section of the unstable manifold flowed forward in time to $\theta_2 = 4\pi$ and the stable manifold flowed backward to $\theta_2 = 0$ gives the existence of doubly-homoclinic orbits. (b) Similarly, flowing the unstable manifold forward in time to $\theta_2 = 6\pi$ demonstrates the existence of triply-homoclinic orbits.}
    \label{fig:L1DoubleHom}
\end{figure}

\subsubsection{Transitions Between the Down-Up and Up-Down State}

Section~\ref{sec:Heteroclinic} numerically demonstrates heteroclinic orbits that transfer one from a UPO near the Down-Up equilibrium to a UPO near the Up-Down equilibrium. Furthermore, applying the reversible symmetry of the double pendulum to these orbits also gives the existence of heteroclinic orbits that transfer from a UPO near the Up-Down equilibrium to a UPO near the Down-Up equilibrium. Using this information we can take $\mathcal{W}_0$ to be these four heteroclinic orbits that go back and forth between UPOs near the Down-Up and Up-Down equilibria at $\mathcal{H} = -0.07$, resulting in the directed graph presented in Figure~\ref{fig:Graphs}(b). Applying Theorem~\ref{thm:Main} leads to the existence of homoclinic, heteroclinic, and periodic orbits that bounce between the neighborhoods of the index-1 saddles and can be understood by traversing the edges of the corresponding graph. Such orbits begin by wobbling around one of the index-1 saddles, mimicking the motion of the nearby UPO, and then shadow one of the heteroclinic orbits to move to a wobbling motion about the other index-1 saddle, according to its nearby UPO. All orbits given by Theorem~\ref{thm:Main} continue this process of wobbling, transferring, wobbling, transferring, and so on to either generate a homoclinic, heteroclinic, or periodic orbit of the system~\eqref{eq:edp_ode}. What differentiates these orbits is where they start and end, with homoclinics starting and ending with wobbles about the same index-1 saddle, heteroclinics starting and ending with wobbles around different saddles, and periodic orbits repeating the wobble, transfer, wobble process ad infinitum.   

\begin{figure}[t]
    \centering
    \includegraphics[width=.95\textwidth]{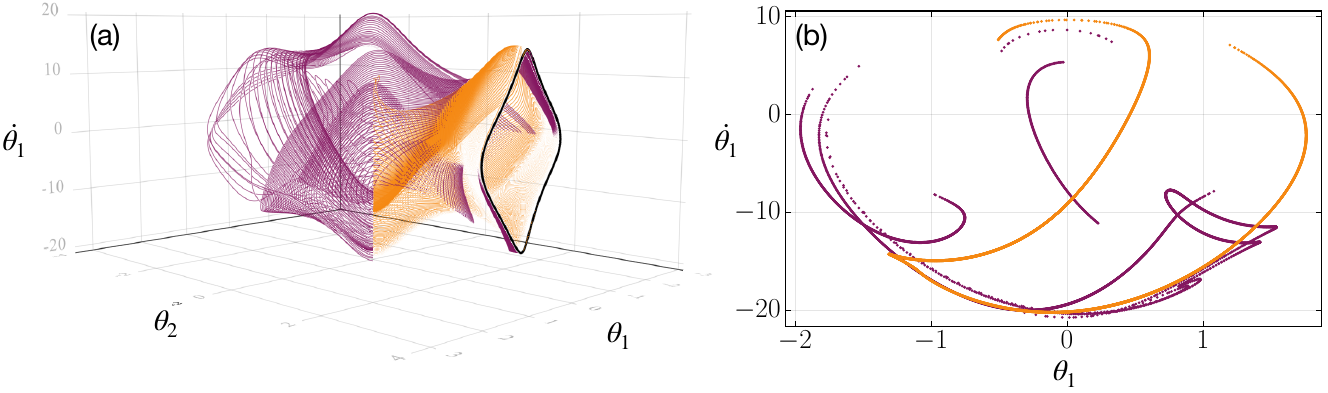}
    \vspace{-.15in}
    \caption{Continuing to flow the unstable manifold of the UPO near the Down-Up equilibrium (black) forward in time reveals the existence of new homoclinic orbits to it. (a) The manifolds in $(\theta_1,\theta_2,\dot\theta_1)$-space. (b) The intersection of the stable and unstable manifolds in the Poincar\'e section $\theta_1 = \theta_2$, but now with the unstable manifold crossing in the other direction. Intersections of the curves in the Poincar\'e section provide the existence of homoclinic orbits.}
    \label{fig:HetApplication}
\end{figure}

Similar to the previous section, we can flow the manifolds of the UPO for a longer time to identify more complex homoclinic and heteroclinic orbits guaranteed by Theorem~\ref{thm:Main}. For example, we have flowed the unstable manifold of the UPO near the Down-Up equilibrium forward in time until we meet the Poincar\'e section $\theta_1 = \theta_2$ again, but this time crossing in the negative direction. In Figure~\ref{fig:HetApplication} we provide this new intersection with the Poincar\'e section, along with the intersection of the stable manifold from~\S\ref{sec:Heteroclinic}. Intersections of these curves represent homoclinic orbits to the UPO near the Down-Up equilibrium that are entirely distinct from those found in~\S\ref{sec:L1Hom}. The curves are not connected in the Poincar\'e section, partially attributed to the extremely long time to flow parts of the manifold back to the Poincar\'e section and partially due to the fact that a portion of the unstable manifold does not turn around and come back to the Poincar\'e section from the other side. We have performed the same calculation for the UPO near the Up-Down equilibrium but the results are not included for brevity.

\subsubsection{Full State-Space Dynamics}

Our final application of Theorem~\ref{thm:Main} unites all of our numerical findings from Section~\ref{sec:TubesDP}. Notice that Figures~\ref{fig:EDP_PoincareCut_L1_VariesEng}, \ref{fig:EDP_TubeStructure_L2}, and \ref{fig:EDP_HeteroclinicPoints} all provide the existence of homoclinic and heteroclinic orbits for $\mathcal{H} = 0.2$. Thus, when the energy $\mathcal{H}$ is sufficiently large, it is possible to transit around the homoclinic orbits to the UPOs near the index-1 saddles and between these UPOs along the heteroclinic orbits. Taking $\mathcal{W}_0$ to be the set of all such homoclinic and heteroclinic orbits found numerically for $\mathcal{H} = 0.2$, we can generate the directed graph presented in Figure~\ref{fig:Graphs}(c) and apply Theorem~\ref{thm:Main}. This gives the existence of long trajectories in phase space that temporarily shadow the base orbits in $\mathcal{W}_0$ to transit back and forth between neighborhoods of the index-1 saddles. 

Precisely, one is able to determine the existence of orbits that perform the acrobatic motion of one arm making a full rotation by following one of the homoclinic orbits in Sections~\ref{sec:L1Hom} and \ref{sec:L2Hom} and transferring between these motions by shadowing the transitory heteroclinic orbits of Section~\ref{sec:Heteroclinic}. Much like the previous applications, Theorem~\ref{thm:Main} allows one to stitch these motions together to form a long, complex orbit of the double pendulum that follows such an itinerary. In fact, Theorem~\ref{thm:Main} gives infinitely many such orbits that follow a given itinerary since one can choose the integers $N_i$ that approximately prescribe how long the intermediary wobbling motion will take place before performing an arm swing or a transition between saddle neighborhoods. Therefore, much like the Interplanetary Transport Network of the solar system, coming from similar results in the PCR3BP and related models, the double pendulum has a similar transport network given by similar tube dynamics.    

\subsection{Application to the PCR3BP}\label{subsec:TBPapplication}

Let us briefly comment on the application of our results to the PCR3BP. In~\cite{Koon} the authors undertook a similar numerical investigation of the Lyapunov orbits about the $L_1$ and $L_2$ Lagrange points of the PCR3BP. Their numerical findings were summarized in Figures~\ref{fig:TBP_HomoclinicOrbit} and \ref{fig:TBP_HeteroclinicOrbit}. These findings were then leveraged to provide analytical results that detail the existence of long trajectories that follow these homoclinic and heteroclinic orbits according to a prescribed itinerary. The results of Theorem~\ref{thm:Main} can similarly be applied to the PCR3BP in the same way we did for the double pendulum in the previous subsection, and in many instances our results herein provide more details on the itineraries than those in~\cite{Koon}. We summarize our extensions as follows:
\begin{enumerate}
	\item Theorem~\ref{thm:Main} demonstrates exactly how trajectories follow a given itinerary by shadowing the homoclinic and heteroclinic orbits in $\mathcal{W}_0$.
	\item The presence of multiple homoclinic and heteroclinic orbits provides different tracks for the itineraries to follow as they leave neighborhoods of the Lagrange points. 
	\item The periodic orbits $\gamma_i(t)$ need not be Lyapunov orbits, and so Theorem~\ref{thm:Main} can be applied to establish transit between any hyperbolic periodic orbits in the PCR3BP for which a transverse heteroclinic connection is known to exist.
\end{enumerate}
In the context of the PCR3BP we have that the values of the integers $N_i$ in Theorem~\ref{thm:Main} approximately correspond to how long the trajectory will orbit around a Lagrange point in the PCR3BP by shadowing its Lyapunov orbit. The larger the value of $N_i$, the longer the trajectory orbits a Lagrange point. Thus, one may use Theorem~\ref{thm:Main} to not only follow a given itinerary, but also to rest for arbitrarily long periods of time near the Lagrange points. This was similarly established in \cite[Theorem~4.2]{Koon} with the values $r_i$, although by different methods to the proofs in the following section.

\section{Proof of Theorem~\ref{thm:Main}}\label{sec:Proofs} 

Throughout this section we will consider $k \geq 1$ and the sequences 
\begin{equation}
	(E_1,E_2,\dots,E_k) \in \mathcal{E}^k, \quad (V_1,V_2,\dots,V_{k+1})\in\mathcal{V}^{k+1}
\end{equation} 
as they are given in the statement of Theorem~\ref{thm:Main}. In an effort to simplify the notation of this section we will define 
\begin{equation}\label{tildegamma}
	\tilde{\gamma}_i(t) = \gamma_{V_i}(t)
\end{equation} 
for all $i = 1,2,\dots,k+1$. This gives a new sequence of periodic orbits $\{\tilde\gamma_1(t),\tilde\gamma_2(t),\dots,\tilde\gamma_{k+1}(t)\}$ which corresponds to the sequence of vertices traversed in the walk on the graph $\mathcal{G}$. We note that depending on the choices of $E_i$ and $V_i$, it may be the case that $\tilde{\gamma}_i(t) = \tilde{\gamma}_j(t)$ for some $i \neq j$, which will necessarily be true when $k > p$. The effect is that we may now define a graph $\mathcal{G}' = (\mathcal{V}',\mathcal{E}')$ for which the vertex set $\mathcal{V}' = \{1,2,\dots,k+1\}$ is meant to represent each of the periodic orbits $\tilde\gamma_1(t),\tilde\gamma_2(t),\dots,\tilde\gamma_{k+1}(t)$. The edge set $\mathcal{E}'$ is solely comprised of the edges that make up the walk, thus giving that we have a directed edge from vertex $i$ to vertex $i+1$ for all $1 \leq i \leq k$. Therefore, $\mathcal{G}'$ is a directed chain meant to represent transits from $\gamma_{V_i}(t)$ to $\gamma_{V_{i+1}}(t)$ dictated by the walk $(E_1,E_2,\dots,E_k)$ on $\mathcal{G}$. In the following subsections we will work exclusively with the $\tilde\gamma_i(t)$ periodic orbits and the graph $\mathcal{G}'$ to simplify notation.

The proof of Theorem~\ref{thm:Main} is broken down over the following subsections. We begin in~\S\ref{subsec:Local} with an understanding of the dynamics of the differential equation~\eqref{ODE} in a neighborhood of each $\tilde\gamma_i(t)$. To achieve this understanding we define an appropriate Poincar\'e section and consider iterates through this section governed by a Poincar\'e mapping. When we restrict the domain of the Poincar\'e mappings close enough to the intersection with the periodic orbits we have that each iterate of the map roughly represents a trajectory of the continuous-time system~\eqref{ODE} that wraps around the periodic orbit once. Upon establishing these local results, we then move to~\S\ref{subsec:Transfer} where we define maps that transport iterates through phase space from one local description to another. Then~\S\ref{subsec:HomHet} and~\S\ref{subsec:Periodic} set up appropriate matching conditions which, when satisfied, construct the trajectories outlined in point (1) and point (2), respectively, of Theorem~\ref{thm:Main}.

\subsection{Dynamics Near the Periodic Orbits}\label{subsec:Local}

In this subsection we will work to understand the dynamics of system~\eqref{ODE} in neighbourhoods of the periodic orbits $\tilde\gamma_1(t), \dots, \tilde\gamma_{k+1}(t)$. The results apply to all $\tilde\gamma_i(t)$ with $1 \leq i \leq k+1$, and so whenever possible we will define global constants that apply in all neighbourhoods of the periodic orbits. Such global constants can always be found since they can be defined as either the maximum or the minimum, depending on the context, of the individual constants needed for each $\tilde\gamma_i(t)$. We also emphasize that all results moving forward can be obtained independently of the value of $k$. This comes from the definition of the $\tilde\gamma_i(t)$ in~\eqref{tildegamma} and the fact that there are only $p \geq 1$ periodic orbits, meaning that we need only understand the dynamics near each of the $p\geq 1$ periodic orbits and carry them over to the redundant $\tilde\gamma_i(t) = \tilde\gamma_j(t)$ with $i \neq j$. As discussed at the beginning of this section, we will continue to use the $\tilde\gamma_i(t)$ notation for convenience, but the reader should keep in mind that there are only finitely many unique choices for our periodic orbits.  

To begin, Hypothesis~\ref{hyp:Ham} implies that we may reduce the full four-dimensional phase space of~\eqref{ODE} to the three-dimensional level set of $\mathcal{H}$ that contains the periodic orbits $\tilde\gamma_i(t)$. Inside of this three-dimensional phase space we can define $\Sigma_i$ to be a Poincar\'e section comprised of the plane orthogonal to $\gamma_i(0)$ restricted to a small neighborhood of this point. Let us define the local Poincar\'e mapping, denoted $P_i$, by 
\begin{equation}\label{Psec}
	P_i: \Sigma'_i \to \Sigma_i,
\end{equation} 
where the domain $\Sigma'_i \subset \Sigma_i$ is such that $P(\Sigma'_i) \subset \Sigma_i$, making the mapping well-defined. Note that $P_i(\tilde\gamma_i(0)) = \tilde\gamma_i(0)$, and from Hypothesis~\ref{hyp:Periodic} the linearization $DP(\tilde\gamma_i(0))$ has two nonzero eigenvalues $\lambda^{-1}_i,\lambda_i \in \mathbb{R}$, corresponding to the nontrivial Floquet multipliers of $\tilde\gamma_i(t)$. Upon potentially relabelling $\lambda_i$ and $\lambda_i^{-1}$, we will assume $0 < |\lambda_i^{-1}| < 1 < |\lambda_i|$, making $\lambda_i$ the unstable eigenvalue and $\lambda_i^{-1}$ the stable eigenvalue. This leads to our first result which details a change of variables to locally straighten the stable and unstable manifolds of the fixed point $\tilde\gamma_i(0)$ in a small neighborhood. 

\begin{lem}\label{lem:Shil} 
Assume Hypotheses~\ref{hyp:Ham} and \ref{hyp:Periodic} are met. Then there exists $\delta > 0$ such that for each ${i \in \{1,\dots,k+1\}}$ there is a smooth change of coordinates mapping $u\in\Sigma'_i$ to $v_i = (v^s_i,v^u_i)$ near the fixed point $u = \tilde\gamma_i(0)$, and smooth functions $f^s_{i,j},f^u_{i,j}:\mathcal{I}\times\mathcal{I}\to \mathbb{R}$, $j = 1,2$, so that (\ref{Psec}) is of the form 
	\begin{equation}\label{Shil}
		\begin{split}
			v^s_{i,n+1} &= [\lambda_i^{-1} + f_{i,1}^s(v^s_{i,n},v^u_{i,n})v^s_{i,n} + f_{i,2}^s(v^s_{i,n},v^u_{i,n})v^u_{i,n}]v^s_{i,n}, \\
			v^u_{i,n+1} &= [\lambda_i + f_{i,1}^u(v^s_{i,n},v^u_{i,n})v^s_{i,n} + f_{i,2}^u(v^s_{i,n},v^u_{i,n})v^u_{i,n}]v^u_{i,n}, \\
		\end{split}
	\end{equation}	
	where $v^s_{i,n},v^u_{i,n} \in \mathcal{I} := [-\delta,\delta]$.
\end{lem}

\begin{proof}
Let $\xi^s_i$ be the eigenvector of $DP_i(\tilde\gamma_i(0))$ associated to $\lambda_i^{-1}$ and $\xi^u_i$ be the eigenvector associated to $\lambda_i$. It follows from the stable manifold theorem for maps that there exists a $\delta_i > 0$ and smooth functions $w^s_i,w^u_i:[-\delta_i,\delta_i] \to \mathbb{R}$ with $w^s_i(0) = (w^s_i)'(0) = w^u_i(0) = (w^u_i)'(0) = 0$ such that the local stable and unstable manifolds of $\tilde\gamma_i(0)$ can be written 
\begin{equation}
	\begin{split}
		W^s_\mathrm{loc}(\tilde\gamma_i(0)) &= \{\tilde\gamma_i(0) + v^s_i\xi^s_i + w^s_i(v^s_i)\xi^u_i:\ v^s_i \in [-\delta_i,\delta_i]\}, \\
		W^u_\mathrm{loc}(\tilde\gamma_i(0)) &= \{\tilde\gamma_i(0) + v^u_i\xi^u_i + w^u_i(v^u_i)\xi^s_i:\ v^u_i \in [-\delta_i,\delta_i]\}.
	\end{split}
\end{equation}
Then, for $u\in\Sigma_i'$ in a small neighborhood of $\tilde\gamma_i(0)$ we may introduce the change of variable
\begin{equation}
	(v^s_i,v^u_i) \mapsto u = \tilde\gamma_i(0) + (v^s_i\xi^s_i + w^s_i(v^s_i)\xi^u_i) + (v^u_i\xi^u_i + w^u_i(v^u_i)\xi^s_i). 
\end{equation}
From the tangency properties of the $w^s_i,w^u_i$ functions at $(v^s_i,v^u_i) = (0,0)$, it follows that the above change of variable is a local diffeomorphism. Putting this change of variable into the Poincar\'e map $P_i$ and expanding in a neighborhood of $(v^s_i,v^u_i) = (0,0)$ gives the expansion~\eqref{Shil}. Moreover, this change of variable gives that $W^s_\mathrm{loc}(\tilde\gamma_i(0)) = \{v^u_i = 0\}$ and $W^u_\mathrm{loc}(\tilde\gamma_i(0)) = \{v^s_i = 0\}$. Finally, we may take $\delta := \min \{\delta_1,\dots,\delta_{k+1}\}$ to arrive at the $\delta > 0$ in the statement of the lemma. This completes the proof.
\end{proof} 

With $\delta > 0$ taken sufficiently small, each iterate of~\eqref{Shil} represents a solution of~\eqref{ODE} that completes a full revolution of $\tilde\gamma_i(t)$. In this way, we can use~\eqref{Shil} to quantify orbits of~\eqref{ODE} that pass through a neighborhood of the periodic orbits $\tilde\gamma_i(t)$. Precisely, the saddle structure gives that orbits enter the neighborhood closely following the stable manifold and exit the neighborhood closely following the unstable manifold, all while wrapping around the periodic orbit. This is made precise with the following lemma which originally appeared in~\cite{BramLDS} and comes as the discrete-time analogue of the main result of~\cite{Schecter}.   

\begin{lem}[\cite{BramLDS} Lemma~3.2] \label{lem:ShilSol} 
	There exists constants $\eta \in (0,1)$ and $M > 0$ such that the following is true: for each $1 \leq i \leq k+1$, $N > 0$, and $a^s,a^u \in \mathcal{I}$ there exists a unique solution to~\eqref{Shil}, written $v_{i,n} = (v^s_{i,n},v^u_{i,n}) \in \mathcal{I}\times\mathcal{I}$ with $n \in \{0,\dots,N\}$, such that
	\begin{equation}
		v_{i,0}^s = a^s, \quad v_{i,N}^u = a^u.
	\end{equation}
	Furthermore, this solution satisfies 
	\begin{equation}\label{LocalBnds}
		|v^s_{i,n}| \leq M\eta^n, \quad |v^u_{i,n}| \leq M\eta^{N-n},
	\end{equation}
	for all $n \in \{0,\dots,N\}$, $v_{i,n} = v_{i,n}(a^s,a^u)$ depends smoothly on $(a^s,a^u)$, and the bounds~\eqref{LocalBnds} also hold for the derivatives of $v$ with respect to $(a^s,a^u)$.
\end{lem}

\subsection{Transferring Between Periodic Orbits}\label{subsec:Transfer}

Recall that $\W$ comprises the set of all homoclinic and heteroclinic orbits of~\eqref{ODE} between the $p \geq 1$ periodic orbits $\{\gamma_i(t)\}_{i = 1}^p$. From Hypothesis~\ref{hyp:HomHet} we have that $\W \neq \emptyset$, and we have assumed the existence of a set $\W_0 \subset \W$ which is a nonempty finite collection of transverse homoclinic and heteroclinic orbits between the periodic orbits. Moreover, the walk along the edges $(E_1,E_2,\dots,E_k)$ can equivalently be written as an ordered list of elements in $\W_0$. Let us denote the element $h_i(t) \in \W_0$ to be the orbit used to form the edge $E_i$ on the graph $\mathcal{G}$. Then, from Hypothesis~\ref{hyp:HomHet} and the definition of the graph $\mathcal{G}'$, each $h_i(t) \in \W_0$ lies on a transverse intersection of $W^u(\tilde\gamma_i(t))$ and $W^s(\tilde\gamma_{i+1}(t))$. Using the local coordinates of Lemma~\ref{lem:Shil}, there exists $h^s_{i+1}$ so that 
\begin{equation}
	(v^s_{i+1},v^u_{i+1}) = (h^s_{i+1},0) \in \{h_i(t) \cap (-\delta,\delta)\times \{0\}\} \subset W^s(\tilde\gamma_{i+1}(0))\cap W^u(\tilde\gamma_i(0))
\end{equation} 
near the fixed point $\tilde\gamma_{i+1}(0)$ of the Poincar\'e mapping $P_{i+1}$. That is, the point $(h^s_{i+1},0)\in\mathcal{I}\times\mathcal{I}$ represents a point of intersection of the orbit $h_i(t)$ with the local Poincar\'e section $\Sigma_{i+1}$, lying along the stable manifold of $\tilde\gamma_{i+1}(t)$. Similarly, there exists $h_i^u$ so that 
\begin{equation}
	(v^s_i,v^u_i) = (0,h_i^u) \in \{h_i(t) \cap \{0\}\times(-\delta,\delta)\} \subset W^s(\tilde\gamma_{i+1}(0))\cap W^u(\tilde\gamma_i(0)),
\end{equation} 
representing a point of intersection of the orbit $h_i(t)$ with the local Poincar\'e section $\Sigma_i$, lying along the unstable manifold of $\tilde\gamma_i(t)$. We note that there are infinitely many choices of $h_{i+1}^s$ and $h_i^u$ and it does not matter which we choose except that they lie entirely in the interior of $\mathcal{I}\times\mathcal{I}$. Furthermore, the assumption of transversality of the orbit $h_i(t)$ implies that $(h^s_{i+1},0)$ represent a transverse point of intersection between the stable manifold of $\tilde\gamma_{i+1}(t)$ inside the Poincar\'e sections $\Sigma_{i+1}$. The analogous statement also holds in that the point $(0,h_i^u)$ lies along a transverse point of intersection inside the Poincar\'e section $\Sigma_i$.  

The following lemma makes explicit use of the assumption that the intersections of $W^s(\tilde\gamma_{i+1}(0))\cap W^u(\tilde\gamma_i(0))$ are transverse. We will adopt the notation $B_r(x)$ to denote the ball of radius $r > 0$ about the point $x$.

\begin{lem}\label{lem:GFn} 
	There exists $\varepsilon > 0$ such that the following is true for all $1 \leq i \leq k$:
	\begin{enumerate}
		\item There exists a smooth function $G^u_{i+1}:B_\varepsilon(h^s_{i+1},0)\to\mathbb{R}$ such that $G^u_{i+1}(v^s_{i+1},v^u_{i+1}) = 0$ if and only if $(v^s_{i+1},v^u_{i+1})\in W^u(\tilde\gamma_i(0))\cap B_\varepsilon(h^s_{i+1},0)$. Furthermore, $\partial_{v^s_{i+1}}G^u_{i+1}(h^s_{i+1},0) \neq 0$.
		\item There exists a smooth function $G^s_i:B_\varepsilon(0,h^u_i)\to\mathbb{R}$ such that $G^s_i(v^s_i,v^u_i) = 0$ if and only if $(v^s_i,v^u_i)\in W^s(\tilde\gamma_{i+1}(0))\cap B_\varepsilon(0,h^u_i)$. Furthermore, $\partial_{v^u_i}G^u_{i,j}(0,h^u_i) \neq 0$.
	\end{enumerate}
\end{lem} 

\begin{proof}
We will only prove the first statement since the second is handled identically. By assumption we have that $(h^s_{i+1},0) \in \mathcal{I}\times\mathcal{I}$ represents a transverse point of intersection between $W^s(\tilde\gamma_{i+1}(0))$ and $W^u(\tilde\gamma_i(0))$. Therefore, we can locally parameterize $W^u(\tilde\gamma_i(0))$ near $(h^s_{i+1},0)$ by the function $v^s_{i+1} = g(v^u_{i+1})$ so that $h^s_{i+1} = g(0)$. Then, defining 
	\begin{equation}
		G^u_{i+1}(v^s_{i+1},v^u_{i+1}) = v^s_{i+1} - g(v^u_{i+1})
	\end{equation}
	gives a function that locally vanishes on $W^u(\tilde\gamma_i(0))$ and satisfies $\partial_{v^s_{i+1}}G^u_{i+1}(h^s_{i+1},0) \neq 0$. We can therefore take $\varepsilon > 0$ small enough to ensure that $B_\varepsilon(h^s_{i+1},0)$ lies inside of $\mathcal{I}\times\mathcal{I}$ and that $G^u_{i+1}$ only vanishes when $(v^s_{i+1},v^u_{i+1}) \in W^u(\tilde\gamma_{i+1}(0))$. This completes the proof of the first point.
\end{proof} 

The final result of this subsection describes a push-forward map that transfers one between the local Poincar\'e sections $\Sigma_i$. This transfer is done by shadowing the homoclinic/heteroclinic orbits $h_i(t)$ which make up the edges $E_i$ in the walk on $\mathcal{G}$. 

\begin{lem}\label{lem:PushForward} 
	There exists an $\varepsilon > 0$ such that for each $1 \leq i \leq k$ there exists a smooth map $\Pi_i: B_\varepsilon(0,h^u_i) \to \mathcal{I}\times\mathcal{I}$ such that $\Pi_i(0,h^u_i) = (h^s_{i+1},0)$ and $\Pi_i$ is diffeomorphism in a neighborhood of $(0,h^u_i)$. 
\end{lem}

\begin{proof}
Smooth dependence on initial conditions for the continuous-time flow~\eqref{ODE} guarantees that for some $\varepsilon > 0$ chosen sufficiently small, each point in $B_\varepsilon(0,h^u_i)$ defines a trajectory that lies close to the orbit $h_i(t)$ that eventually will intersect the local component of the Poincar\'e section $\Sigma_{i+1}$ in a small neighborhood of $(h^s_{i+1},0)$. We define $\Pi_i$ to be the map that transports the initial conditions in $B_\varepsilon(0,h^u_i)$ to their image in the small neighborhood of $(h^s_{i+1},0)$ in $\Sigma_{i+1}$, with $\varepsilon$ taken sufficiently small that the range lies entirely inside of $\mathcal{I}\times\mathcal{I}$ when converted to the coordinates $(v^s_{i+1},v^u_{i+1})$ defined in Lemma~\ref{lem:Shil}. By definition, $\Pi_i(0,h^u_i) = (h^s_{i+1},0)$, smooth dependence on initial conditions guarantees that $\Pi_i$ is a smooth map, and uniqueness of solutions to~\eqref{ODE} guarantees that $\Pi_i$ is invertible on its range. This completes the proof. 	
\end{proof} 

\subsection{Constructing Homoclinic and Heteroclinic Trajectories}\label{subsec:HomHet}

With the previous subsections we were able to understand the local dynamics near each of the periodic orbits $\tilde\gamma_i(t)$ in the Poincar\'e sections $\Sigma_i$ and how we can transfer between these local dynamical systems. Here we will now prove Theorem~\ref{thm:Main}, starting with the homoclinic and heteroclinic orbits described in the point (1). It should be noted that we need only consider $k \geq 2$ since the case $k = 1$ is handled by the elements $h_i(t) \in \W_0$. 

We start by employing Lemma~\ref{lem:ShilSol} to obtain solutions of~\eqref{Shil} in a neighborhood of $\tilde\gamma_i(0)$, $2 \leq i \leq k$, given by
\begin{equation}
	\begin{split}
		\{(v^s_{2,n},v^u_{2,n})\}_{n = 0}^{N_2},& \quad v^s_{2,0} = h^s_2 +a^s_2, \quad v^u_{2,N_2}= h^u_2 +a^u_2 \\
		\{(v^s_{3,n},v^u_{3,n})\}_{n = 0}^{N_3},& \quad v^s_{3,0} = h^s_3 +a^s_3, \quad v^u_{3,N_3}= h^u_3 +a^u_3 \\
		&\vdots \\
		\{(v^s_{k,n},v^u_{k,n})\}_{n = 0}^{N_{k}},& \quad v^s_{k,0} = h^s_k +a^s_{k}, \quad v^u_{k,N_{k}}= h^u_k +a^u_{k},
	\end{split}	
\end{equation} 
respectively. Here the $a^s_i$ and $a^u_i$ are taken to be small and the $N_i \geq 1$ sufficiently large to guarantee that 
\begin{equation}\label{Expansions}
	\begin{split}
		(v^s_{i,0},v^u_{i,0}) &= (h^s_i+a_i^s,\mathcal{O}(\eta^{N_i})) \in B_\varepsilon(h^s_i,0) \\
		(v^s_{i,N_j},v^u_{i,N_j}) &= (\mathcal{O}(\eta^{N_i}),h^u+a_i^u) \in B_\varepsilon(0,h^u_i) \\
	\end{split}
\end{equation} 
for each $2 \leq j \leq k$, where $\varepsilon > 0$ is chosen small enough to satisfy the requirement for Lemmas~\ref{lem:GFn} and \ref{lem:PushForward}. Each of the above solutions represents a trajectory of the continuous-time system~\eqref{ODE} that wraps at least $N_i \geq 1$ times around $\tilde\gamma_i(t)$. We may stitch these solutions together to create a homoclinic or heteroclinic orbit by attempting to satisfy the following matching conditions:
\begin{equation}\label{HomMatch}
	\begin{split}
		G^u_2(v^s_{2,0},v^u_{2,0}) &= 0 \\
		\Pi_2(v^s_{2,N_2},v^u_{2,N_2}) - (v^s_{3,0},v^u_{3,0}) &= 0 \\
		\Pi_3(v^s_{3,N_3},v^u_{3,N_3}) - (v^s_{4,0},v^u_{4,0}) &= 0 \\ 
		&\vdots \\
		\Pi_{k-1}(v^s_{k-1,N_{k-1}},v^u_{k-1,N_{k-1}}) - (v^s_{k,0},v^u_{k,0}) &= 0 \\
		G^s_k(v^s_{k,N_{k}},v^u_{k,N_{k}}) &= 0.
	\end{split}
\end{equation} 
Indeed, from Lemma~\ref{lem:GFn} we have that the first condition guarantees that the trajectory belongs to $W^u(\tilde\gamma_1(0))$, and the final condition guarantees that the trajectory belongs to $W^s(\tilde\gamma_{k+1}(0))$. The intermediary conditions use the push-forward maps $\Pi_i$ to guarantee that each time the trajectory leaves the neighborhood $\tilde\gamma_i(0)$ it moves to the neighborhood of $\tilde\gamma_{i+1}(0)$ to complete $N_i \geq 1$ iterations before leaving for the next neighborhood. We are now in position to prove Theorem~\ref{thm:Main}(1). Prior to doing so, we present the following lemma coming from~\cite{RadialPulse} that will be used throughout the proof.

\begin{lem}[\cite{RadialPulse} Lemma~7.2]\label{lem:Roots} 
If $H:\R^n \to \R^n$ is smooth and there are constants $0 < \kappa < 1$ and $\rho > 0$, a vector $w_0 \in \R^n$, and an invertible matrix $A \in \R^{n\times n}$ so that
\begin{enumerate}
	\item $\|1 - A^{-1}DH(w)\| \leq \kappa$ for all $w \in B_\rho(w_0)$, and
	\item $\|A^{-1}H(w_0)\| \leq (1 - \kappa)\rho$
\end{enumerate}
then $H$ has a unique root $w_*$ in $B_\rho(w_0)$, and this root satisfies $|w_* - w| \leq \frac{1}{1-\kappa}\|A^{-1}F(w_0)\|$.
\end{lem}  

\begin{proof}[Proof of Theorem~\ref{thm:Main}(1)]
Let us begin with the case $k = 2$ for illustration. In this case~\eqref{HomMatch} becomes
\begin{equation}\label{2HomMatch}
\begin{split}
		G^u_2(v^s_{2,0},v^u_{2,0}) &= 0 \\
		G^s_2(v^s_{2,N_{2}},v^u_{2,N_{2}}) &= 0 \\
	\end{split}
\end{equation}
meaning that the orbit asymptotically connects $\tilde\gamma_1(0)$ and $\tilde\gamma_3(0)$ while iterating $N_2$ times through the neighborhood of $\tilde\gamma_2(0)$. Let us define the function $H_2:\R^2 \to \R^2$ by 
\begin{equation}
	H_2(a^s_2,a^u_2) := \begin{bmatrix}
		G^u_2(v^s_{2,0}(a^s_2,a^u_2),v^u_{2,0}(a^s_2,a^u_2)) \\ G^s_2(v^s_{2,N_2}(a^s_2,a^u_2),v^u_{2,N_2}(a^s_2,a^u_2))	
	\end{bmatrix}
\end{equation} 
so that for $(a^s_2,a^u_2)$ small the roots of $H_2$ are exactly the values of $(a^s_2,a^u_2)$ that satisfy the matching conditions~\eqref{2HomMatch}. Using the expansions~\eqref{Expansions} we have
\begin{equation}\label{2HomExpansion}
		\begin{split}
		G^s_2(v^s_{2,0},v^u_{2,0}) &= G^s_2(h^s_2+a^s_2,\mathcal{O}(\eta^{N_2})) = G^s(h^s_2+a^s_2,0) + \mathcal{O}(\eta^{N_2})  \\
		G^u_2(v^s_{2,N_2},v^u_{2,N_2}) &= G^u_2(\mathcal{O}(\eta^{N_2}),h^u_2+a^u_2) =  G^u(0,h^u_2+a^u_2) + \mathcal{O}(\eta^{N_2})
	\end{split}	
	\end{equation}  
	when $N_2$ is taken sufficiently large. 
	
	Let us define the matrix 
	\begin{equation}
		A_2 = \begin{bmatrix}
			\partial_{v^s_2}G^s_2(h^s_2,0) & 0 \\ 0 & \partial_{v^u_2}G^u_2(0,h^u_2)
		\end{bmatrix},
	\end{equation}
	which from Lemma~\ref{lem:GFn} is invertible since the partial derivatives are assumed to be nonzero. We then have $H_2(0,0) = \mathcal{O}(\eta^{N_2})$, and so $\|A_2^{-1}H_2(0,0)\| = \mathcal{O}(\eta^{N_2})$. Furthermore, 
	\begin{equation}
		\|1 - A_2^{-1}DH_2(a^s_2,a^u_2)\| = \mathcal{O}(|a^s_2| + |a^u_2| + \eta^{N_2}).
	\end{equation}
	Then, with $\rho = \frac{1}{4}$, $\kappa = \frac{1}{2}$, and $N_2 \gg 1$ sufficiently large, Lemma~\ref{lem:Roots} gives that there exists a unique solution $(a^s_2,a^u_2) = (\bar a^s_2,\bar a^u_2)$ satisfying $H_2(\bar a^s_2,\bar a^u_2) = 0$ and 
	\begin{equation}
		\|(\bar a^s_2,\bar a^u_2)\| = \mathcal{O}(\eta^{N_2}).
	\end{equation}
	Hence, we have satisfied the matching conditions~\eqref{2HomMatch} and therefore proven Theorem~\ref{thm:Main}(1) in the case $k = 2$.
	
	The cases $k \geq 3$ are similar except that they include more matching conditions coming from the push-forward functions $\Pi_i$. Let us illustrate with $k = 3$. Then,~\eqref{HomMatch} becomes 
	\begin{equation}\label{3HomMatch}
\begin{split}
		G^u_2(v^s_{2,0},v^u_{2,0}) &= 0 \\
		\Pi_2(v^s_{2,N_2},v^u_{2,N_2}) - (v^s_{3,0},v^u_{3,0}) &= 0 \\
		G^s_3(v^s_{3,N_{3}},v^u_{3,N_{3}}) &= 0 \\
	\end{split}
\end{equation}
and as in the previous case, we will define the function $H_3:\R^4 \to \R^4$ by
\begin{equation}
	H_3(a^s_2,a^u_2,a^s_3,a^u_3) := \begin{bmatrix}
		G^u_2(v^s_{2,0}(a^s_2,a^u_2),v^u_{2,0}(a^s_2,a^u_2)) \\ \Pi_2(v^s_{2,N_2}(a^s_2,a^u_2),v^u_{2,N_2}(a^s_2,a^u_2)) - (v^s_{3,0}(a^s_3,a^u_3),v^u_{3,0}(a^s_3,a^u_3)) \\ G^s_3(v^s_{3,N_3}(a^s_3,a^u_3),v^u_{3,N_3}(a^s_3,a^u_3))	
	\end{bmatrix}
\end{equation}
so that when $(a^s_2,a^u_2,a^s_3,a^u_3)$ are sufficiently small the roots of $H_4$ are exactly the values of $(a^s_2,a^u_2,a^s_3,a^u_3)$ that satisfy the matching conditions~\eqref{3HomMatch}.

The expansions~\eqref{2HomExpansion} are still valid here with the appropriate changes of indices to accommodate the presence of $(v^s_{3,N_3},v^u_{3,N_3})$, and therefore we focus on the second condition. From~\eqref{Expansions} we have
	\begin{equation}
		\begin{split}
		\Pi_2(v^s_{2,N_2},v^u_{2,N_2}) - (v^s_{3,0},v^u_{3,0}) &= \Pi_2(\mathcal{O}(\eta^{N_2}),h^u_2 + a^u_2) - (h^s_3 + a^s_3,\mathcal{O}(\eta^{N_3})) \\
		&= \Pi_2(0,h^u_2 + a^u_2 ) - (h^s_3 + a^s_3,0) + \mathcal{O}(\eta^{\min\{N_2,N_3\}}).
		\end{split}
	\end{equation} 
Let us define the matrix 
\begin{equation}\label{A3}
	A_3 := \begin{bmatrix}
			\partial_{v^s_2}G^u_2(h^s_2,0) & 0 & 0 & 0 \\ 0 & \zeta^s_2 & -1 & 0 \\ 0 & \zeta^u_2 & 0 & 0 \\ 0 & 0 & 0 & \partial_{v^u_3}G^s(0,h^u_3)
		\end{bmatrix} + \mathcal{O}(\eta^{\min\{N_2,N_3\}})
\end{equation}
where
\begin{equation}\label{xiEqn}
		[\zeta^s_2, \zeta^u_2]^T = \partial_{a^u_2}\Pi_2(0,h^u_2).
\end{equation}
The matrix $A_3$ is invertible because Lemma~\ref{lem:GFn} gives that $\partial_{v^s_2}G^u_2(h^s_2,0),\partial_{v^u_3}G^s(0,h^u_3) \neq 0$ and $\zeta^u_2 \neq 0$ since vary $a^u_2$ in a neighborhood of zero causes $\Pi_2(0,h^u_2+a^u_2)$ to locally parametrize a connected component of $W^u(\tilde\gamma_2(0))$, which by assumption transversely intersects $W^s_\mathrm{loc}(\tilde\gamma_3(0)) = \{v^u_3 = 0\}$ at $(h^s_3,0) \in \mathcal{I}\times\mathcal{I}$ in a neighborhood of $\tilde\gamma_3(0)$. Then, $H_3(0,0,0,0) = \mathcal{O}(\eta^{\min\{N_2,N_3\}})$, thus giving that $\|A_3^{-1}H_3(0,0,0,0)\| = \mathcal{O}(\eta^{\min\{N_2,N_3\}})$. Furthermore, 
\begin{equation}
	\|1 - A_3^{-1}DH_3(a^s_2,a^u_2,a^s_3,a^u_3)\| = \mathcal{O}(|a^s_2| + |a^u_2| + |a^s_3| + |a^u_3| + \eta^{\min\{N_2,N_3\}}),
\end{equation}  
and therefore with $\rho = \frac{1}{4}$, $\kappa = \frac{1}{2}$, and $N_2,N_3 \gg 1$, Lemma~\ref{lem:Roots} gives that there exists a unique solution $(a^s_2,a^u_2,a^s_3,a^u_3) = (\bar a^s_2,\bar  a^u_2,\bar  a^s_3,\bar a^u_3)$ satisfying $H_3(\bar a^s_2,\bar  a^u_2,\bar  a^s_3,\bar a^u_3) = 0$ and 
\begin{equation}
	\|(\bar a^s_2,\bar  a^u_2,\bar  a^s_3,\bar a^u_3)\| = \mathcal{O}(\eta^{\min\{N_2,N_3\}}). 
\end{equation}
This solution satisfies the matching conditions~\eqref{3HomMatch} and therefore proves Theorem~\ref{thm:Main}(1) for the case $k = 3$.

The general setting of $k \geq 4$ is almost identical to the case of $k = 3$. Indeed, we can define a function $H_k:\R^{2k-2}\to\R^{2k-2}$ so that when $(a^s_2,a^u_2,\dots,a^s_k,a^u_k)$ are sufficiently small the roots of $H_k$ are exactly the values of $(a^s_2,a^u_2,\dots,a^s_k,a^u_k)$ that satisfy the matching conditions~\eqref{HomMatch}. We then define the matrix 
\begin{equation}
		A_k:= \begin{bmatrix}
			\partial_{v^s_2}G^u_2(h^s_2,0) & 0 & 0 & \cdots & 0 & 0& 0 \\ 0 & \zeta^s_2 & -1 & \cdots & 0 & 0 & 0 \\ 0 & \zeta^u_2 & 0 & \cdots & 0 & 0 & 0 \\ 
			\vdots & \vdots & \vdots & \ddots & \vdots & \vdots & \vdots \\
			0 & 0 & 0 & \cdots & \zeta^s_{k-1} & -1 & 0\\
			0 & 0 & 0 & \cdots &\zeta^u_{k-1} & 0 & 0\\
			0 & 0 & 0 & \cdots & 0 & 0 & \partial_{v^u_k}G^s_k(0,h^u_k)
		\end{bmatrix}
\end{equation} 
where 
\begin{equation}\label{Zeta}
	[\zeta^s_i, \zeta^u_i]^T = \partial_{a^u_i}\Pi_i(0,h^u_i).	
\end{equation}
for all $2\leq i \leq k-1$. The same arguments that were used to show that $A_3$ defined in~\eqref{A3} can be used to show that $A_k$ is invertible for every $k \geq 4$. Then, the application of Lemma~\ref{lem:Roots} to obtain a root of $H_k$ is as in the cases $k = 2,3$ above, thus satisfying the matching conditions~\eqref{HomMatch} and proving Theorem~\ref{thm:Main}(1) for the remain cases with $k \geq 4$.  
\end{proof} 

\subsection{Constructing Periodic Trajectories}\label{subsec:Periodic}

In this subsection we will prove Theorem~\ref{thm:Main}(2), concerning the existence of periodic trajectories. Much of the proofs will be the same as in the previous section and therefore we seek to only highlight the differences. Recall that the initial vertex $V_1$ of the walk is the same as the terminal vertex $V_{k+1}$, thus giving that $\tilde\gamma_{k+1}(t) = \tilde\gamma_1(t)$. This means that the desired periodic orbit enters and leaves the neighbourhoods of the orbits $\gamma_i(t)$ exactly $k$-times, while the final entry into the neighborhood of $\tilde\gamma_{k+1}(t)$ marks a completion of one full period of the trajectory. In these neighbourhoods we may employ Lemma~\ref{lem:ShilSol} to obtain $k \geq 1$ solutions of the form
\begin{equation}
	\begin{split}
		\{(v^s_{1,n},v^u_{1,n})\}_{n = 0}^{N_1},& \quad v^s_{1,0} = h^s_1 +a^s_1, \quad v^u_{1,N_1}= h^u_1 +a^u_1 \\
		\{(v^s_{2,n},v^u_{2,n})\}_{n = 0}^{N_2},& \quad v^s_{2,0} = h^s_2 +a^s_2, \quad v^u_{2,N_2}= h^u_2 +a^u_2 \\
		\{(v^s_{3,n},v^u_{3,n})\}_{n = 0}^{N_3},& \quad v^s_{3,0} = h^s_3 +a^s_3, \quad v^u_{3,N_3}= h^u_3 +a^u_3 \\
		&\vdots \\
		\{(v^s_{k,n},v^u_{k,n})\}_{n = 0}^{N_{k}},& \quad v^s_{k,0} = h^s_k +a^s_{k}, \quad v^u_{k,N_{k}}= h^u_k +a^u_{k},
	\end{split}	
\end{equation}  
with the $a^s_i$ and $a^u_i$ taken to be small and the $N_i \geq 1$ sufficiently large to again satisfy~\eqref{Expansions}. The required matching conditions to prove Theorem~\ref{thm:Main}(2) then become
\begin{equation}\label{PerMatch}
	\begin{split}
		\Pi_1(v^s_{1,N_1},v^u_{1,N_1}) - (v^s_{2,0},v^u_{2,0}) &= 0 \\
		\Pi_2(v^s_{2,N_2},v^u_{2,N_2}) - (v^s_{3,0},v^u_{3,0}) &= 0 \\
		&\vdots \\
		\Pi_{k-1}(v^s_{k-1,N_{k-1}},v^u_{k-1,N_{k-1}}) - (v^s_{k,0},v^u_{k,0}) &= 0 \\
		\Pi_k(v^s_{k,N_k},v^u_{k,N_k}) - (v^s_{1,0},v^u_{1,0}) &= 0 .
	\end{split}
\end{equation} 
Notice the similarity between~\eqref{PerMatch} and~\eqref{HomMatch}, with the only difference coming from the first and last matching conditions. The first matching condition in~\eqref{PerMatch} simply describes iterates in the Poincar\'e section near $\tilde\gamma_1(0)$, while the final condition is the periodicity constraint that guarantees that after the trajectory transfers between $k$ local Poincar\'e sections it returns to where it began. 

The process of satisfying the matching conditions~\eqref{PerMatch} is nearly identical to that of satisfying~\eqref{HomMatch} and will therefore be omitted. Briefly, one defines a function whose roots are exactly the choices of $(a^s_1,a^u_1,\dots,a^s_k,a^u_k)$ near the origin that satisfy the matching equations~\eqref{PerMatch}. The invertible matrix required to apply Lemma~\ref{lem:Roots} is given by
\begin{equation}
	\begin{bmatrix}
			0 & \zeta^s_2 & -1 & \cdots & 0 & 0 & 0 \\ 0 & \zeta^u_2 & 0 & \cdots & 0 & 0 & 0 \\ 
			\vdots & \vdots & \vdots & \ddots & \vdots & \vdots & \vdots \\
			0 & 0 & 0 & \cdots & \zeta^s_{k-1} & -1 & 0\\
			0 & 0 & 0 & \cdots &\zeta^u_{k-1} & 0 & 0\\
			-1 & 0 & 0 & \cdots & 0 & 0 & \zeta^s_k \\
			0 & 0 & 0 & \cdots & 0 & 0 & \zeta^u_k
		\end{bmatrix} \in \R^{2k\times 2k},	
\end{equation}  
where the $\zeta^s_i$ and $\zeta^u_i$ are as they are defined in~\eqref{Zeta}. The argument that the above matrix is invertible follows from the previous arguments that the $\zeta^u_i \neq 0$ for each $1 \leq i \leq k$, and so the application of Lemma~\ref{lem:Roots} is entirely analogous to the proof of Theorem~\ref{thm:Main}(1) in the previous subsection.

\clearpage
\section{Discussion}\label{sec:Conclusion}

This manuscript has provided a comprehensive exploration of saddle mediated transport in the double pendulum. Our work describes the state space organization of a benchmark chaotic system with a focus on the implications of our findings in physical space. Our methods combined both numerical and analytical techniques, guided by previous investigations into saddle mediated transport, while also providing a general result that describes macroscopic transport in general Hamiltonian dynamical systems. We began by reviewing index-1 saddles of a Hamiltonian system, while also demonstrating the importance of these equilibria by recreating known results from the PCR3BP. We then showed that the state space organization of the physical double pendulum is similarly organized to the PCR3BP, primarily attributed to the presence of two index-1 saddles. The end result of this manuscript is twofold: we have extended the literature on saddle mediated transport to the double pendulum and provided a powerful theorem that can be used to produce itineraries throughout phase space for a large class of Hamiltonian systems. 

The results of this manuscript have extended the concept of the Interplanetary Transport Network to the double pendulum, similarly showing that there are physically determined pathways that allow one to move throughout phase space at no extra energy expenditure. The relative ease of building a double pendulum means that these transport mechanisms can be tested and followed physically as a table-top analogue of interplanetary space travel. Unfortunately, unlike the PCR3BP, a physical double pendulum model has frictional forces which destroy the tube structures as global invariant manifolds. However, for small dissipation in the system, these tubes will approximately remain on long timescales for which a control algorithm can be implemented to correct any discrepancies due to friction. Many researchers have sought to control the chaotic dynamics of the double pendulum by putting its anchored arm on a movable cart. These previous investigations are primarily concerned with stabilizing the unstable steady-states of the pendulum~\cite{driver2004design,zhong2001energy,ControlPend,hesse2018reinforcement}, such as the index-1 saddles, while the work herein opens up the possibility to stabilize a variety of acrobatic movements. With a double pendulum on a cart, methods of controlling chaos~\cite{OGY,ControlChaos1,ControlChaos2} could be implemented to introduce small movements of the cart to re-orient the pendulum arms to correct for any divergence of the tubes due to friction. Implementing these control methods on a physical system may require the aid of data-driven control methods~\cite{BramUPO} and/or discrepancy modelling~\cite{kaheman2019learning}, while also being guided by recent theoretic investigations into the effect of small amounts of dissipation or periodic forcing that break the Hamiltonian structure~\cite{DissRoss,ZhongGeometry,PerForced}. This is an avenue of ongoing work that builds on the theory in this manuscript.      

In a follow-up theoretical investigation, one may wish to understand the dynamics generated by the index-2 Up-Up saddle point in the double pendulum. This equilibrium has a two-dimensional stable and unstable manifold, and so does not have the nearby UPOs which give the tube structure near the index-1 saddles. Nonetheless, the two-dimensional manifolds of the Up-Up state could intersect along a one-dimensional curve, resulting in homoclinic orbits that similarly produce macroscopic transport mechanisms at higher energy values in the double pendulum. One may similarly seek to identify the transport pathways in the triple pendulum, formed by appending a third arm to the double pendulum. Similar to the double pendulum, experimental investigations have gone into stabilizing the arms of the triple pendulum on a cart at various unstable steady-states \cite{gluck2013swing,yaren2018stabilization,graichen2005fast,medrano1997robust}. Furthermore, the triple pendulum has three index-1 saddles: Down-Down-Up, Down-Up-Down, and Up-Down-Down, which are named analogously to the states in the double pendulum. The `tube' structures near these index-1 saddles are now three-dimensional, corresponding to a sphere crossed with a line, and so intersections should be expected to be two-dimensional~\cite{Chemical2}. Thus, the homoclinic and heteroclinic orbits in the double pendulum become two-dimensional invariant manifolds whose dynamics asymptotically approach neighborhoods of the index-1 saddles of the triple pendulum. Both situations require a delicate numerical treatment that will be left to a follow-up investigation.

\section*{Acknowledgments} 
The authors acknowledge funding support from the Army Research Office (ARO W911NF-19-1-0045) and National Science Foundation AI Institute in Dynamic Systems (grant number 2112085). 
The authors also would like to thank Shane Ross for valuable discussions.

\appendix

\section{Governing Equations of the PCR3BP}

In Section~\ref{sec:Background} we illustrated the tube dynamics of index-1 saddles using numerical simulations from the planar circular restricted three-body problem (PCR3BP). We assume there are two massive bodies with masses $\mu_1$ and $\mu_2$ which form their own two body problem, while a (relatively) weightless body such as a satellite or shuttle moving in the gravitational field generated by these objects. Here we present the relevant dynamical system for the third weightless body restricted to the orbital plane of the two massive bodies which constitute the Hamiltonian PCR3BP system. We define the states of the system as $x, y, v_x$ and $v_y$, representing the position of the third body in the plane and its velocity in each independent direction. 

The potential energy of the system is given by 
\begin{equation}
    \label{TBP_V}
    V_{3BP}=-\frac{1}{2}\left(\mu_{1} r_{1}^{2}+\mu_{2} r_{2}^{2}\right)-\frac{\mu_{1}}{r_{1}}-\frac{\mu_{2}}{r_{2}},
\end{equation}
where 
\begin{equation}\label{TBP_Parameters}
    \begin{split}
        r_{1}^{2}&=\left(x+\mu_{2}\right)^{2}+y^{2}, \\
        r_{2}^{2}&=\left(x-\mu_{1}\right)^{2}+y^{2},\\
        \mu&=\frac{m_{2}}{m_{1}+m_{2}},\\
        \mu_1&=1-\mu,\\
        \mu_2&=\mu.
    \end{split}
\end{equation}
The kinetic energy is given by
\begin{equation}
    T_{3BP}=\frac{1}{2}(v_x^2+v_y^2),
\end{equation}
so that the total energy of the PCR3BP is 
\begin{equation}\label{Eng_TBP}
    \mathcal{H}_{3BP}=T_{3BP}+V_{3BP}.
\end{equation}
Using the above Hamiltonian, we can arrive at the corresponding first-order ODE describing the equations of motion of the PCR3BP. The result is given by
\begin{equation}\label{PCR3BP_EOM}
    \begin{split}
        \dot{x}&=v_x,\\
        \dot{y}&=v_y,\\
        \dot{v}_x&=2v_y+x-\frac{(x-(1 - \mu))\mu}{r_1^3}-\frac{(1 - \mu)(x+\mu)}{r_2^3},\\
        \dot{v}_y&=-2v_x+y-\frac{\mu y}{r_1^3}-\frac{(1-\mu)y}{r_2^3},
    \end{split}
\end{equation}
Here $\mu$ is the mass ratio of the two massive bodies. For the examples provided in Section~\ref{sec:Background} we have followed~\cite{Koon} used taken $\mu=9.537\times 10^{-4}$, which is the mass ratio of Jupiter to the Sun in our solar system. 

The $L_1$ and $L_2$ index-1 saddle Lagrangian points lie on the line $y=0$. Such points are equilibria of~\eqref{PCR3BP_EOM} with $y = 0$ and, per \cite[Section~2.5]{RossBook}, are given by $x = (1- \mu) \pm \gamma$, where $\gamma$ are solutions of the quintic polynomials
\begin{equation}\label{TBP_quintic}
    \gamma^{5} \mp(3-\mu) \gamma^{4}+(3-2 \mu) \gamma^{3}-\mu \gamma^{2} \pm 2 \mu \gamma-\mu=0.
\end{equation}
Precisely, the $L_1$ Lagrange point comes from a solution $\gamma$ with the upper signs in the above polynomial, while the $L_2$ Lagrange point comes from using lower signs. With the value $\mu=9.537\times 10^{-4}$, the corresponding $\gamma$ values in~\eqref{TBP_quintic} are $\gamma_1=0.06667655835869524$ and $\gamma_2=0.06978018274821962$, thus leading to the $x$ positions of the $L_1$ and $L_2$ Lagrangian points
\begin{equation}\label{TBP_L1L2_Position}
    \begin{split}
        x_{L_1} &=(1-\mu)-\gamma_1=0.9323697416413048, \\
        x_{L_2} &=(1-\mu)+\gamma_2=1.0688264827482197.
    \end{split}
\end{equation}
The coordinates $(x_{L_i},0,0,0)$, $i = 1,2$, thus represent the location of the index-1 saddles of interest in the PCR3BP~\eqref{PCR3BP_EOM}.

\section{The Point-Mass Double Pendulum}

Here we will provide the relevant equations of motion for the point-mass double pendulum and describe the analogous linear analysis to that of~\S\ref{sec:linearanalysis}. The goal is to show that the Down-Up and Up-Down equilibria remain index-1 saddles for the simplified point-mass double pendulum. Furthermore, we comment that although it is not reported here, there is a nearly identical tube structure to what we have shown for the physical double pendulum, meaning that our election to study the more physically realistic physical double pendulum presents little difference to the more typically studied point-mass double pendulum. 

To begin, let $m_1$ denote the point mass of the first (base) arm and $m_2$ be the point mass of the second arm. We take the length of the first and second arms to be $\ell_1$ and $\ell_2$, respectively. The angle $\theta_1$ represents the angular displacement of the first arm from the downward vertical position, while $\theta_2$ is the angular displacement of the second arm. We again refer the reader to the left panel of Figure~\ref{Fig:DP_Illustrate} for visual reference. 

The kinetic energy of the double pendulum is given by
\begin{equation}\label{DPKinetic}
	T = \frac{1}{2}(m_1+m_2)\ell_1^2\omega_1^2+\frac{1}{2}m_2\ell_2^2\omega_2^2+m_2\ell_1\ell_2\omega_1\omega_2\cos(\theta_2 - \theta_1)
\end{equation}
and the potential energy is given by
\begin{equation}\label{DPPotential}
	V = -(m_1+m_2)\ell_1g\cos(\theta_1) - m_2\ell_2g\cos(\theta_2).	
\end{equation}
Hence, the total energy of the system is given by the Hamiltonian function 
\begin{equation}\label{DPHamiltonian}
	\mathcal{H}(\theta_1,\theta_2,\omega_1,\omega_2) = T + V
\end{equation}
composed of the sum of the kinetic~\eqref{DPKinetic} and potential~\eqref{DPPotential} energies. Following as in Section~\ref{sec:EOM} we arrive at the corresponding first-order conservative ODE describing the equations of motion for the point-mass double pendulum:
\begin{equation}\label{DPeom}
	\begin{split}
        \dot{\theta}_1&=\omega_1\\
        \dot{\theta}_2&=\omega_2\\
        \dot{\omega}_1&=\frac{A_1+A_2+A_3}{\ell_1^2\ell_2(m_1+m_2-m_2\cos^2(\theta_2 - \theta_1))}\\
        \dot{\omega}_2&=\frac{B_1+B_2+B_3}{\ell_1\ell_2^2m_2(m_1+m_2 - m_2\cos^2(\theta_2 - \theta_1))}
    \end{split}.
\end{equation}
Here we have 
\begin{equation}\label{DPeom2}
	\begin{split}
		A_1 &=m_2\ell_1\ell_2^2\omega_2^2\sin(\theta_2 - \theta_1),\\
        A_2 &=\frac{1}{2}m_2\ell_1^2\ell_2\omega_1^2\sin(2\theta_2 - 2\theta_1),\\
        A_3 &= \ell_1\ell_2g[m_2\sin(\theta_2)\cos(\theta_2 - \theta_1)-(m_1+m_2)\sin(\theta_1)],\\
        B_1 &=-\frac{1}{2}\ell_1\ell_2^2m_2^2\omega_2^2\sin(2\theta_2 - 2\theta_1),\\
        B_2 &=-\ell_1^2\ell_2m_2(m_1+m_2)\omega_1^2\sin(\theta_2 - \theta_1),\\
        B_3 &=-\ell_1\ell_2m_2g(m_1+m_2)[\sin(\theta_2)-\sin(\theta_1)\cos(\theta_2 - \theta_1)],
	\end{split}
\end{equation}
where we have neglected any friction or control terms that could be added to the system.  

We now follow the linear analysis of~\S\ref{sec:linearanalysis} to show that again the Down-Up and Up-Down steady-states are index-1 saddles of~\eqref{DPeom}. Again, equilibria come in the form $(\theta_1,\theta_2,\omega_1,\omega_2) = (\pi k_1,\pi k_2,0,0)$ for every pair of integers $(k_1,k_2) \in \mathbb{Z}^2$ and all of these equilibria are symmetric with respect to the action of the reverser~\eqref{Reverser}. As with the physical double pendulum, by periodicity of the $\theta_1,2$ components we have four equilibria that merit investigation: Down-Down, Down-Up, Up-Down, and Up-Up, as described in~\eqref{DPsteady}. 

Linearizing~\eqref{DPeom} about an equilibrium $(\theta_1,\theta_2,\omega_1,\omega_2) = (\pi k_1,\pi k_2,0,0)$ with $k_1,k_2\in\mathbb{Z}$ results in the matrix
\begin{equation}\label{DPLinMat2}
	\begin{bmatrix}
		0 & 0 & 1 & 0 \\
		0 & 0 & 0 & 1 \\
		\frac{-g(m_1 + m_2)\cos(\pi k_1)}{\ell_1(m_1+m_2-m_2\cos^2(\pi k_2 - \pi k_1))} & \frac{gm_2 \cos(\pi k_2)\cos(\pi k_2 - \pi k_1)}{\ell_1(m_1+m_2-m_2\cos^2(\pi k_2 - \pi k_1))} & 0 & 0 \\
		\frac{g(m_1 + m_2)\cos(\pi k_1)\cos(\pi k_2 - \pi k_1)}{\ell_2(m_1+m_2 - m_2\cos^2(\pi k_2 - \pi k_1))} & \frac{-g(m_1 + m_2) \cos(\pi k_2)}{\ell_2(m_1+m_2 - m_2\cos^2(\pi k_2 - \pi k_1))} & 0 & 0
	\end{bmatrix}.
\end{equation} 
The block structure of~\eqref{DPLinMat2} gives the square of its eigenvalues are equal to those of the lower left $2\times 2$ matrix
\begin{equation}\label{MiniMat2}
	\begin{bmatrix}
		\frac{-g(m_1 + m_2)\cos(\pi k_1)}{\ell_1(m_1+m_2-m_2\cos^2(\pi k_2 - \pi k_1))} & \frac{gm_2 \cos(\pi k_2)\cos(\pi k_2 - \pi k_1)}{\ell_1(m_1+m_2-m_2\cos^2(\pi k_2 - \pi k_1))} \\
		\frac{g(m_1 + m_2)\cos(\pi k_1)\cos(\pi k_2 - \pi k_1)}{\ell_2(m_1+m_2 - m_2\cos^2(\pi k_2 - \pi k_1))} & \frac{-g(m_1 + m_2) \cos(\pi k_2)}{\ell_2(m_1+m_2 - m_2\cos^2(\pi k_2 - \pi k_1))} 
	\end{bmatrix}.
\end{equation}
We therefore arrive at the following results.\\

\noindent{\bf Down-Down.} The Down-Down state is obtained by setting $k_1 = k_2 = 0$. In this case~\eqref{MiniMat2} becomes
\begin{equation}
	\begin{bmatrix}
		\frac{-g(m_1 + m_2)}{\ell_1m_1} & \frac{gm_2}{\ell_1m_1} \\
		\frac{g(m_1 + m_2)}{\ell_2m_1} & \frac{-g(m_1 + m_2)}{\ell_2m_1}
	\end{bmatrix}.
\end{equation} 
Since the trace of the above matrix is negative and the determinant is positive for all $m_1,m_2,\ell_1,\ell_2 > 0$, it follows that both eigenvalues are distinct and strictly negative. Therefore, the linearization~\eqref{DPLinMat2} about the Down-Down state has two pairs of purely complex eigenvalues, making it a linear center. \\

\noindent{\bf Down-Up.} The Down-Up equilibrium is obtained by setting $k_1 = 0$ and $k_2 = 1$. The matrix~\eqref{MiniMat2} becomes 
\begin{equation}
	\begin{bmatrix}
		\frac{-g(m_1 + m_2)}{\ell_1m_1} & \frac{gm_2}{\ell_1m_1} \\
		\frac{-g(m_1 + m_2)}{\ell_2m_1} & \frac{g(m_1 + m_2)}{\ell_2m_1}
	\end{bmatrix}
\end{equation} 
which has a negative determinant for all relevant parameter values. Therefore, the above matrix has one positive and one negative real eigenvalue, thus making the Down-Up equilibrium an index-1 saddle since~\eqref{DPLinMat2} has one positive eigenvalue, one negative eigenvalue, a pair purely imaginary eigenvalues when $k_1 = 0$ and $k_2 = 1$.  \\

\noindent{\bf Up-Down.} We obtain the Up-Down equilibrium by setting $k_1 = 1$ and $k_2 = 0$. The matrix~\eqref{MiniMat2} is then given by
\begin{equation}
	\begin{bmatrix}
		\frac{g(m_1 + m_2)}{\ell_1m_1} & \frac{-gm_2}{\ell_1m_1} \\
		\frac{g(m_1 + m_2)}{\ell_2m_1} & \frac{-g(m_1 + m_2)}{\ell_2m_1}
	\end{bmatrix}
\end{equation} 
and following as in the Down-Up equilibrium, we find that the Up-Down equilibrium is also an index-1 saddle. \\

\noindent{\bf Up-Up.} The Up-Up state has $k_1 = k_2 = 1$, resulting in~\eqref{MiniMat2} taking the form
\begin{equation}
	\begin{bmatrix}
		\frac{g(m_1 + m_2)}{\ell_1m_1} & \frac{-gm_2}{\ell_1m_1} \\
		\frac{-g(m_1 + m_2)}{\ell_2m_1} & \frac{g(m_1 + m_2)}{\ell_2m_1}
	\end{bmatrix}.
\end{equation} 
The above matrix is the result of negating all of the entries of the Down-Down analysis by $-1$. It follows that the linearization~\eqref{DPLinMat2} about the Up-Up state has two positive eigenvalues and two negative eigenvalues. From the nomenclature above, the Up-Up state is an {\em index-2 saddle}. \\

\noindent From above we can see that the Down-Up and Up-Down equilibria are index-1 saddles, similar to the physical pendulum. \\


 \begin{spacing}{.77}
\addcontentsline{toc}{section}{References}
\bibliographystyle{abbrv}
\bibliography{DPtransport}
 \end{spacing}

\end{document}